\documentclass[11pt,reqno]{amsart}
\usepackage{amscd,amssymb,amsmath,amsthm}
\usepackage[english]{babel}
\usepackage[T1]{fontenc}
\usepackage[arrow,matrix]{xy}
\usepackage{graphicx,subfigure,tikz,tikz-3dplot}
\usetikzlibrary{positioning,matrix,arrows,calc}
\usepackage{cite}
\usepackage{dsfont}
\usepackage{enumitem}
\usepackage{bbm}
\usepackage{hyperref}
\hypersetup{
	colorlinks=true,
	linkcolor=blue}

\oddsidemargin=0.3in \evensidemargin=0.3in \theoremstyle{plain}
\newtheorem{theorem}{Theorem}[section]
\newtheorem{lemma}[theorem]{Lemma}
\newtheorem{proposition}[theorem]{Proposition}
\newtheorem{corollary}[theorem]{Corollary}

\theoremstyle{definition}
\newtheorem{remark}[theorem]{Remark}

\newtheorem{example}[theorem]{Example}

\numberwithin{equation}{section} 

\begin{document}
	
	\title[]{Iterated function systems of affine expanding and contracting maps on the unit interval}
	\author[Ale Jan Homburg, Charlene Kalle]{Ale Jan Homburg and Charlene Kalle}
	
	\address{A.J. Homburg\\ KdV Institute for Mathematics, University of Amsterdam, Science park 107, 1098 XG Amsterdam, Netherlands\newline Department of Mathematics, VU University Amsterdam, De Boelelaan 1081, 1081 HV Amsterdam, Netherlands}
	\email{a.j.homburg@uva.nl}
	
	\address{C.C.C.J. Kalle\\ Mathematical Institute, University of Leiden, PO Box 9512, 2300 RA Leiden, The Netherlands}
	\email{kallecccj@math.leidenuniv.nl}

\begin{abstract}
We analyze the two-point motions of iterated function systems on the unit interval generated by expanding and contracting affine maps, where the expansion and contraction rates are determined by a pair $(M,N)$ of integers. 

This dynamics depends on the Lyapunov exponent.
For a negative Lyapunov exponent we establish synchronization, meaning convergence of orbits with different initial points. For a vanishing Lyapunov exponent we establish intermittency, where orbits are close for a set of iterates of full density, but are intermittently apart. For a positive Lyapunov exponent we show the existence of an absolutely continuous stationary measure for the two-point dynamics and discuss its consequences.

For nonnegative Lyapunov exponent and pairs $(M,N)$ that are multiplicatively dependent integers, we provide explicit expressions for absolutely continuous stationary measures of the two-point motions. These stationary measures are infinite $\sigma$-finite measures in the case of zero Lyapunov exponent.
For varying Lyapunov exponent we find here a phase transition for the system of two-point motions, in which the support of the stationary measure explodes with intermittent dynamics and an infinite stationary measure at the transition point. 
\end{abstract}
	
	
	\maketitle

	\section{Introduction}

In this article we introduce a natural and simple toy model of iterated function systems on the interval with affine expanding and contracting maps and explore its dynamics. We focus in particular on the dynamics of two orbits simultaneously, the so-called two-point motions. Our set-up is as follows. Given a pair $(M,N)$ of integers $M,N \ge 2$, let
\[ f_0: [0,1)\to [0,1); \, x \mapsto N x \pmod 1\]
be the $N$-adic map and let
\[ f_i: [0,1)\to [0,1) ; \, x \mapsto  (x  + i - 1)/M, \qquad 1 \le i \le M,\]
be $M$ contracting maps. Figure~\ref{f:graph32} depicts the graphs for a few values of $(M,N)$.

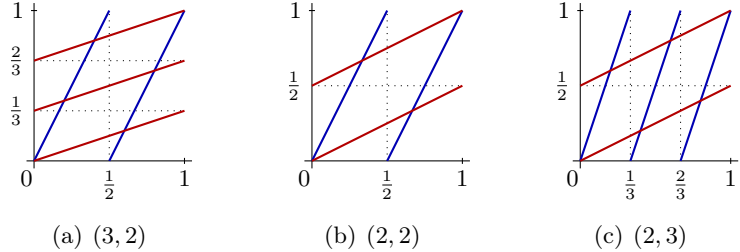
\begin{figure}[!ht]
\begin{center}
\subfigure[$(3,2)$]{
\begin{tikzpicture}[scale=2]
\draw(-.05,0)node[below]{\footnotesize $0$}--(.5,0)node[below]{\footnotesize $\frac12$}--(1,0)node[below]{\footnotesize $1$}--(1.05,0);
\draw(0,-.05)--(0,1/3)node[left]{\footnotesize $\frac13$}--(0,2/3)node[left]{\footnotesize $\frac23$}--(0,1)node[left]{\footnotesize 1}--(0,1.05);
\draw(1,-.02)--(1,.02)(-.02,1)--(.02,1);
\draw[dotted](.5,0)--(.5,1)(0,1/3)--(1,1/3)(0,2/3)--(1,2/3);
\draw[thick, blue!70!black] (0,0)--(.5,1)(.5,0)--(1,1);
\draw[thick, red!70!black] (0,0)--(1,1/3)(0,1/3)--(1,2/3)(0,2/3)--(1,1);
\end{tikzpicture}
}
\hspace{.5cm}
\subfigure[$(2,2)$]{
\begin{tikzpicture}[scale=2]
\draw(-.05,0)node[below]{\footnotesize $0$}--(.5,0)node[below]{\footnotesize $\frac12$}--(1,0)node[below]{\footnotesize $1$}--(1.05,0);
\draw(0,-.05)--(0,1/2)node[left]{\footnotesize $\frac12$}--(0,1)node[left]{\footnotesize 1}--(0,1.05);
\draw(1,-.02)--(1,.02)(-.02,1)--(.02,1);
\draw[dotted](.5,0)--(.5,1)(0,1/2)--(1,1/2);
\draw[thick, blue!70!black] (0,0)--(.5,1)(.5,0)--(1,1);
\draw[thick, red!70!black] (0,0)--(1,1/2)(0,1/2)--(1,1);
\end{tikzpicture}}
\hspace{.5cm}
\subfigure[$(2,3)$]{
\begin{tikzpicture}[scale=2]
\draw(-.05,0)node[below]{\footnotesize $0$}--(1/3,0)node[below]{\footnotesize $\frac13$}--(2/3,0)node[below]{\footnotesize $\frac23$}--(1,0)node[below]{\footnotesize $1$}--(1.05,0);
\draw(0,-.05)--(0,1/2)node[left]{\footnotesize $\frac12$}--(0,1)node[left]{\footnotesize 1}--(0,1.05);
\draw(1,-.02)--(1,.02)(-.02,1)--(.02,1);
\draw[dotted](1/3,0)--(1/3,1)(2/3,0)--(2/3,1)(0,1/2)--(1,1/2);
\draw[thick, blue!70!black] (0,0)--(1/3,1)(1/3,0)--(2/3,1)(2/3,0)--(1,1);
\draw[thick, red!70!black] (0,0)--(1,1/2)(0,1/2)--(1,1);
\end{tikzpicture}}
\caption{Graphs of $f_0,\ldots,f_M$ for $(M,N) = (3,2), (2,2), (2,3)$.}
\label{f:graph32}
\end{center}	
\end{figure}

For a sequence $\omega = (\omega_0, \omega_1, \ldots ) \in \{0,1, \ldots, M \}^\mathbb N$, write
\begin{equation}\label{q:iterates}
f_\omega^n = f_{\omega_{n-1}} \circ \cdots \circ f_{\omega_1} \circ f_{\omega_0},
\end{equation}
for $n$ compositions of maps from $\{ f_0, f_1, \ldots, f_M\}$ with $f_\omega^0$ equal to the identity mapping. We consider orbits $(f_\omega^n(x))_{n \ge 1}$ for points $x \in [0,1)$, where the $\omega_i \in \{0,1, \ldots, M \}$ are picked independently and identically distributed with probabilities $p_i$. Throughout the article we make the following assumption on the probability vector $\mathbf p = (p_0,\ldots, p_M)$: Choose the map $f_0$ with probability $0 < p_0 < 1$ and all maps $f_i$, $1 \le i \le M$, with equal probability $ p_i = \frac{1-p_0}{M}$. So the randomness depends on a single parameter $p_0 \in (0,1)$ and the probability vector $\mathbf p$ is of the special form
\begin{equation}\label{q:vectorp}
\mathbf p = \left(p_0 , \frac{1-p_0}{M} , \ldots, \frac{1-p_0}{M} \right).
\end{equation}
Let $\nu$ denote the $\mathbf p$-Bernoulli measure on $\{0,1, \ldots, M \}^\mathbb N$. Let $\lambda$ denote the Lebesgue measure.

\vskip .2cm
We are interested in results on the two-point motions $(f^n_\omega(x),f^n_\omega(y))_{n \ge 0}$ for $x,y \in [0,1)$ and $\omega \in \{0,1, \ldots, M \}^\mathbb N $. Statistical properties of such two-point motions are obtained by studying the iterated function system on $[0, 1)^2$ generated by the maps
\begin{equation}\label{q:2ptifs}
f^{(2)}_i (x,y) = ( f_i (x), f_i(y)), \qquad  0\le i \le M.
\end{equation}
We note that two-point motions in contexts of stochastic differential equations are considered in work by Baxendale, see in particular \cite{MR1144097,MR968817,MR1678447}. For compositions of independent random diffeomorphisms it is investigated in \cite{MR968818}.

\vskip .2cm
Here we consider two types of results. Firstly, we investigate the asymptotics of the distances $|f^n_\omega(x)-f^n_\omega(y)|$ when $n \to \infty$. Below we show that, with the probability vector $\mathbf p$ from \eqref{q:vectorp}, the Lebesgue measure is a stationary measure for the iterated function system $\{ f_i \, ; \, 0 \le i \le M\}$ on $[0,1)$. (We note that Lebesgue measure is not always the unique stationary measure, examples of non-uniqueness can be deduced from \cite{MR2150697}.) In this sense we treat conservative systems and one expects points from typical orbits to lie uniformly distributed in the unit interval. However, we will see that different values of $M$ and $N$, or different values of $p_0$, lead to significant differences in the behavior of two orbits with different initial conditions under the same composition of maps. We distinguish three different types of dynamical behavior, the occurrence of which hinges on the sign of the Lyapunov exponent
\begin{align}\label{e:Lp0}
 L_{p_0} &= \lim_{n\to\infty} \frac{1}{n} \sum_{i=0}^{n-1}  \ln f_{\omega_i}' =   p_0 \ln (N) - (1-p_0) \ln (M).
\end{align}
This limit exists almost surely and equals the given constant by the strong law of large numbers. The following theorem assembles our main results on the asymptotics of $|f^n_\omega(x)-f^n_\omega(y)|$. 
 
\begin{theorem}\label{t:main}
Let $M,N \ge 2$ be integers and $0 < p_0 < 1$ be given. For the iterated function system $\{ f_0, f_1, \ldots, f_M\}$ and probability vector $\mathbf p$ as in \eqref{q:vectorp}, we have the following.
\begin{itemize}
	
\item[(i)] Suppose $L_{p_0} < 0$. Then 
\[ \lim_{n \to \infty} |f^n_\omega (x) - f^n_\omega(y)|=0\]
for all $x,y \in [0,1)$ and $\nu$-almost all $\omega$.

\item[(ii)] Suppose $L_{p_0}=0$. Then for every $\varepsilon >0$ we have
\[ \lim_{n \to \infty} \frac1n |\{ 0 \le i < n \, ; \, |f_\omega^i(x)-f_\omega^i(y)| < \varepsilon \} |=1\]
for all $x,y \in [0,1)$ and $\nu$-almost all $\omega$, while for any small $\beta >0$, any $x,y \in [0,1)$ and $\nu$-almost all $\omega$ either $|f^n_\omega (x) - f^n_\omega(y)|=0$ for some $n$ or $|f^n_\omega (x) - f^n_\omega(y)|> \beta$ for infinitely many values of $n$.

\item[(iii)] Suppose $L_{p_0} >0$. 
Then \[P(\varepsilon) =  \lim_{n\to \infty} \frac1n \left|\{ 0 \le i < n \, ; \, |f_{\omega}^i(x)-f_\omega^i(y)| < \varepsilon \}\right|\] exists for $\nu\times \lambda$-almost all $(\omega,x,y)$,
and 
\[
\lim_{\varepsilon \to 0} P(\varepsilon)  = 0.
\]
\end{itemize}
\end{theorem}

This theorem combines statements of Theorem~\ref{t:M>Nconvergence}, Theorem~\ref{t:Lp=0frequency}, Theorem~\ref{t:Lp=0frequency2} and Theorem~\ref{t:divergence} below.

\vskip .2cm
To put the results in a broader context, we comment on the phenomena of synchronization, intermittency and instability, observed in the three different cases.

\begin{description}[leftmargin=!]
	
\item[\underline{$L_{p_0} < 0$}]\; In this case the contraction wins from the expansion and synchronization occurs, which means that orbits from different initial points in $[0,1)$ converge to each other almost surely. This is comparable to synchronization by noise \cite{Pik84,zbMATH03864265}.  Related synchronization results have been obtained in diverse settings, see e.g. \cite{MR1144097,MR3663624,MR3862799,MR3820004}. Kleptsyn and Volk \cite{MR3236497} treat such a phenomenon in the context of smooth monotone interval maps forced by transitive subshifts of finite type. Closely related is \cite{MR3600645} that provides cases of synchronization for iterated function systems generated by interval diffeomorphisms. Synchronization by noise in random logistic maps is considered in \cite{AD00,MR1908881,MR3826118}. \\
	
\item[\underline{$L_{p_0} = 0$}]\; In this neutral case a phenomenon reminiscent of intermittency arises. 
Intermittency, first studied in \cite{PM80}, refers to the phenomenon where a dynamical system shows sudden transitions from a long period of exhibiting one type of dynamical behavior to a  period of another type of dynamics. 
Recently this was analyzed in the context of random dynamics for the random Gauss-R\'enyi map \cite{KKV17,BRS20,Tay21,KMTV,afgv}, random LSV maps \cite{BBD14,BB16}, random logistic maps \cite{MR2033186,MR3826118} and more general families in \cite{homburg2021critical,KZ}.

In the setting of  Theorem~\ref{t:main}, orbits of different initial points are intermittently very close together or some distance apart. The set of iterates for which orbits are close has full density, but the complement is still an infinite set. 
A similar mechanism arises in iterated function systems of interval diffeomorphisms \cite{MR3600645}, or more generally for skew product systems with interval diffeomorphisms as fiber maps \cite{MR4108906}. In both papers one gets a singular distribution of orbit points instead of the uniform distribution that we find.  	\\

\item[\underline{$L_{p_0} > 0$}]\; Here the expansion wins from the contraction 
and orbits tend to diverge from each other. Random interval maps with a condition on average expansion have been studied
extensively, see e.g. \cite{MR807105,MR1707698,MR722776,KALLE_2020,MR3350045,MR3816673}. Following the definition from \cite{MR722776} the expanding on average condition would correspond to $\frac{p_0}{N} + (1-p_0)M < 1$, which does not align with the condition that $L_{p_0}>0$. Hence we do not rely on these expansion on average results here.
The two-point maps are connected to Jablonski maps \cite{MR4287984}, and in the positive Lyapunov exponent case to  
research on invariant measures for random Jablonski maps \cite{MR1345800,MR4342141,MR2110096,Hsi08}.

\end{description}

\vskip .2cm
Our second set of main results concerns invariant measures of the iterated function system from \eqref{q:2ptifs}, again with the probability vector $\mathbf p$ from \eqref{q:vectorp}, under an additional assumption on the expansion and contraction factors $M$ and $N$. Two integers $M,N \ge 1$ are called {\em multiplicatively dependent} if they are powers of the same natural number, i.e., if 
$M = \kappa^\ell$ and $N = \kappa^k$ for some integers $\kappa>1$ and $k, \ell \ge 1$. Here we always take $k,\ell$ to be relatively prime. This condition is equivalent to $\ln(N)/\ln(M) = k/\ell \in \mathbb Q$. If $M,N$ are not multiplicatively dependent, they are called multiplicatively independent. Some of the difficulties in the analysis for Theorem~\ref{t:main} are caused by the points of discontinuity of $f_0$ and can be circumvented in case $M,N$ are multiplicatively dependent. This leads to the following theorem.

\vskip .2cm
Write $\Delta = \{ (x,x) \; ; \; x \in [0,1)\}$ for the diagonal in $[0,1)^2$ and write $\Delta_\varepsilon = \{ (x,y) \in [0,1)^2 \; ; \; |y-x|<\varepsilon \}$ for the $\varepsilon$-neighborhood of $\Delta$. We let $a (\varepsilon) \sim b(\varepsilon)$ stand for $a(\varepsilon)/b(\varepsilon)$ bounded and bounded away from zero as $\varepsilon \to 0$.
 
\begin{theorem}\label{t:main2}
	Let $M,N \ge 2$ be integers and $0 < p_0 < 1$ be given. 
	Assume that $M$ and $N$ are multiplicatively dependent with $N = \kappa^k$ and $M = \kappa^\ell$.
	For the iterated function system $\{ f_0^{(2)}, f_1^{(2)}, \ldots, f_M^{(2)}\}$ and probability vector $\mathbf p$ as in \eqref{q:vectorp}, we have the following.
	\begin{itemize}
		
		\item[(i)] Suppose $L_{p_0} < 0$. 
			Then the iterated function system of two-point maps admits Lebesgue measure on $\Delta$ as stationary measure. 
		
		\item[(ii)] Suppose $L_{p_0}=0$. 
		Then the iterated function system of two-point maps admits Lebesgue measure on $\Delta$ as stationary measure. 
		Furthermore, it admits an infinite $\sigma$-finite
		absolutely continuous stationary measure of full topological support.
		
		\item[(iii)] Suppose $L_{p_0} >0$. 
		Then the iterated function system of two-point maps admits Lebesgue measure on $\Delta$ as stationary measure. Furthermore, it admits an  
		absolutely continuous stationary probability measure $\mu^{(2)}$ of full topological support and with
		\[\mu^{(2)} (\Delta_\varepsilon)  \sim \varepsilon^{-\ln( \nu_1) / \ln (\kappa)}, \]
		where $\nu_1$ is the unique real solution in $(0,1)$ to $p_0 z^{k+\ell} - z^\ell + 1 - p_0 = 0$.	
		The density of $\mu^{(2)}$ is bounded precisely if $\nu_1 \kappa < 1$.

	\end{itemize}
\end{theorem}

The measure  $\mu^{(2)} (\Delta_\varepsilon)$ from Theorem~\ref{t:main2}(iii) quantifies the proportion of iterates that typical orbits $f^n_\omega (x)$ and $f^n_\omega (y)$ are close. This theorem combines statements of Corollary~\ref{c:diagonal}, Theorem~\ref{t:M>Nconvergence}, Theorem~\ref{t:2point}, Theorem~\ref{t:stationary-multdep} and Remark~\ref{r:boundeddensity} below.
Further results of a similar flavor, in particular with explicit expressions for stationary measures, or for stationary measures in case of multiplicatively independent pairs $(M,N)$, are found in Section~\ref{s:2point}.

\vskip .2cm
We stress that the original iterated function system $\{ f_i \, ; \, 0 \le i \le M\}$ behaves independently of $p_0$ in the sense that Lebesgue measure on the interval $[0,1)$ is stationary for all values of $p_0$. The above theorem however makes clear that, depending on the parameters, the corresponding two-point motions show a range of different behaviors. In particular the theorem describes a bifurcation or phase transition
in the iterated function system of two-point motions as 
the Lyapunov exponent crosses zero for varying $p_0$,  compare also \cite{MR1678447,dacosta2021stochastic}.
 This phase transition involves a discontinuous change
 of the support of the stationary measure of the two-point motion
 (an explosion of its support) and an infinite stationary measure at the bifurcation point.   
 Bifurcations that involve a Lyapunov exponent crossing zero are also considered in studies of noise-induced order such as \cite{MR711470,MR4127778} and in settings with skew product systems such as \cite{TZ22}.

\vskip .2cm
The article is outlined as follows. In the next section we introduce preliminaries on random dynamics, we prove that Lebesgue measure is stationary for the iterated function systems $\{ f_0, f_1, \ldots, f_M\}$ and we introduce several extensions of these  systems that are useful in later parts of the text. In particular we explain a connection to a class of generalized Baker maps in three dimensions, a particular example of which has recently been studied in \cite{MR4287984} in the context of heterogeneous chaos. In Section~\ref{s:2point} we study the iterated function system $\{ f_0^{(2)}, f_1^{(2)}, \ldots, f_M^{(2)}\}$ and derive our main results. The section is divided into three parts depending on the sign of $L_{p_0}$. All parts come with their own techniques. The case of a vanishing Lyapunov exponent uses theory of random walks involving stopping times with time dependent stopping criteria. This material is developed in Appendix~\ref{s:A}. We end the article with a short description of possible future extensions of this research.
	
	\hspace{0.5cm}\\
	\noindent {\bf Acknowledgements.}
	The idea for this paper started with a project for a bachelor thesis of Pjotr Thibaudier. Discussions with him were quite helpful.
	
\section{Skew product systems}

\subsection{Lebesgue measure is stationary}	
As usual an approach using a skew product system aids to describe the iterated function system as a single dynamical system, and to use the machinery of dynamical systems theory and ergodic theory. Write $\Sigma=  \{0,\ldots,M\}^\mathbb{N}$ for the space of one-sided infinite sequences of symbols in $\{0,  \ldots, M\}$, endowed with the product topology obtained from the discrete topology on $\{0,\ldots,M\}$. Elements $\omega \in \Sigma$ will be written as  
$\omega = (\omega_i)_{i\in\mathbb{N}}$. Let $\sigma: \Sigma \to \Sigma$ be the {\em left shift operator} defined by
\[(\sigma \omega)_i = \omega_{i+1}, \quad i \ge 0. \]
Write $[a_0\cdots a_k]$ for the {\em cylinder} 
\[ [a_0\cdots a_k] = \{ \omega \in \Sigma \, ; \, \omega_j  = a_j, \, 0\le j \le k\}.\]
We equip $\Sigma$ with the Borel $\sigma$-algebra. Given any $0 < p_0 < 1$ and the corresponding positive probability vector $\mathbf p$ as specified in \eqref{q:vectorp}, we let $\nu = \nu_{\mathbf p}$ denote the Bernoulli measure on $\Sigma$ that is defined on the cylinder sets by 
\[ \nu ([a_0\cdots a_k]) = \prod_{j=0}^k  p_{a_j}.\]
The measure $\nu$ is an ergodic invariant measure for the shift map $\sigma$.

\vskip .2cm
Define the skew product system $F: \Sigma \times [0,1) \to \Sigma \times [0,1)$ by
\[ F(\omega,x) = (\sigma \omega , f_{\omega_0} (x)).\]
We use the notation $F^k (\omega,x) = (\sigma^k \omega , f^k_\omega (x))$ for iterates, where $f^k_\omega$ is as defined in \eqref{q:iterates}. We also write $f^k_\eta (x)$ for elements $\eta = \eta_0 \cdots \eta_{m-1} \in \{0,\ldots,M\}^m$, called {\em words}, with $k \le m$. With slight abuse of notation we will use $\lambda$ to denote the one-, two- and three-dimensional Lebesgue measure, since the meaning will be clear from the context.
		
\begin{proposition}\label{p:inv}
Let $0 < p_0 < 1$. Then the corresponding product measure $\mu:= \nu \times \lambda$ on $\Sigma \times [0,1)$ is an invariant probability measure for $F$.
\end{proposition}
		
\begin{proof}
For invariance it suffices to consider product sets $A = [a_0 \cdots a_j] \times J$ of cylinder sets $[a_0 \cdots a_j]$ and intervals $J$. Note that for each $x \in [0,1)$ there are $N$ inverse images in $f_0^{-1}\{x\}$ and there is a unique $1 \le j \le M$ for which an inverse image $y\in [0,1)$ with $f_j(y)=x$ exists. One immediately computes that
\[ \begin{split}
\mu (F^{-1}(A)) =\ &  \mu \left(\bigcup_{i=0}^M [i \, a_0 \cdots a_j] \times f_i^{-1} (J) \right)\\
=\ & \nu([a_0\cdots a_j]) \left[ p_0  \lambda (f_0^{-1} (J)) + \frac{1-p_0}{M}  \sum_{i=1}^M \lambda(f_i^{-1} (J))\right]\\
=\ & \nu([a_0\cdots a_j]) \left[ p_0  N \frac{\lambda (J)}{N} +  \frac{1-p_0}{M} M \lambda(J)\right]\\
=\ & \mu (A). \qedhere
\end{split}\]
\end{proof}
	
Note that the proof of Proposition~\ref{p:inv} uses the specifics of the probability vector $\mathbf p$.

\vskip .2cm
Invariance of $\mu$ for $F$ implies that $\lambda$ is a {\em stationary measure} for the iterated function system $\{ f_i \, ; \, 0 \le i \le M\}$ with probability vector $\mathbf p$ in the sense that
\begin{align*} 
	\lambda &= \sum_{i=0}^M   p_i  (f_i)_* \lambda.
\end{align*}	
Here $(f_i)_*$ stands for the push forward measure $(f_i)_* \lambda (A) = \lambda (f_i^{-1} (A))$. Therefore, a direct consequence of Proposition~\ref{p:inv} above is the following.

\begin{corollary}\label{c:diagonal}
The diagonal $\Delta = \{ (x,x) \, ; \, x \in [0,1) \}$ is an invariant set for the iterated function system $\{ f^{(2)}_i \, ; \, 0 \le i \le M \}$ from \eqref{q:2ptifs} with probability vector $\mathbf p$ and Lebesgue measure restricted to $\Delta = \{ (x,x) \, ; \, x \in [0,1) \}$ is a stationary measure.
\end{corollary}

Below we will also verify the ergodicity of the measure $\mu$ for the skew product $F$. Instead of writing that $\mu$ is ergodic, we also say that the corresponding stationary measure $\lambda$ is ergodic to mean the same. The proofs of ergodicity provided in the next section are different for the three cases identified in Theorem~\ref{t:main}. They use a map that is isomorphic to $F$ as well as an extension of this map. Later we will also use a multivalued map.  For easy reference we use the remainder of this section to introduce all these different maps.

\subsection{One- two- and three-dimensional piecewise affine maps}\label{s:generalbaker}	
We first conjugate the shift map to an expanding interval map. Write  
\[ r_i =   \sum_{j=0}^{i-1}  p_j, \qquad 0\le i \le M+1.\]	
This gives $0 = r_0 < r_1 < \cdots < r_M < r_{M+1} = 1$. Define the expanding interval map $L:[0,1] \to [0,1]$ by setting 
\begin{equation}\label{e:L}
L(w) =  \begin{cases}
\displaystyle \frac{w}{p_0}, & 0 \le w < p_0,\\
\\
\displaystyle \frac{M(w-r_i)}{1-p_0}, &  r_i \le w <  r_{i+1}, \, 1\le i \le  M.
\end{cases}
\end{equation}
See Figure~\ref{f:graph22half}(a) for an example. Then the map $h: \Sigma \to [0,1]$ given by
\begin{equation}\label{q:isomh}
h(\omega) = \sum_{i=0}^\infty  \prod_{j=0}^{i-1}  p_{\omega_j}  r_{\omega_i}
\end{equation}
satisfies $h \circ \sigma = L \circ h$ and $h_\ast \nu = \lambda$. There is only a countable set of codes in $\Sigma$ on which $h$ is not injective. So, as $h$ is invertible after removing sets of zero measure, it defines a measurable isomorphism. From this we see that the skew product map $F$ is measurably isomorphic to $G: [0,1)^2 \to [0,1)^2$ given by
\begin{equation}\label{e:G}
G(w,x) =  \begin{cases}
\displaystyle \left(\frac{w}{p_0} , \, Nx \hspace{-.2cm} \pmod 1\right), & 0 \le w < p_0,\\
\\
\displaystyle \left( \frac{M(w - r_i)}{1-p_0},\,  \frac{x+i-1}{M} \right), &  r_i \le w <  r_{i+1}, 1\le i \le M.
\end{cases}
\end{equation}
See Figure~\ref{f:graph22half}(b) for an example.

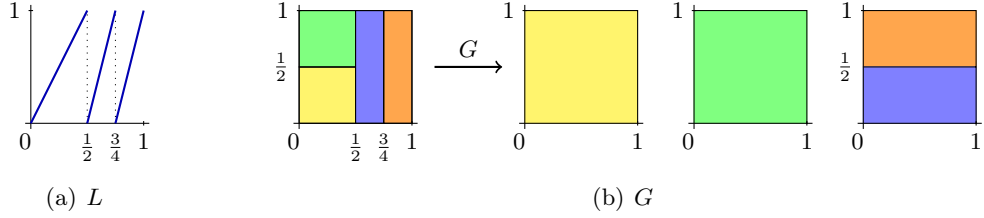
\begin{figure}[h]
\begin{center}
\subfigure[$L$]{
\begin{tikzpicture}[scale=1.5]
\draw(-.05,0)node[below]{\footnotesize $0$}--(1/2,0)node[below]{\footnotesize $\frac12$}--(3/4,0)node[below]{\footnotesize $\frac34$}--(1,0)node[below]{\footnotesize $1$}--(1.05,0);
\draw(0,-.05)--(0,1)node[left]{\footnotesize 1}--(0,1.05);
\draw(1,-.02)--(1,.02)(-.02,1)--(.02,1);
\draw[dotted](1/2,0)--(1/2,1)(3/4,0)--(3/4,1);
\draw[thick, blue!70!black] (0,0)--(1/2,1)(1/2,0)--(3/4,1)(3/4,0)--(1,1);
\end{tikzpicture}}
\hspace{1cm}
\subfigure[$G$]{
\begin{tikzpicture}[scale=1.5]
\filldraw[fill=yellow!70] (0,0) rectangle (1/2,1/2);
\filldraw[fill=green!50] (0,.5) rectangle (1/2,1);
\filldraw[fill=blue!50](1/2,0) rectangle (.75,1);
\filldraw[fill=orange!70] (.75,0) rectangle (1,1);
\draw(-.05,0)node[below]{\footnotesize $0$}--(1/2,0)node[below]{\footnotesize $\frac12$}--(3/4,0)node[below]{\footnotesize $\frac34$}--(1,0)node[below]{\footnotesize $1$}--(1.05,0)(1,0)--(1,1)--(0,1)(0,.5)--(.5,.5)(.5,0)--(.5,1)(.75,0)--(.75,1);
\draw(0,-.05)--(0,.5)node[left]{\footnotesize $\frac12$}--(0,1)node[left]{\footnotesize 1}--(0,1.05);
\draw(1,-.02)--(1,.02)(-.02,1)--(.02,1);

\draw[->, thick] (1.2,.5)--(1.8,.5);
\node[above] at (1.5,.5) {\footnotesize $G$};

\filldraw[yellow!70] (2,0) rectangle (3,1);
\draw(1.95,0)node[below]{\footnotesize $0$}--(3,0)node[below]{\footnotesize $1$}--(3.05,0)(3,0)--(3,1)--(2,1);
\draw(2,-.05)--(2,1)node[left]{\footnotesize 1}--(2,1.05);
\draw(3,-.02)--(3,.02)(1.98,1)--(2.02,1);

\filldraw[green!50] (3.5,0) rectangle (4.5,1);
\draw(3.45,0)node[below]{\footnotesize $0$}--(4.5,0)node[below]{\footnotesize $1$}--(4.55,0)(4.5,0)--(4.5,1)--(3.5,1);
\draw(3.5,-.05)--(3.5,1)node[left]{\footnotesize 1}--(3.5,1.05);
\draw(4.5,-.02)--(4.5,.02)(3.48,1)--(3.52,1);

\filldraw[orange!70] (5,.5) rectangle (6,1);
\filldraw[blue!50] (5,0) rectangle (6,.5);
\draw(4.95,0)node[below]{\footnotesize $0$}--(6,0)node[below]{\footnotesize $1$}--(6.05,0)(6,0)--(6,1)--(5,1)(5,.5)--(6,.5);
\draw(5,-.05)--(5,.5)node[left]{\footnotesize $\frac12$}--(5,1)node[left]{\footnotesize 1}--(5,1.05);
\draw(6,-.02)--(6,.02)(4.98,1)--(5.02,1);

\end{tikzpicture}}
\caption{Graphs of $L$ and $G$ for $(M,N) = (2,2)$ and $p_0=\frac12$. $G$ maps the colored areas in the unit square on the left to the areas of the same color on the right.}
\label{f:graph22half}
\end{center}
\end{figure}

\vskip .2cm
Consider the invertible extension $\Gamma: [0,1)^3 \to [0,1)^3$ of the maps $L: [0,1) \to [0,1)$  from \eqref{e:L} and $G: [0,1)^2 \to [0,1)^2$ from \eqref{e:G} given by
\begin{equation}\label{e:Gamma}
	\Gamma (w,x,y) =\begin{cases}
		\displaystyle \left(\frac{w}{p_0} , Nx -j , \frac{p_0(y+j)}{N} \right),  
		& \begin{array}{@{}c@{}} 0 \le w < p_0, \\ 
			\frac{j}{N} \le x < \frac{j+1}{N}, 0\le j < N,
		\end{array}
		\\
		\\
		\displaystyle \left( \frac{M(w - r_{i})}{1-p_0} , \frac{x+i-1}{M}  , (1-p_0) y + p_0 \right),  & r_{i} \le w < r_{i+1},  1\le i \le M. 
	\end{cases}
\end{equation}
For $M=N=2$ and $p_0=\frac12$, we get
\[ \Gamma (w,x,y) =	\begin{cases} 
	\displaystyle \left(2 w , 2x -j, \frac{y+j}{4} \right),  & \begin{array}{@{}c@{}} 0 \le w < 1/2, \\ 
		\frac{j}{2} \le x < \frac{j+1}{2}, j=0,1,
	\end{array}\\
	\\
	\displaystyle \left( 4 w - (2+i)  , \frac{x+i}{2}  , \frac{y+1}{2} \right),  &  \frac12 + \frac{i}{4} \le w < \frac12 + \frac{i+1}{4},  i=0,1,
\end{cases}\]
a graphical depiction of which is shown in Figure~\ref{f:baker}. This particular map is somewhat reminiscent of the two-dimensional baker map $B$ on $[0,1)^2$ given by
\[ B(w,x) =   \begin{cases}
	\displaystyle \left(2w , \frac{x}{2}\right), & 0\le w < \frac{1}{2},\\
	\\
	\displaystyle \left(2w -1  , \frac{x+1}{2}\right), &  \frac{1}{2} \le w < 1,
\end{cases}\]
which has an expanding and a contracting direction, or more specific, a positive Lyapunov exponent $\ln (2)$ and a negative Lyapunov exponent $-\ln(2)$. The iterated function systems that we analyze in this article thus inspire three-dimensional analogues of the baker map. Similar maps feature in \cite{MR4287984} in a study of hetero-chaos.

\vskip .2cm
The map $\Gamma$ is invertible, the inverse being given by
\[ \Gamma^{-1} (w,x,y) = \begin{cases}
	\displaystyle \left( p_0 w  , \frac{x+j}{N} ,  \frac{Ny}{p_0} -j \right), & \frac{jp_0}{N}  \le y <  \frac{(j+1)p_0}{N},  0\le j < N,\\
	\\
	\displaystyle \left(\frac{(1-p_0)(w+i)}{M}  + p_0, M x -i ,   \frac{y-p_0}{1-p_0}  \right), & \begin{array}{@{}c@{}} p_0 \le y < 1, \\ 
		\frac{i}{M} \le x < \frac{i+1}{M},  0\le i < M.
	\end{array}
\end{cases}\]
Note that all maps $L$, $G$ and $\Gamma$ have Lebesgue measure, with appropriate dimension, as invariant measure. We have the following relation between $\Gamma$ and $F$ (for the purpose of the statement considered on compact spaces).

\begin{lemma}\label{p:factor}
The skew product $F: \Sigma \times [0,1] \to \Sigma \times [0,1]$ is a factor of the three-dimensional map $\Gamma: [0,1]^3 \to [0,1]^3$.
\end{lemma}

\begin{proof}
	Recall the definition of the isomorphism $h: \Sigma \to [0,1]$ between the map $L:[0,1] \to [0,1]$ and the left shift $\sigma: \Sigma \to \Sigma$ from \eqref{q:isomh}. Let $\pi_{w,x}: [0,1]^3 \to [0,1]^2, \, (w,x,y) \to (w,x)$ be the canonical projection onto the first two coordinates. One easily verifies that the map $h^{-1} \circ \pi_{w,x} : [0,1]^3 \to \Sigma \times [0,1]$ (up to sets of measure zero) is surjective, measurable, measure preserving and satisfies $F \circ (h^{-1} \circ \pi_{w,x}) = (h^{-1} \circ \pi_{w,x}) \circ \Gamma$, thus constituting a factor map. 
\end{proof} 

\tdplotsetmaincoords{70}{115}

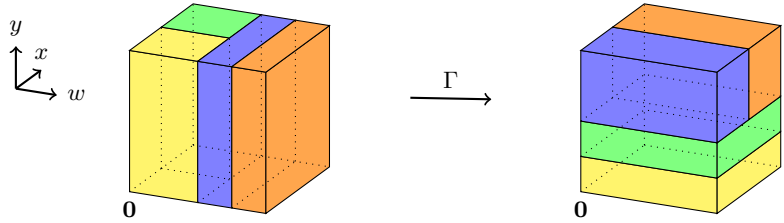
\begin{figure}[ht]
	\begin{center}
		\begin{tikzpicture}[scale=2,tdplot_main_coords]  
			\draw[fill=yellow!70] (0,0,0) -- (0,.5,0) -- (0,.5,1) -- (0,0,1) -- cycle;
			\draw[fill=yellow!70] (0,0,1)--(0,.5,1)--(-1/2,1/2,1)--(-1/2,0,1)--cycle;
			\draw[fill=green!50] (-1/2,0,1)--(-1/2,.5,1)--(-1,1/2,1)--(-1,0,1)--cycle; 
			\draw[fill=blue!50!white] (0,1/2,0)--(0,3/4,0)--(0,3/4,1)--(0,1/2,1)--cycle; 
			\draw[fill=blue!50] (0,1/2,1)--(0,3/4,1)--(-1,3/4,1)--(-1,1/2,1)--cycle; 
			\draw[fill=orange!70] (0,3/4,0)--(0,1,0)--(0,1,1)--(0,3/4,1)--cycle;
			\draw[fill=orange!70] (0,1,0)--(-1,1,0)--(-1,1,1)--(0,1,1)--cycle;
			\draw[fill=orange!70] (0,3/4,1)--(0,1,1)--(-1,1,1)--(-1,3/4,1)--cycle;
			
			\draw[dotted] (-1,0,0)--(-1,0,1)(0,0,0)--(-1,0,0)(-1,0,0)--(-1,1,0);
			\draw[dotted] (0,3/4,0)--(-1,3/4,0)(0,1/2,0)--(-1,1/2,0)(-1,3/4,0)--(-1,3/4,1)(-1,1/2,0)--(-1,1/2,1);
			\draw[dotted] (-1/2,0,0)--(-1/2,0,1)(-1/2,1/2,0)--(-1/2,1/2,1)(-1/2,0,0)--(-1/2,1/2,0);
			
			
			\draw[->, thick] (-1,-1.3,.2)--(-1,-1,.2)node[right]{\footnotesize $w$};
			\draw[->, thick] (-1,-1.3,.2)--(-1.4,-1.3,.2)node[above]{\footnotesize $x$};
			\draw[->, thick] (-1,-1.3,.2)--(-1,-1.3,.5)node[above]{\footnotesize $y$}; 
			
			\draw[->, thick] (-1.2,1.5,.5)--(-1.4,2,.5);
			\node[above] at (-1.3,1.75,.5) {\footnotesize $\Gamma$};
			
			\node[below] at (0,0,0) {\footnotesize $\mathbf 0$};
			
		\end{tikzpicture}  
		\hspace{.7cm}
		\begin{tikzpicture}[scale=2,tdplot_main_coords]  
			\draw[fill=yellow!70] (0,0,0) -- (0,1,0) -- (0,1,1/4) -- (0,0,1/4) -- cycle;
			\draw[fill=yellow!70] (0,1,0)--(-1,1,0)--(-1,1,1/4)--(0,1,1/4)--cycle;
			\draw[fill=green!50] (0,0,1/4)--(0,1,1/4)--(0,1,1/2)--(0,0,1/2)--cycle;
			\draw[fill=green!50] (0,1,1/4)--(-1,1,1/4)--(-1,1,1/2)--(0,1,1/2)--cycle; 
			\draw[fill=blue!50] (0,0,1/2)--(0,1,1/2)--(0,1,1)--(0,0,1)--cycle; 
			\draw[fill=blue!50] (0,0,1)--(0,1,1)--(-1/2,1,1)--(-1/2,0,1)--cycle; 
			\draw[fill=blue!50] (0,1,1/2)--(-1/2,1,1/2)--(-1/2,1,1)--(0,1,1)--cycle; 
			\draw[fill=orange!70] (-1/2,0,1)--(-1,0,1)--(-1,1,1)--(-1/2,1,1)--cycle;
			\draw[fill=orange!70] (-1/2,1,1/2)--(-1/2,1,1)--(-1,1,1)--(-1,1,1/2)--cycle;
			
			\draw[dotted] (-1,0,0)--(-1,0,1)(0,0,0)--(-1,0,0)(-1,0,0)--(-1,1,0);
			\draw[dotted] (0,0,1/4)--(-1,0,1/4)--(-1,1,1/4)(0,0,1/2)--(-1,0,1/2)--(-1,1,1/2);
			\draw[dotted] (-1/2,0,1)--(-1/2,0,1/2)--(-1/2,1,1/2);
			
			\node[below] at (0,0,0) {\footnotesize $\mathbf 0$};
		\end{tikzpicture}
	\end{center}
	\caption{The map $G$ for $M=N=2$ and $p_0=\frac12$ maps the regions on the left to the regions on the right according to the colors.}
	\label{f:baker}
\end{figure}

\subsection{Multivalued maps}\label{s:multivalued}

It is sometimes helpful to consider an associate iterated function system of multivalued maps. Write $\mathbb{K}$ for the class of nonempty compact subsets of $[0,1)$. Define the multivalued map $F_1 : \mathbb{K} \to \mathbb{K}$ by 
\begin{equation}\label{e:F0}
F_1   (A) = \bigcup_{i=1}^M f_i(A),
\end{equation}
in which the contracting maps $f_1,\ldots,f_M$ are combined.
We also write 
\begin{equation}\label{e:F1}
F_0 (A) = f_0(A).
\end{equation}
We can then look at the iterated function system generated by $F_0$ and $F_1$. A composition $F^n_\eta$ with $\eta \in \{0,1\}^\mathbb{N}$ is a multivalued map. As all maps $f_1,\ldots,f_M$ that make up $F_1$ have the same constant derivative $1/M$, and $f_0$ has constant derivative $N$, we can speak of $(F^n_\eta)'$. The graphs in $F_\eta^n$ are equally spaced line pieces with constant slope $(F^n_\eta)'$. The number of elements in the set $F_\eta^n (\{x\})$ is independent of $x \in [0,1)$, so $\# F_\eta^n (\{x\}) = \# F_\eta^n (\{0\})$, and $F^n_\eta (\{0\})$ is always of the form 
\[S_j := \{ i/M^j \; ; \; 0\le i < M^j\}\]
for some $j \ge 0$. Let $\Pi : \Sigma \to \Sigma_2$ be the projection that maps all symbols $1,\ldots,M$ to $1$; so $\Pi (\omega) = \eta$ with
\[ \eta_i = \begin{cases} 
0, & \omega_i =0, \\
1, & \omega_i \in \{ 1,\ldots,M\}.
\end{cases}\]
We use the next lemma in the section on intermittency.

\begin{lemma}\label{l:feta}
Let $\eta = \eta_0 \cdots \eta_{n-1} \in \{0,1\}^n$, $n \ge 1$, and let $j \ge 0$ be such that $F^n_\eta (\{0\}) =  S_j$. Then for each $i \ne k$,
\[ \nu (\{ \omega \in \Pi^{-1}[\eta] \, ; \, f_\omega^n(0) = i/M^j \}) = \nu (\{ \omega \in \Pi^{-1}[\eta] \, ; \, f_\omega^n(0) = k/M^j \}).\]   
\end{lemma}

\begin{proof}
Set $\gamma = \# \{ 0 \le i \le n-1\, ; \,  \eta_i =1  \}$ for the number of occurrences of the digit 1 in $\eta$. Note that $\Pi^{-1}[\eta]$ is the disjoint union of $M^\gamma$ cylinders of length $n$ in $\Sigma$. Note also that $F_1(S_0) = S_M$ and $F_1(S_l) = S_{l M}$ for any $l >0$, while $F_0 (S_l) \subseteq S_l$. There are $i \neq k$ such that $f_0(i/M^l) = f_0(k/M^l)$ if and only if there is an $i$ such that $f_0(i/M^l) =0$, so such that $Ni/M^l \in \mathbb N$, if and only if $N$ and $M$ share a common prime factor. Hence, if $M$ and $N$ are relatively prime, then $F^n_\eta (\{0\}) = S_\gamma$ and for each $i$,
\[ \nu (\{ \omega \in \Pi^{-1}[\eta] \, ; \, f_\omega^n(0) = i/M^\gamma \}) = p^{n-\gamma} \Big( \frac{1-p}{M} \Big)^\gamma.\]
Suppose $N$ and $M$ are not relatively prime. The map $f_0$ wraps the unit interval around itself $N$ times with constant expansion factor. So, for any $0 \le m < l$ for which $f_0 (S_l) = S_m$ it follows that for each $i_1,i_2$,
\[ \# \{ 0 \le k \le M^l-1 \, ; \, f_0 (k/M^l) = i_1 \} = \#  \{ 0 \le k \le M^l-1 \, ; \, f_0 (k/M^l) = i_2 \}.\]
Since all cylinders $[\omega_0 \cdots  \omega_{n-1}] \subseteq \Pi^{-1}[\eta]$ have equal $\nu$-measure, this implies the lemma.
\end{proof}

Properties of graphs of $f^n_\omega$ and $F^n_\eta$ in relation to each other
are illustrated in Figures~\ref{f:1010010}~and~\ref{f:0010010}.

\begin{figure}[!ht]
\begin{center}
\subfigure[$(3,3),\eta=(11010010)$]{
\begin{tikzpicture}[scale=3.5]
\draw(-.05,0)node[below]{\footnotesize $0$}--(1/9,0)node[below]{\footnotesize $\frac19$}--(3/9,0)node[below]{\footnotesize $\frac13$}--(5/9,0)node[below]{\footnotesize $\frac59$}--(7/9,0)node[below]{\footnotesize $\frac79$}--(1,0)node[below]{\footnotesize $1$}--(1.05,0);
\draw(0,-.05)--(0,1/9)node[left]{\footnotesize $\frac19$}--(0,1/3)node[left]{\footnotesize $\frac13$}--(0,5/9)node[left]{\footnotesize $\frac59$}--(0,7/9)node[left]{\footnotesize $\frac79$}--(0,1)node[left]{\footnotesize 1}--(0,1.05);
\draw(1,-.02)--(1,.02)(1/9,-.01)--(1/9,.01)(2/9,-.01)--(2/9,.01)(3/9,-.01)--(1/3,.01)(4/9,-.01)--(4/9,.01)(5/9,-.01)--(5/9,.01)(6/9,-.01)--(6/9,.01)(7/9,-.01)--(7/9,.01)(8/9,-.01)--(8/9,.01);
\draw(-.02,1)--(.02,1)(-.01,1/9)--(.01,1/9)(-.01,2/9)--(.01,2/9)(-.01,1/3)--(.01,1/3)(-.01,4/9)--(.01,4/9)(-.01,5/9)--(.01,5/9)(-.01,6/9)--(.01,6/9)(-.01,7/9)--(.01,7/9)(-.01,8/9)--(.01,8/9);
\draw[thick, blue!70!black] (0,0)--(1,1)(1/9,0)--(1,8/9)(2/9,0)--(1,7/9)(3/9,0)--(1,6/9)(4/9,0)--(1,5/9)(5/9,0)--(1,4/9)(6/9,0)--(1,3/9)(7/9,0)--(1,2/9)(8/9,0)--(1,1/9)(0,1/9)--(8/9,1)(0,2/9)--(7/9,1)(0,3/9)--(6/9,1)(0,4/9)--(5/9,1)(0,5/9)--(4/9,1)(0,6/9)--(3/9,1)(0,7/9)--(2/9,1)(0,8/9)--(1/9,1);
\draw[ultra thick, red!70!black] (0,1/9)--(1/9,2/9)(1/9,1/9)--(2/9,2/9)(2/9,1/9)--(3/9,2/9)(3/9,1/9)--(4/9,2/9)(4/9,1/9)--(5/9,2/9)(5/9,1/9)--(6/9,2/9)(6/9,1/9)--(7/9,2/9)(7/9,1/9)--(8/9,2/9)(8/9,1/9)--(1,2/9);
\end{tikzpicture}
}
\hspace{.5cm}
\subfigure[$(3,2),\eta=(1010010)$]{
\begin{tikzpicture}[scale=3.5]
\draw(-.05,0)node[below]{\footnotesize $0$}--(1/8,0)node[below]{\footnotesize $\frac18$}--(1/4,0)node[below]{\footnotesize $\frac14$}--(3/8,0)node[below]{\footnotesize $\frac38$}--(1/2,0)node[below]{\footnotesize $\frac12$}--(5/8,0)node[below]{\footnotesize $\frac58$}--(3/4,0)node[below]{\footnotesize $\frac34$}--(7/8,0)node[below]{\footnotesize $\frac78$}--(1,0)node[below]{\footnotesize $1$}--(1.05,0);
\draw(0,-.05)--(0,1/9)node[left]{\footnotesize $\frac19$}--(0,1/3)node[left]{\footnotesize $\frac13$}--(0,5/9)node[left]{\footnotesize $\frac59$}--(0,7/9)node[left]{\footnotesize $\frac79$}--(0,1)node[left]{\footnotesize 1}--(0,1.05);
\draw(1,-.02)--(1,.02)(1/16,-.01)--(1/16,.01)(1/8,-.01)--(1/8,.01)(3/16,-.01)--(3/16,.01)(1/4,-.01)--(1/4,.01)(5/16,-.01)--(5/16,.01)(3/8,-.01)--(3/8,.01)(7/16,-.01)--(7/16,.01)(1/2,-.01)--(1/2,.01)(9/16,-.01)--(9/16,.01)(5/8,-.01)--(5/8,.01)(11/16,-.01)--(11/16,.01)(3/4,-.01)--(3/4,.01)(13/16,-.01)--(13/16,.01)(14/16,-.01)--(14/16,.01)(15/16,-.01)--(15/16,.01);
\draw(-.02,1)--(.02,1)(-.01,1/27)--(.01,1/27)(-.01,2/27)--(.01,2/27)(-.01,3/27)--(.01,3/27)(-.01,4/27)--(.01,4/27)(-.01,5/27)--(.01,5/27)(-.01,6/27)--(.01,6/27)(-.01,7/27)--(.01,7/27)(-.01,8/27)--(.01,8/27)(-.01,9/27)--(.01,9/27)(-.01,10/27)--(.01,10/27)(-.01,11/27)--(.01,11/27)(-.01,12/27)--(.01,12/27)(-.01,13/27)--(.01,13/27)(-.01,14/27)--(.01,14/27)(-.01,15/27)--(.01,15/27)(-.01,16/27)--(.01,16/27)(-.01,17/27)--(.01,17/27)(-.01,18/27)--(.01,18/27)(-.01,19/27)--(.01,19/27)(-.01,20/27)--(.01,20/27)(-.01,21/27)--(.01,21/27)(-.01,22/27)--(.01,22/27)(-.01,23/27)--(.01,23/27)(-.01,24/27)--(.01,24/27)(-.01,25/27)--(.01,25/27)(-.01,26/27)--(.01,26/27);
\draw[thick,dotted](0,0)--(1,1);
\draw[thick, blue!70!black] (0,0)--(1,16/27)(1/16,0)--(1,15/27)(2/16,0)--(1,14/27)(3/16,0)--(1,13/27)(4/16,0)--(1,12/27)(5/16,0)--(1,11/27)(6/16,0)--(1,10/27)(7/16,0)--(1,9/27)(8/16,0)--(1,8/27)(9/16,0)--(1,7/27)(10/16,0)--(1,6/27)(11/16,0)--(1,5/27)(12/16,0)--(1,4/27)(13/16,0)--(1,3/27)(14/16,0)--(1,2/27)(15/16,0)--(1,1/27);
\draw[thick, blue!70!black](0,1/27)--(1,17/27)(0,2/27)--(1,18/27)(0,3/27)--(1,19/27)(0,4/27)--(1,20/27)(0,5/27)--(1,21/27)(0,6/27)--(1,22/27)(0,7/27)--(1,23/27)(0,8/27)--(1,24/27)(0,9/27)--(1,25/27)(0,10/27)--(1,26/27)(0,11/27)--(1,1);
\draw[thick, blue!70!black](0,12/27)--(15/16,1)(0,13/27)--(14/16,1)(0,14/27)--(13/16,1)(0,15/27)--(12/16,1)(0,16/27)--(11/16,1)(0,17/27)--(10/16,1)(0,18/27)--(9/16,1)(0,19/27)--(8/16,1)(0,20/27)--(7/16,1)(0,21/27)--(6/16,1)(0,22/27)--(5/16,1)(0,23/27)--(4/16,1)(0,24/27)--(3/16,1)(0,25/27)--(2/16,1)(0,26/27)--(1/16,1);
\draw[ultra thick, red!70!black] (0,26/27)--(1/16,1)(1/16,18/27)--(4/16,21/27)(4/16,24/27)--(7/16,1)(7/16,18/27)--(8/16,19/27)(8/16,26/27)--(9/16,1)(9/16,18/27)--(12/16,21/27)(12/16,24/27)--(15/16,1)(15/16,18/27)--(1,19/27);
\end{tikzpicture}
}
\caption{Left picture: a plot of the graphs of $F^{8}_\eta$  for $(M,N) = (3,3)$ and $\eta = (11010010)$. The red graph is the graph of $f^{8}_\omega$ for $\omega = (12020020)$.
			Right picture: a plot of $F^7_\eta$  for $(M,N) = (3,2)$ and $\eta = (1010010)$. The red graph is the graph of $f^7_\omega$ for $\omega = (3020020)$.}
\label{f:1010010}
\end{center}	
\end{figure}
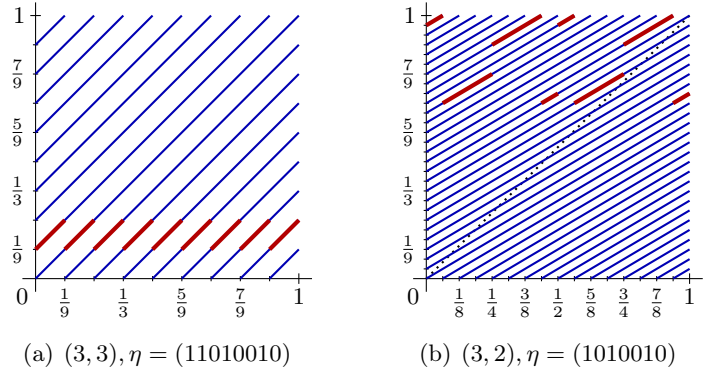

\section{Two-point dynamics}\label{s:2point}

This central section treats the dynamics of the skew product systems for different values of $(M,N)$ and $p_0$, focusing on convergence and divergence of orbits and statistical properties of orbits. We treat separately the cases with $L_{p_0} < 0$, $L_{p_0} = 0$ and $L_{p_0} > 0$. Note that for $p_0 = 1/2$, this is the same as $M>N$, $M=N$ and $M<N$, respectively. 
			
\subsection{$L_{p_0} < 0$ (Synchronization)}

If the contraction is stronger than the expansion,  one may expect the orbits of nearby points to converge to each other under identical compositions. The numerical observation in Figure~\ref{f:convergence} illustrates this. 
\begin{figure}[!ht]
	\begin{center}
		\includegraphics[height=3cm]{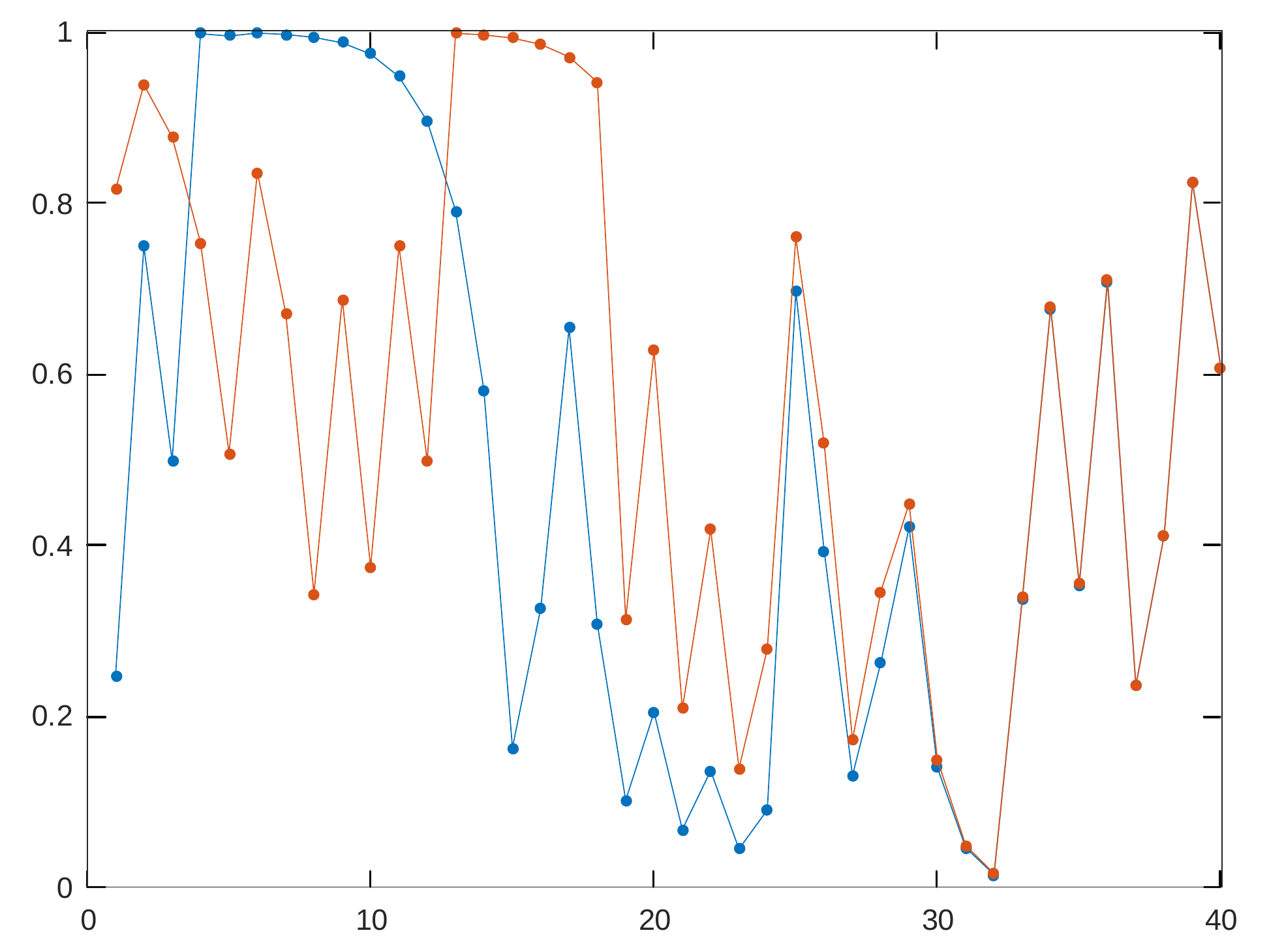}
		\hspace{1cm}
		\includegraphics[height=3cm]{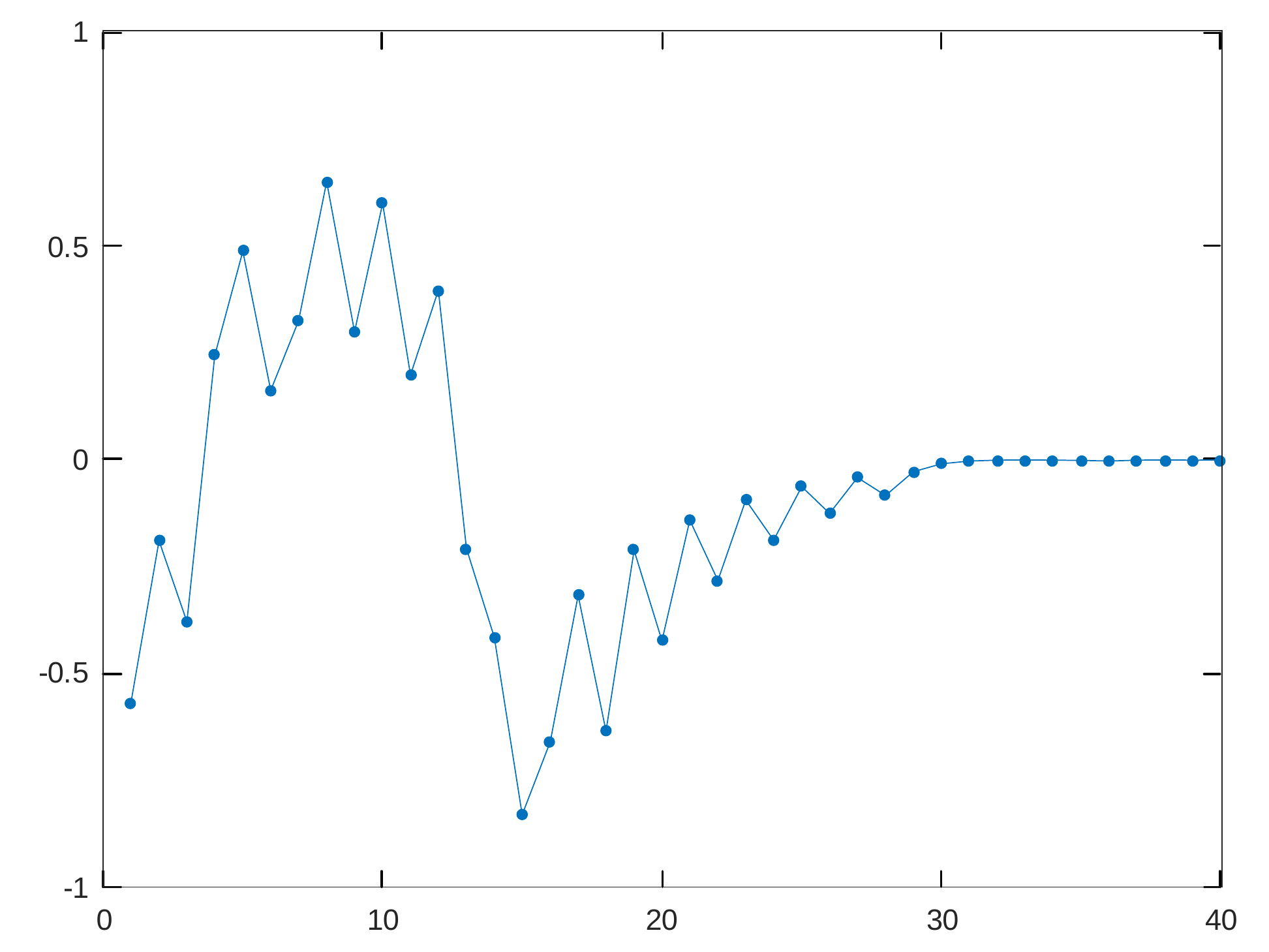}
		\caption{\label{f:convergence} Two time series of $f^n_\omega$ for two different initial points, for $(M,N)=(3,2)$ and $p_0=\frac12$. The signed difference between the two, depicted in the right panel, shows convergence of the orbits to each other.
		}
	\end{center}	
\end{figure}		

\vskip .2cm
We will establish such convergence in fact uniformly on $[0,1)$. The discontinuities of $f_0$ form an obstacle in the analysis, since nearby points are mapped a positive distance apart if they are on different sides of a point of discontinuity of $f_0$. Iterates $f^n_\omega$ may have many discontinuities on $[0,1)$, as the graph in Figure~\ref{f:graph<} illustrates. A Borel-Cantelli argument (see for instance \cite{MR3930614} for the Borel-Cantelli lemmas) however makes clear that orbits are only infrequently very close to the points of discontinuity of $f_0$, which allows to prove the following result.

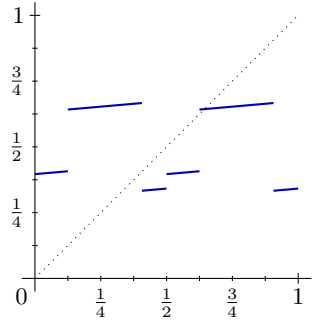
\begin{figure}[!ht]
\begin{center}
\begin{tikzpicture}[scale=3.5]
\draw(-.05,0)node[below]{\footnotesize $0$}--(1/4,0)node[below]{\footnotesize $\frac14$}--(1/2,0)node[below]{\footnotesize $\frac12$}--(3/4,0)node[below]{\footnotesize $\frac34$}--(1,0)node[below]{\footnotesize $1$}--(1.05,0);
\draw(0,-.05)--(0,1/4)node[left]{\footnotesize $\frac14$}--(0,1/2)node[left]{\footnotesize $\frac12$}--(0,3/4)node[left]{\footnotesize $\frac34$}--(0,1)node[left]{\footnotesize 1}--(0,1.05);
\draw(1,-.02)--(1,.02)(1/8,-.01)--(1/8,.01)(2/8,-.01)--(2/8,.01)(3/8,-.01)--(3/8,.01)(4/8,-.01)--(4/8,.01)(5/8,-.01)--(5/8,.01)(6/8,-.01)--(6/8,.01)(7/8,-.01)--(7/8,.01);
\draw(-.02,1)--(.02,1)(-.01,1/8)--(.01,1/8)(-.01,2/8)--(.01,2/8)(-.01,3/8)--(.01,3/8)(-.01,4/8)--(.01,4/8)(-.01,5/8)--(.01,5/8)(-.01,6/8)--(.01,6/8)(-.01,7/8)--(.01,7/8);
\draw[thick, blue!70!black] (0,289/729)--(1/8,11/27)(1/2,289/729)--(5/8,11/27)(29/32,1/3)--(1,83/243)(13/32,1/3)--(1/2,83/243)(5/8,52/81)--(29/32,2/3)(1/8,52/81)--(13/32,2/3);
\draw[dotted](0,0)--(1,1);
\end{tikzpicture}
\caption{The graph of an iterate $f^{12}_\omega$ for $(M,N)=(3,2)$ and some $\omega$. Although the slope $(f^{12}_\omega)'$  is small, the map is not a contraction on $[0,1)$ because of the discontinuities.}
\label{f:graph<}
\end{center}	
\end{figure}

\begin{theorem}\label{t:M>Nconvergence}
Consider $L_{p_0} < 0$. For all $x,y \in [0,1)$,
\[
\lim_{n\to \infty} | f^n_\omega (y) - f^n_\omega(x) | = 0
\]		
for $\nu$-almost all $\omega \in \Sigma$.
\end{theorem}	
	
\begin{proof}
Let $\zeta$ be a number with $e^{L_{p_0}} < \zeta < 1$. With $a_i = \ln (f_i')$, we can write 
\[ (f_\omega^n)' = e^{\sum_{i=0}^{n-1} a_{\omega_i}}.\]
Recall from \eqref{e:Lp0}  that
\[ \lim_{n\to\infty} \frac{1}{n} \sum_{i=0}^{n-1} a_{\omega_i} = L_{p_0} < \ln(\zeta) < 0,\] 
for $\nu$-almost all $\omega$. So, for $\nu$-almost all $\omega$,  $\sum_{i=0}^{n-1} a_{\omega_i}  -  n \ln(\zeta)$ converges to $-\infty$ and therefore $e^{\sum_{i=0}^{n-1} a_{\omega_i} } / \zeta^n$ goes to zero as $n\to\infty$. We find that for $\nu$-almost all $\omega$,
\[  \max  \{  (f_\omega^n)' / \zeta^n, n\ge 0  \}\] 
exists. Hence, for $\nu$-almost all $\omega \in \Sigma$, there exists a $C_\omega >0$ so that 
\[ (f^n_\omega)' \le C_\omega \zeta^n.\]
For any $\tilde{\varepsilon} >0$ one can choose $C>1$ and a set $\Omega_C \subset \Sigma$ of measure $\nu(\Omega_C) > 1 - \tilde{\varepsilon}$, so that $(f^n_\omega)' \le C \zeta^n$ for all $\omega \in \Omega_C$. Let $n_1 = n_1 (\tilde \varepsilon)$ be such that $C \zeta^{n_1} < 1/n_1^2$.

\vskip .2cm
Write
\begin{equation}\label{q:nbh}
B (r) = \bigcup_{i=1}^{N-1}  \left[\frac{i}{N} - r , \frac{i}{N} + r\right]
\end{equation}
for the $r$-neighborhood of the set $\mathcal{C} = \{1/N,\ldots,(N-1)/N\}$ of discontinuity points of $f_0$.
Let $E_n = \Sigma \times B(1/n^2)$. By the $F$-invariance of $\mu$ it holds that $\mu(F^{-n}(E_n)) = \mu(E_n) = 1/n^2$, so by the Borel-Cantelli lemma, we get for $\mu$-almost all $(\omega,x) \in \Sigma \times [0,1)$ that $F^n (\omega,x) \in E_n$ for at most finitely many n. Hence, the set of points
\[ B = \{ (\omega,x) \in \Sigma \times [0,1) \, ; \, \exists \, n_0 = n_0(\omega,x) \, \text{ s.t. } f^n_\omega (x) \not \in  B\left( 1/n^2 \right) \, \text{ for all } n \ge n_0 \}\]
satisfies $\mu(B)=1$.

\vskip .2cm
Write
\[ \mathcal C^* = \bigcup_{n \ge 0} \bigcup_{\stackrel{i_1\cdots i_n \in} {\{0,1, \ldots, M\}^n }} f_{i_1 \ldots i_n}^{-1} (\mathcal C)\]
for the set of points in $[0,1)$ that are eventually mapped to $\mathcal C$ by some $\omega \in \Sigma$. As $\mathcal C^*$ is a countable set,
\[ \mu \big((\Omega_C \times [0,1) \setminus \mathcal C^*) \cap B\big) > 1-\tilde \varepsilon,\]
which means that we can find an $x \in [0,1) \setminus \mathcal C^*$, such that
\begin{equation}\label{q:setlarger0}
\nu \big(\{ \omega \in \Sigma \, ; \, (\omega,x) \in (\Omega_C \times [0,1) \setminus \mathcal C^*) \cap B \} \big)>0.
\end{equation} 
Fix such a point $x$. Then for any $\omega$ in the set from \eqref{q:setlarger0} there is, by continuity, an open interval $J_\omega$ with $x \in J_\omega$, such that
\[ f_\omega^n (J_\omega) \cap \mathcal C = \emptyset, \quad \text{ for all } n < \max \{n_0,n_1 \}.\]
By the choice of $(\omega,x)$ we get for all $n \ge \max\{ n_0,n_1\} $ that $f_\omega^n (x) \not \in B(1/n^2)$ and $(f_\omega^n)' \le C \zeta^n < 1/n^2$. Hence, we recursively obtain that for all $n  \ge \max\{ n_0,n_1\}$ the set $f_\omega^n (J_\omega)$ is an interval and
\[ \lambda (f_\omega^n (J_\omega)) \le \frac{\lambda(f_\omega^{n-1} (J_\omega))}{ n^2} \le \frac1{n^2},\]
so that $f_\omega^n (J_\omega) \cap \mathcal C = \emptyset$.
Moreover, for every $\omega$ in the set from \eqref{q:setlarger0} there is an $n \ge 1$, such that
\[ \Big( x-\frac1n, x+\frac1n \Big) \subseteq J_\omega.\]
So we can find an $n_2 \ge 1$ such that the set
\[ \hat \Omega:= \Big\{ \omega \in \Sigma\, ; \,  (\omega,x) \in (\Omega_C \times [0,1) \setminus \mathcal C^*) \cap B \, \text{ and } \, \Big( x-\frac1{n_2}, x+\frac1{n_2} \Big) \subseteq J_\omega \Big\} \]
satisfies $\nu(\hat \Omega)>0$.

\vskip .2cm
For each $t \ge 1$ and $\eta \in \{1,\ldots,M\}^t$, the set $f^t_\eta ([0,1))$ is an interval of length $1/M^t$. The union of these intervals, varying over all $\eta \in  \{1,\ldots,M\}^t$ for fixed $t$, covers $[0,1)$. Hence, there exist $t \in \mathbb{N}$ and $\eta \in \{1,\ldots,M\}^t$ with $f^t_\eta ([0,1)) \subset \big( x-\frac1{n_2}, x+\frac1{n_2} \big) $. 
Then for each concatenated sequence $\tilde{\omega} = \eta\omega$, with $\omega \in \hat \Omega$, and each $n \ge 1$ the image $f^n_{\tilde{\omega}} ([0,1))$ is an interval with $\lim_{n\to\infty} \lambda (  f^n_{\tilde{\omega}} ([0,1)) ) = 0$. Hence, we have found a set $\Psi = \eta \hat \Omega \subset \Sigma$ with $\nu(\Psi)>0$, so that for any $y,z \in [0,1)$,
\[ \lim_{n \to \infty} |f_{\tilde \omega}^n(y)-f_{\tilde \omega}^n(z)|=0.\]

\vskip .2cm
Assume that there is a set $\Xi \subseteq \Sigma$ with $\nu(\Xi)>0$ of $\omega$ for which $f_\omega^n([0,1))$ is not contained in an interval of length shrinking to 0. We will derive a contradiction from this. By the Lebesgue density theorem we can take a density point $\xi$ of $\Xi$, meaning
\[
\lim_{j\to \infty}  \frac{\nu ( [\xi_1  \cdots \xi_j] \cap \Xi  )} {\nu(  [\xi_1 \cdots \xi_j] ) } = 1.
\]
(The Lebesgue density theorem is formulated for Lebesgue measure on the interval, but transfers to Bernoulli measure on $\Sigma$, compare \eqref{e:L}.) Then
\[ \lim_{j \to \infty} \nu (\sigma^j ( [\xi_1  \cdots \xi_j] \cap \Xi  ))=1\]
and moreover,
\[ \sigma^j ( [\xi_1  \cdots \xi_j] \cap \Xi  ) \subset \Xi.\]
This contradicts the construction of the set $\Psi$ with $\nu(\Psi)>0$, since $\Psi \cap \Xi = \emptyset$.
\end{proof}	
	
Next we set out to prove that the product measure $\mu = \nu \times \lambda$ is ergodic for $F$. To do so we use the system $\Gamma$ from Section~\ref{s:generalbaker}. The proof relies on the statements on the dynamics in Theorem~\ref{t:M>Nconvergence}.
	
\begin{theorem}\label{t:M>Nergodic}
Consider $L_{p_0} <0$. The measure $\mu$ is an ergodic invariant measure for $F$.
\end{theorem}		
	
\begin{proof}
Ergodicity of Lebesgue measure for $\Gamma$ will be established by exploiting invertibility of $\Gamma$ and using a Hopf argument as in \cite[Section~4.2.6]{MR3558990}. From Proposition~\ref{p:factor} it then follows that $\mu$ is ergodic for $F$. 

\vskip .2cm
Let $\varphi$ be a continuous function on $[0,1)^3$ and consider the time averages
\begin{align*}
	\varphi^+ (u) &= \lim_{n\to\infty} \frac{1}{n} \sum_{i=0}^{n-1}  \varphi (\Gamma^i (u)),
	\\
	\varphi^- (u) &= \lim_{n\to\infty} \frac{1}{n} \sum_{i=0}^{n-1}  \varphi (\Gamma^{-i} (u)).
\end{align*}	
As Lebesgue measure is invariant,
we find that there is a set $V \subset [0,1)^3$ of full Lebesgue measure,
so that for $u \in V$, 
the two limits exist and are equal (see \cite[Section~3.2.3]{MR3558990}):
\[
\varphi^+(u) = \varphi^- (u), \; \textrm{ for Lebesgue almost all } u.
\] 
For $u \in [0,1)^3$, write 
\[ W^s (u) = \{ v \in [0,1)^3 \, ; \, v = u + (0,x,y) \textrm{ for some } x,y\}\]
and
\[ W^u (u) = \{ v \in [0,1)^3 \, ; \, v = u + (w,0,0) \textrm{ for some } w\}.\]
Using Theorem~\ref{t:M>Nconvergence} we get that for $u$ in a set of full 
Lebesgue measure, if $v \in W^s(u)$ then $|\Gamma^n (u) - \Gamma^n(v)| \to 0$ as $n\to\infty$ and therefore $\varphi^+ (u) = \varphi^+(v)$.
Likewise for $u$ in a set of full 
Lebesgue measure, if $v \in W^u(u)$ then $|\Gamma^{-n} (u) - \Gamma^{-n}(v)| \to 0$ as $n\to\infty$ and therefore $\varphi^- (u) = \varphi^-(v)$.
We conclude that 
for $u$ in a set $U \subset V$ of full 
Lebesgue measure,
$\varphi^+ $ is constant along  $W^s (u)$ and $\varphi^- $ is constant along  $W^u (u)$.
 
 \vskip .2cm
As in \cite[Lemma~4.2.17]{MR3558990}, using Fubini's theorem one sees that there is a set $Y$ of full Lebesgue measure in $[0,1)^3$ so that for given $u,v \in Y$ there are $u',v' \in Y\cap U$ with $u' \in W^s (u)$, $ v' \in W^s (v)$ and moreover $v' \in W^u(u')$. It follows that for such points $u,v \in Y \cap U$, 
\[ \varphi^- (u) = \varphi^+ (u) =  \varphi^+ (u') = \varphi^- (u') =\varphi^- (v') = \varphi^+ (v') = \varphi^+ (v) = \varphi^- (v). \]
Hence, $\varphi^+$ and $\varphi^-$ exist and are constant on a set of full Lebesgue measure.  
\end{proof}		

As a corollary, typical orbits of the iterated function system $\{ f_i \, ; \, 0 \le i \le M \}$ are uniformly distributed on $[0,1)$. This is 
made explicit in the following result. Let $A \subset [0,1)$ and write $\chi_A$ for its characteristic function:
\[ \chi_A (x) = \begin{cases} 
0, & x \not\in A, \\
1, & x \in A.
\end{cases}\]

\begin{proposition}\label{p:uniformlylambda<0}
Consider $L_{p_0} <0$. For a Borel set $A$ with $\lambda(A)>0$,
\[ \lim_{n\to\infty} \frac{1}{n}  \sum_{i=0}^{n-1} \chi_A (f^i_{\omega} (x)) = \lambda(A)\]
for $\mu$-almost all $(\omega,x) \in \Sigma \times [0,1)$.
\end{proposition}	

\begin{proof}
The measure $\mu$ is ergodic for $F$ by Theorem~\ref{t:M>Nergodic}. The proposition therefore follows from an application of Birkhoff's ergodic theorem to the integrable function $\chi_A \circ \pi_2$, where $\pi_2: \Sigma \times [0,1) \to [0,1)$ is the canonical projection on the second coordinate of $F$. 	
\end{proof}

\subsection{$L_{p_0} = 0$ 
	(Intermittency)  }\label{s:2point=0}
In the case where expansion and contraction balance each other, a phenomenon reminiscent of intermittency arises.
	\begin{figure}[!ht]
	\begin{center}
		\includegraphics[height=3.5cm]{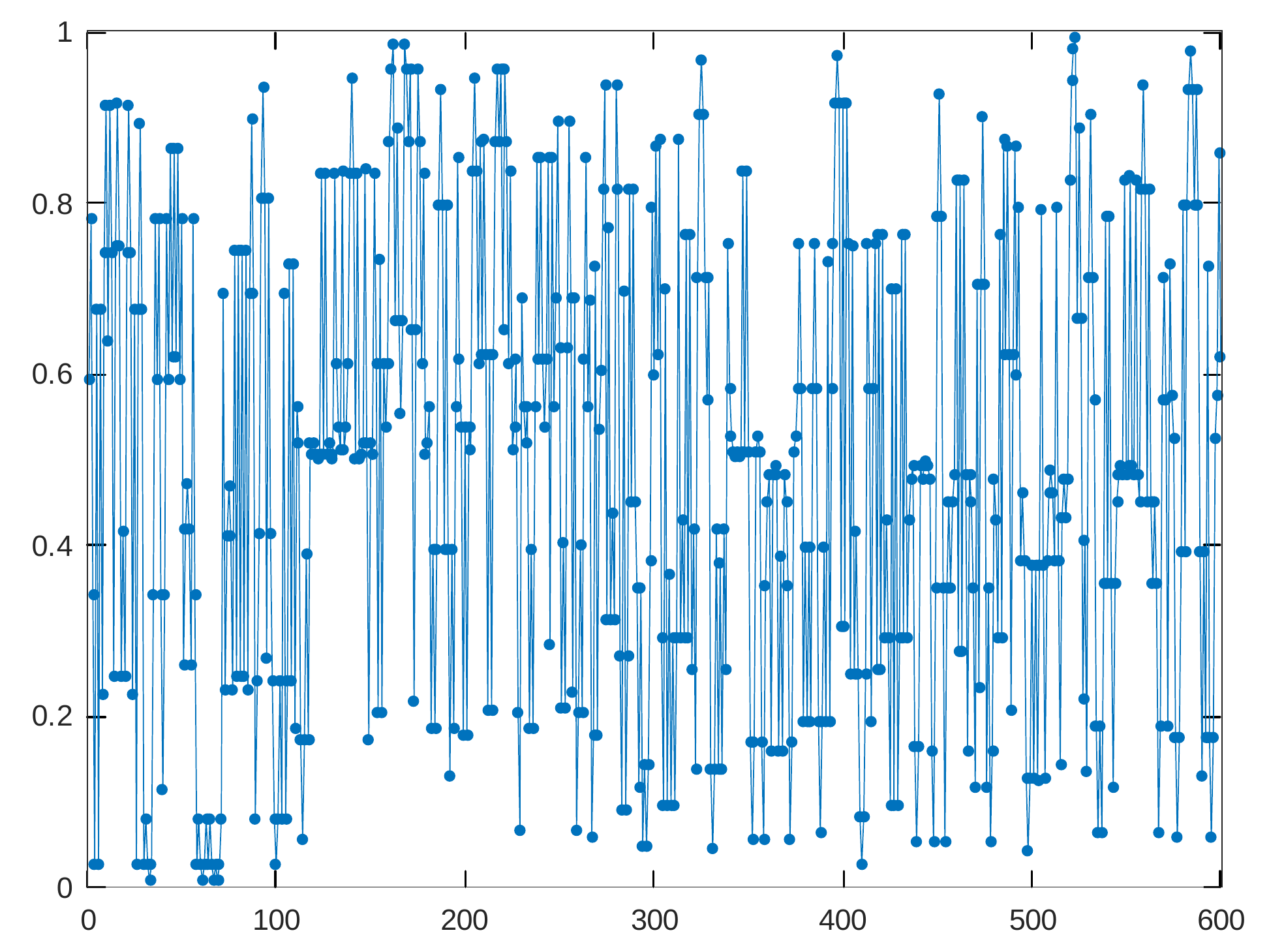}
		\hspace{1cm}
		\includegraphics[height=3.5cm]{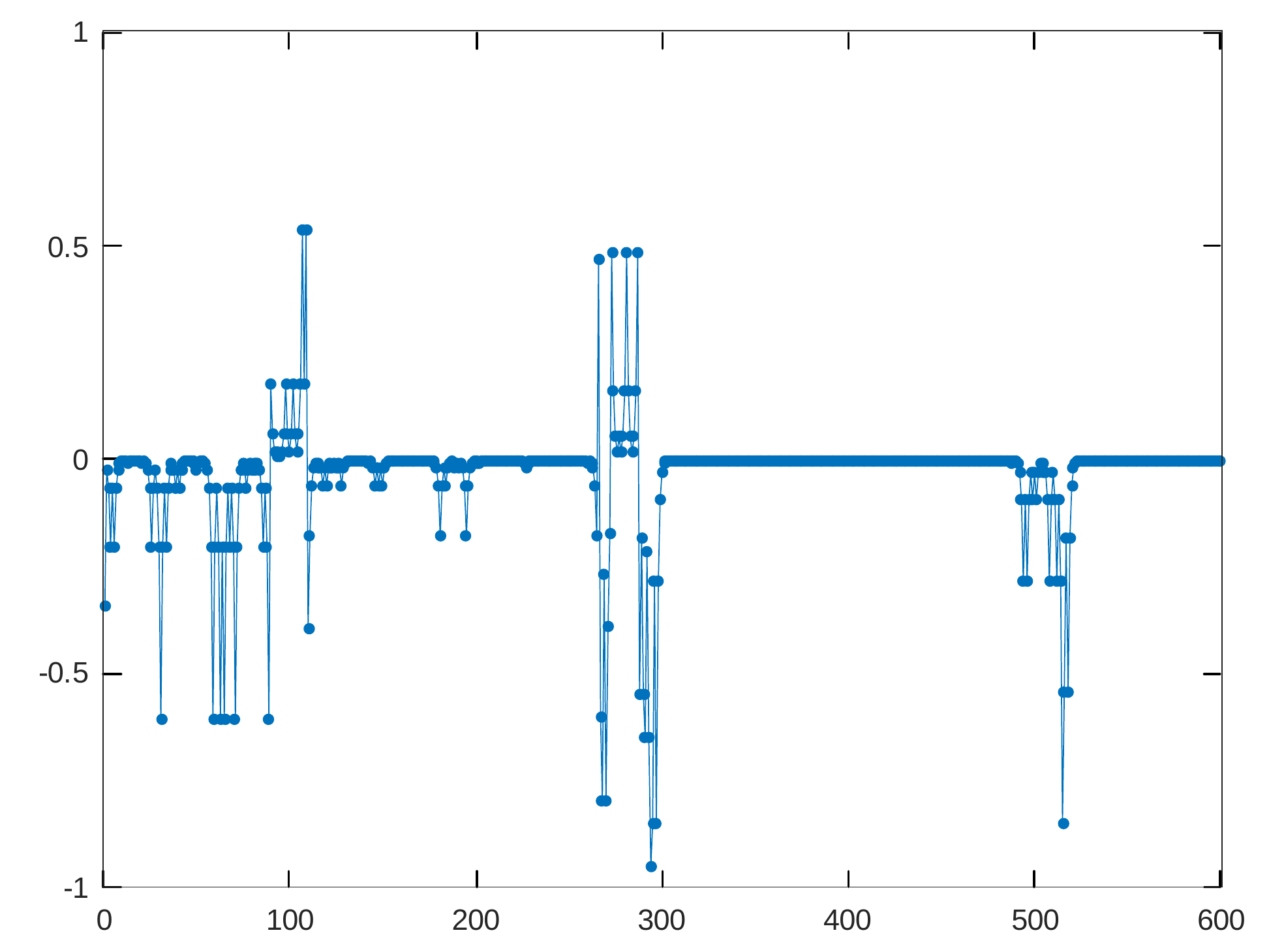}
		\caption{\label{f:intermittent} A time series of $f^n_\omega (x)$ for $(M,N)=(3,3)$. The right panel shows the signed difference with another time series with the same $\omega$.
		}
	\end{center}
\end{figure}
In Figure~\ref{f:intermittent} the panel on the right shows signed distances between two orbits with different starting points but identical $\omega$. One sees that the orbits are mostly close together with occasional bursts where the orbits diverge. We make this statement quantitative and provide proofs below. The reader is invited to compare the results with  \cite{MR3600645} on iterated functions systems of interval diffeomorphisms. The novelty in the setting here is the use of an expanding map, which enables having Lebesgue measure as stationary measure and a uniform distribution of orbit points.

\begin{theorem}\label{t:Lp=0frequency}
Consider $L_{p_0} = 0$. For every $\varepsilon>0$, for all $x,y \in [0,1)$,
\[
\lim_{n\to\infty} \frac{1}{n}	\left| \{ 0\le i < n \, ; \,  |f^i_\omega (x) - f^i_\omega (y) | < \varepsilon \}\right| =1 
\]
for $\nu$-almost all $\omega\in \Sigma$.
\end{theorem}

We  describe the strategy in words before giving the formal proof. As in the proof of Theorem~\ref{t:M>Nconvergence}, we will rely on a Borel-Cantelli type argument. The proof also uses statements on random walks on the line that are developed in Appendix~\ref{s:A}.

\vskip .2cm
Let $\varepsilon >0$ be small. The main strategy is to subdivide time into periods where we are sure that the distances $  |f^i_\omega (x) - f^i_\omega (y) |$ are small enough (smaller than $\varepsilon$) followed by periods where the distances might be too big. Then we proceed by showing that the latter periods take up only a negligible part of time. To do this we fix a word $\zeta_1\cdots\zeta_D \in \{1,\ldots,M\}^D$ with $D$ large enough so that $1/M^D < \varepsilon$. Now start with the entire interval $[0,1)$ and iterate under $f^n_\omega$ until an iterate $n_1 = n_1(\omega)$ with $n_1 > D$ and $\sigma^{n_1 - D + 1} \omega \in [\zeta_1\cdots\zeta_D]$. So the final $D$ symbols $\omega_{n_1-D+1}\cdots\omega_{n_1}$ equal $\zeta_1\cdots\zeta_D$.
Then $f^{n_1 + 1}_\omega ([0,1))$ is contained in the interval
\begin{align}\label{e:J}
	J &= f^D_\zeta ([0,1)),
\end{align}
which has length smaller than $\varepsilon$. Hence, we arrived at a time $n_1+1$ for which $ |f^{n_1+1}_\omega (x) - f^{n_1+1}_\omega (y) | < \varepsilon$. The expected stopping time, that is, the average value of $\omega \mapsto n_1(\omega)$, is finite.

Now we are interested in the number of iterates it takes until the image $f^n_\omega ([0,1))$, $n > n_1+1$, is no longer contained in an interval of size $\varepsilon$. As in the proof of Theorem~\ref{t:M>Nconvergence} the discontinuities of $f_0$ pose a difficulty here. To control possible intersections of  $f^n_\omega ([0,1)) \subset f^{n-(n_1+1)}_{\sigma^{n_1+1}\omega}(J)$ with the set $\mathcal{C}$ of critical points, we take a slightly different approach. Suppose $p$ is such that $1/ (p+1)^2 < \varepsilon < 1/p^2$. We iterate instead until the image $f^{m_1+1}_\omega ([0,1))$, $m_1 > n_1$, is no longer contained in an interval of size $1/(p + m_1 - n_1)^2$. 
Note that this criterion depends on the number of iterates $m_1$. 
This defines an iterate $m_1 = m_1(  \sigma^{n_1} \omega)$ that marks the end of a period of time where we are sure that the images $f^n_\omega ([0,1))$ are small enough. The expected stopping time, the average value of $m_1-n_1$, will be shown to be infinite.

Finally, we continue the procedure and obtain a sequence $0 = m_0  < n_1 < m_1 < n_2 < m_2 <  \cdots$ of stopping times.
Here
\begin{enumerate}
\item	
$f^{n_i+1}_\omega ([0,1))$ is contained in an interval of length $\varepsilon$. This is guaranteed by stopping after the appearance of a specific word  $\omega_{n_i-D+1}\cdots\omega_{n_i} = \zeta_1\cdots\zeta_D$, so that we find $f^{n_i+1}_\omega ([0,1)) \subset J$;

\item 
$f^{m_i+1 - n_i}_{\sigma^{n_i}\omega} (J)$ is not (more accurately, with the conditions we use it can no longer be guaranteed to be) contained in an interval of length 
$1 / (p + m_i-n_i)^2$  (which is smaller than $\varepsilon$). 
The use of interval lengths that are decreasing in $m_i-n_i$, instead of working with a fixed interval length $\varepsilon$, is done to control and be able to avoid intersections with critical points.  
\end{enumerate}	

The natural number $m_i$ is the first integer beyond $n_i$ for which this holds, and $n_i$ is the first integer beyond $m_{i-1} + D$  (we let $m_0=0$) with $\sigma^{n_i-D+1} \omega \in [\zeta_1\cdots\zeta_D]$.
As $f^{j - n_i}_{\sigma^{n_i}\omega} (J)$ contains  $f^{j}_\omega ([0,1))$, we have that
during iterates $n_i+1 \le j \le m_i$,  $f^j_\omega ([0,1))$ is contained
in an interval of length $\varepsilon$. Note that this does not mean that $f^j_\omega ([0,1))$ 
is itself an interval.
We find that the $n_i - m_{i-1}$ have finite expectation and the  $m_i - n_i$ have infinite expectation. 
This is combined to prove the occurrence of intermittency. 

\begin{proof}[Proof of Theorem~\ref{t:Lp=0frequency}]	
Fix an $\varepsilon >0$ and a word $\zeta_1\cdots\zeta_D \in \{1,\ldots,M\}^D$ with $D$ large enough so that $1/M^D < \varepsilon$. Let $J = f^D_\zeta ([0,1))$, then $\lambda(J) < \varepsilon$. Let $p\in \mathbb{N}$ satisfy $1/(p+1)^2 < \varepsilon < 1/p^2$ and let $x_J$ denote the midpoint of $J$. Recall the definition of $r$-neighborhoods $B(r)$ of the points of discontinuity $\mathcal{C}$ from \eqref{q:nbh}. We seek estimates for the stopping time
\begin{equation}\label{e:W}
W(\omega) =  \min \left\{ n>0 \; ; \; (f^n_\omega)' > \frac1{(p + n)^2\varepsilon}  \quad \textrm{or}\quad   f^n_\omega (x_J) \in B\Big(\frac1{(p + n)^2}\Big) \quad \text{with}\; \omega_{n} = 0  \right\}.
\end{equation}

To understand the conditions, note that if  $f^i_\omega (J) \cap \mathcal{C} = \emptyset$ for $0\le i < n$ and
$(f^n_\omega)' \le \frac1{(p + n)^2\varepsilon}$, then $f^n_\omega (J)$ is an interval with $\lambda( f^n_\omega (J) ) < \frac1{(p + n)^2}$.
Further, if for $0\le i < n$ we have $f^i_\omega (x_J) \not\in B\big(\frac1{(p+i)^2}\big)$ when $\omega_i=0$,  and $(f^i_\omega)' \le \frac1{(p + i)^2\varepsilon}$,
then  $f^i_\omega (J) \cap \mathcal{C} = \emptyset$ for $0\le i < n$.
Consequently, these conditions plus $(f^n_\omega)' \le \frac1{(p + n)^2\varepsilon}$ imply that $f^n_\omega (J)$ is an interval with $\lambda( f^n_\omega (J) ) < \frac1{(p + n)^2}$. 

Hence, $W$ is such that for all for $n < W(\omega)$ the set $f^n_\omega (J)$ is an interval with $\lambda( f^n_\omega (J) ) \le \frac1{(p + n)^2} < \varepsilon$.

\vskip .2cm
In the following analysis we first look at the derivatives of compositions, so at the first condition in \eqref{e:W}.
For this, consider the process for $d_n = (f^n_\omega)'$, given by $d_0=1$ and 
\[ d_{n+1} =  \begin{cases}
N d_n, &  \omega_{n+1} = 0,\\
d_n/M, & \omega_{n+1} \in \{1,\ldots,M\}. 	
\end{cases}
\]
For each $n \ge 0$, let $z_n = -\ln (d_n)$. For $z_n$ we obtain the random walk given by $z_0 = 0$ and
\[ z_{n+1} =  \begin{cases}
z_n - \ln(N), &  \omega_{n+1} = 0,\\
z_n  + \ln(M), & \omega_{n+1} \in \{1,\ldots,M\},
\end{cases}
\]
for $n \ge 0$. Recall that $L_{p_0} = 0$ means $ p_0 \ln (N) - (1-p_0) \ln (M) = 0$. The average step size for this random walk, equal to $-L_{p_0}$, is zero. The criterion $d_n > \frac1{(p + n)^2\varepsilon}$, which is the first condition appearing in the definition \eqref{e:W} of $W$,  is equivalent to $ z_n < -\ln\big(\frac1{(p + n)^2}\big) + \ln(\varepsilon)$. Therefore, we are interested in the stopping time
\[ 
W_1 (\omega) = \min \left\{ n>0 \, ; \, z_n < -\ln\Big(\frac1{(p + n)^2}\Big) + \ln(\varepsilon) \right\} = \min \left\{ n>0 \, ; \, (f^n_\omega)'   > \frac1{(p + n)^2\varepsilon}  \right\},
\]
which satisfies $W_1 \ge W$. By Lemma~\ref{l:stoppingtimeinfinite} in the appendix, the average of the stopping time $W_1$ is infinite:
\[
\int_{\Sigma}  W_1(\omega) \, d\nu(\omega) = \infty.
\]
If for each $u >0$ we set
\begin{equation}\label{e:1...1}
C_u = \{ \omega \in \Sigma \; ; \; \omega_i \in \{ 1,\ldots,M\} \text{ for } 0\le i < u\}, 
\end{equation}
then we also have $\int_{C_u}  W_1(\omega) \, d\nu(\omega) = \infty$ for any $u>0$.
This holds since the first $u$ iterates give contractions and the $\omega_i$'s are independent.

\vskip .2cm	
To study the second condition in the definition of $W$, write $\Sigma_2 = \{0,1\}^\mathbb N$ and let $\Pi : \Sigma \to \Sigma_2$ be as in Section~\ref{s:multivalued}, i.e., projecting all symbols $1,2, \ldots, M$ to 1. Consider iterates $F^n_\eta$ with $F_0$, $F_1$ as defined in \eqref{e:F0}, \eqref{e:F1}. With $\eta = \Pi (\omega)$ we have $(F^n_\eta)' = (f^n_\omega)'$. The calculated stopping time $W_1$ is thus identical for any symbol sequence in $\Pi^{-1} \{ \eta\}$ and we may write $W_1(\eta)$.

\vskip .2cm
Fix $\eta \in \Sigma_2$ and let $n = W_1(\eta)$. The multivalued map $F^i_\eta$ is built from affine graphs with slope $(F^i_\eta)'$, stacked in an equidistant fashion (see for example Figure~\ref{f:1010010}).
Let 
\[
\Xi_i =  \left\{ \omega \in \Pi^{-1} ([\eta_0 \cdots \eta_{n-1}]) \, ; \, f^i_\omega (x_J) \not \in B\Big(\frac1{(p + i)^2}\Big) \;\textrm{whenever}\; \omega_{i} = 0 \right\}
\]
and set
\[\Xi = \cap_{i=0}^{n-1} \Xi_i.\]
Since $n = W_1(\eta)$, the set $\Xi$ contains all sequences $\omega \in \Pi^{-1}([\eta_0 \cdots \eta_{n-1}])$ that satisfy both conditions from the definition of $W$ in \eqref{e:W} up to the stopping time $W_1$, so for which $W=W_1$. We next show that for certain $\eta$ this collection is large enough to conclude that $\int_\Sigma W(\omega) d\nu(\omega) = \infty$.

\vskip .2cm

From $(F^i_\eta)' = (f^i_\omega)'  \le \frac{1}{(p+i)^2 \varepsilon}$ for $0 \le i <n$, we get that there are at least $\lceil \varepsilon (p+i)^2 \rceil$ different graphs in $F^i_\eta$. Let $\xi_i$ be the number of points in $F^i_\eta (\{x_J\})$ lying in $B ({1/(p+i)^2})$
and write  $\psi_i = \xi_i /  \# F^i_\eta (\{x_J\})$ for the proportion of points from 
$F^i_\eta (\{x_J\})$ contained in $B ({1/(p+i)^2})$.
A closed  interval of length $b$ contains at most $\lceil b \, (\# F^i_\eta (\{x_J\}) +1) \rceil$ points of $F^i_\eta (\{x_J\})$.
It follows that there is a constant $K>0$, independent of $i$ and $\eta$, with
\[
\psi_i \le K / (p+i)^2.
\] 
We conclude that $\psi_i$ is summable. So there exists a $u >0$ with $\sum_{i=u}^\infty \psi_i < 1/2$. Here $u$ can be taken uniformly in $\eta$, as the above estimates are uniform in $\eta$.

\vskip .2cm

Now let $\eta \in [1^u]$ with $W_1(\eta)=n$. This implies that $W(\omega)>u$ for each $\omega \in \Pi^{-1}([\eta_0 , \ldots, \eta_{n-1}])$. From Lemma~\ref{l:feta} it follows that
\[
\psi_i =  \frac{\nu (\Pi^{-1}([\eta_0 \cdots \eta_{n-1}]) \setminus \Xi_i) }{ \nu (\Pi^{-1}([\eta_0 \cdots \eta_{n-1}])) },
\]
assuming $\omega_i=0$. Since $\eta_i =1$ for all $0 \le i < u$ and $\sum_{i=u}^\infty \psi_i < 1/2$ we then have
\begin{align}\label{e:1/2}
\frac{\nu (\Xi) }{ \nu (\Pi^{-1} ([\eta_0 \cdots\eta_{n-1}]) ) } &\ge 1/2.
\end{align}
Hence
\begin{align*}
	\int_{C_u} W(\omega)  \, d\nu(\omega) &\ge \int_{\{ \omega \in C_u \, ; \, W(\omega)=W_1(\omega) \}}  W(\omega)\, d\nu(\omega)
	\\
	&\ge  \sum_{\eta \in [1^u], W_1(\eta)=n , n \ge u}  n \,  \nu (\Pi^{-1} ([\eta_0 \cdots\eta_{n-1}]) )/2
	\\
	&= \infty.
\end{align*}
In the second estimate we used \eqref{e:1/2} which says that for each $\eta \in [1^u]$ with $S(\eta)=n$, at least half of  $\Pi^{-1} ([\eta_0 \cdots\eta_{n-1}])$ counts in the integral.
We conclude
\[
\int_\Sigma   W(\omega) \, d\nu(\omega) = \infty.
\]

\vskip .2cm	
Define the stopping time
\[
V(\omega) = \min \{ n \ge D-1 \; ; \; \sigma^{n-D+1} \omega \in [\zeta_1 \cdots \zeta_D] \},
\]
where $\zeta_1 \cdots \zeta_D$ is the word fixed at the beginning of the proof.
For $\nu$-almost every $\omega \in \Sigma$ the stopping time $V(\omega)$ is finite. So for  $\nu$-almost every $\omega \in \Sigma$
it takes a finite number of iterates $k$ before $f_\omega^k([0,1)) \subseteq  f^D_\zeta ([0,1))=J$ and thus $\lambda (f_\omega^k([0,1))) < \varepsilon$.
It is well known that the average of the stopping time $V$ is bounded:
\[
\int_\Sigma V(\omega) \, d\nu(\omega) < \infty.
\]

\vskip .2cm
Combining the knowledge on the stopping times $V$ and $W$, we get for $\nu$-almost all $\omega\in\Sigma$ an infinite sequence of stopping times $0 < n_1 < m_1 < n_2 < m_2 < \cdots$ with
	\begin{align*}
		n_i &= V(\sigma^{m_{i-1}} \omega),
		\\
		m_i &=  W(\sigma^{n_i} \omega)		
	\end{align*}		
	(where we set $m_0=0$).
		By	the strong law of large numbers, see \cite[Theorems 2.4.1 and 2.4.5]{MR3930614} (for finite and infinite expectations respectively)
we have that for $\nu$-almost all $\omega \in \Sigma$, 
	\begin{align*}
		\lim_{n\to \infty} \frac{1}{n} \sum_{i=1}^n (n_i-m_{i-1}) &< \infty,
		\\
		\lim_{n\to \infty} \frac{1}{n} \sum_{i=1}^n (m_i - n_i) &= \infty.
	\end{align*}
	This implies the theorem (see the calculation in the proof of \cite[Theorem~4]{MR2033186}). 
\end{proof}

\begin{theorem}\label{t:Lp=0frequency2}
Consider $L_{p_0} = 0$. Let $\beta$ be a small positive number. Let $x,y \in [0,1)$. Then for $\nu$-almost all $\omega \in \Sigma$, either $| f^n_\omega (x) - f^n_\omega(y)| = 0$ for some $n$ or $| f^n_\omega (x) - f^n_\omega(y)|  > \beta$ for infinitely many values of $n$. 
\end{theorem}

\begin{proof}
The proof follows from the following observation. If $x < y$ are close, then as long as $f^i_\omega$, $1\le i < n$,  is continuous on the interval $[x,y]$, we have
\[
\left| f^n_\omega (x) - f^n_\omega (y)\right| = 
(f^n_\omega)'  |x-y|.
\]
The values $z_n =  -\ln ((f^n_\omega)')$ are given by a random walk $z_0 = 0$ and
\[
z_{n+1} = \left\{ 
\begin{array}{ll}
	z_n - \ln(N), &  \omega_{n+1} = 0,
	\\
	z_n  + \ln(M), & \omega_{n+1} \in \{1,\ldots,M\},
\end{array}	
\right.
\]
for $n \ge 0$. Because $p_0 \ln (N) - (1-p_0) \ln (M) = 0$, this random walk is recurrent. So
$(f_\omega^n)' = e^{-z_n}$ takes on arbitrarily large values.
\end{proof}	

Modifying the proof of Theorem~\ref{t:M>Nergodic} allows to prove ergodicity of $\mu$ from Proposition~\ref{p:inv} also in case $L_{p_0}=0$..

\begin{theorem}\label{t:M=Nergodic}
Consider $L_{p_0} =0$. The measure $\mu = \nu \times \lambda$ is an ergodic invariant measure for $F$.
\end{theorem}

\begin{proof}
The proof follows that of Theorem~\ref{t:M>Nergodic}, replacing the statement of Theorem~\ref{t:M>Nconvergence} by the statement and arguments of Theorem~\ref{t:Lp=0frequency}. The proof of Theorem~\ref{t:Lp=0frequency} calculates stopping times to get the statement on the dynamics of the $x$-coordinate. We must incorporate the $y$-coordinate. The following observations show how this works.

\vskip .2cm
Write  $\pi_w,\pi_x,\pi_y: [0,1)^3 \to [0,1)$ for the coordinate projection to the $w$-coordinate, $x$-coordinate, and $y$-coordinate, respectively. First consider a $\zeta \in \Sigma$ with $\zeta_i \in \{1,\ldots,M\}$ for $0\le i < D$ for some large $D$. Then $J=f^D_\zeta([0,1))$ is a small interval, see \eqref{e:J}. Recall the definition of the isomorphism $h$ between the left shift $\sigma$ and the expanding interval map $L$ from \eqref{q:isomh}. Note that $\pi_y \Gamma^D (h(\zeta),x,y)$ is independent of $x$ and $H = \pi_y \Gamma^D (h(\zeta) , x,[0,1))$ is an interval of length $\lambda(H)  = (1-p_0)^D$. This is small for $D$ large.

\vskip .2cm
Next, whenever $u,v \in [0,1)^3$ with $\pi_w u = \pi_w v$ and $\pi_x \Gamma (u), \pi_x \Gamma(v)$ are close to each other, then $\pi_y \Gamma$ contracts the distance between the points with a uniform contraction factor.
So if  $|\pi_x \Gamma^i (u) - \pi_x \Gamma^i(v)|$ stays small, then also $|\pi_y \Gamma^i (u) - \pi_y \Gamma^i(v)|$ stays small.
\end{proof}	

As in Proposition~\ref{p:uniformlylambda<0} we conclude that typical orbits are uniformly distributed.
Recall from \eqref{q:2ptifs} the definition of the two-point maps $f^{(2)}_i: [0,1)^2 \to [0,1)^2$ given by 
\[
f^{(2)}_i (x,y) = (f_i (x), f_i(y)), \, 0 \le i \le M.
\]
In Corollary~\ref{c:diagonal} we established that Lebesgue measure on the diagonal $\Delta = \{ (x,x) \; ; \; x \in [0,1) \} \subset [0,1)^2$ is stationary for the iterated function system on $[0,1)^2$ generated by $f^{(2)}_i$ with probabilities $p_i$, $0\le i \le M$. In the theorem below we write $\left.\lambda\right|_A$ for two-dimensional Lebesgue measure restricted to $A$. The theorem gives, for multiplicatively dependent $M,N$, an explicit expression for an infinite stationary measure of full topological support, with a density that diverges along the diagonal $\Delta$.

\begin{theorem}\label{t:2point}
Consider $L_{p_0} = 0$ and $(M,N)$ with $N,M$ multiplicatively dependent:  $N = \kappa^{k}$ and  $M =  \kappa^\ell$.
Then the iterated function system generated by the two-point maps $f^{(2)}_i$, $0\le i \le M$, admits a $\sigma$-finite infinite absolutely continuous stationary measure.
\end{theorem}	

\begin{proof}
We look for an invariant measure $m^{(2)}$ of the form $m^{(2)} = \sum_{h=0}^\infty m_h$ with		
\begin{align} \label{e:mtobnull}
	m_h &= b_h \sum_{j=0}^{\kappa^h-1}  \kappa^h  \left.\lambda\right|_{ [j/\kappa^h,(j+1)/\kappa^h )^2},
\end{align}			
where  $b_h$ can be read as the mass assigned to $\bigcup_{j=0}^{\kappa^h-1} \left[j/\kappa^h,(j+1)/\kappa^h \right)^2$, which is the union of the squares on the diagonal of size $\frac1{\kappa^{2h}}$ determined by $\kappa$-adic neighbors $i/\kappa^{2h},(i+1)/\kappa^{2h}$. So $m_h ([0,1)^2) = b_h$ and 
for the total measure we have
\[ m^{(2)} ( [0,1)^2) = \sum_{h=1}^\infty m_h ([0,1)^2) = \sum_{h=1}^\infty b_h.\]

Consider the push-forward map of measures given by
\begin{align}\label{e:Pnull}
	\mathcal{P} m &= p_0 \left(f_0^{(2)} \right)_* m + \sum_{j=1}^M  \frac{1-p_0}{M} \left( f_j^{(2)} \right)_* m.
\end{align}
Then $\mathcal{P}$ maps $\sum_{i=0}^\infty m_i$ to  $\sum_{i=0}^\infty \hat{m}_i$,
with
\begin{align*}
	\hat{m}_0 &=  p_0  ( m_0 + \cdots + m_k ),
	\\
	\hat{m}_1 &=   p_0 m_{k+1},
	\\
	\vdots \; &=  \quad\vdots
	\\
	\hat{m}_{\ell-1} &= p_0 m_{k+\ell-1},
	\\
	\hat{m}_\ell &= (1-p_0) m_{0} + p_0 m_{k+\ell},
	\\
	\vdots \; &=  \quad\vdots
	\\
	\hat{m}_{\ell + j} &=  (1-p_0) m_{j} + p_0 m_{j+k+\ell},
	\\
	\vdots \; &=  \quad\vdots
\end{align*}	
A measure of the sought for form $\sum_{h=0}^\infty m_h$ with $m_h$ as in \eqref{e:mtobnull} is determined by the sequence of numbers $(b_i)_{i\in\mathbb{N}}$. 
Write
\begin{align*}
	\mathcal{M}  (\mathbb{N}) &= \left\{  (b_i)_{i\in\mathbb{N}} \; ; \; b_i \ge 0   \right\},
\end{align*}	
which can be identified with the set of $\sigma$-finite measures on $\mathbb{N}$.
The push-forward map $\mathcal{P}$ from \eqref{e:Pnull} induces a map 
$\mathcal{Q}: \mathcal{M} (\mathbb{N} ) \to \mathcal{M} (\mathbb{N})$. To make this explicit, suppose $\mathcal{P}$ maps $\sum_{i=0}^\infty m_i$ to  $\sum_{i=0}^\infty \hat{m}_i$, and  $(b_i)_{i\in\mathbb{N}}$ is given by \eqref{e:mtobnull}
and likewise  $(\hat{b}_i)_{i\in\mathbb{N}}$ corresponds to $\sum_{i=0}^\infty \hat{m}_i$.
Denoting $\mathbf{b} = (b_i)_{i\in\mathbb{N}}$, 
$\hat{\mathbf{b}}  =(\hat{b}_i)_{i\in\mathbb{N}}$, then
\begin{align*} 
	\mathcal{Q} (\mathbf{b}) &= \hat{\mathbf{b}}.
\end{align*}	

Identifying the union of squares $\cup_{j=0}^{\kappa^h-1} [j/\kappa^h , (j+1)/\kappa^h)^2$
with the integer $h$,  
$\mathcal{Q}$ becomes the push-forward operator associated to the random walk on $\mathbb{N}$ given by
\begin{align}\label{e:rwxnul}
	x_{n+1} = \left\{  
	\begin{array}{ll}   \max\{ 0 , x_n - \ell \}, &   \omega_{n+1} = 0, 
		\\     
		x_n + k, & \omega_{n+1} \in \{1,\ldots,M\}.
	\end{array}            
	\right.
\end{align}

As $k$ and $\ell$ are relatively prime, by B\'ezout's identity there are integers $\alpha,\beta$ with $\alpha k + \beta \ell =1$. 
Noting this, it follows from $L_{p_0}=0$ that \eqref{e:rwxnul} is recurrent.
It is in fact null-recurrent, and not positively recurrent,  since $L_{p_0}=0$  implies that the expected return time to a site is infinite  (see \cite[Section~2.3]{MR2255511}).
By \cite{MR60757} there is a unique infinite stationary measure for \eqref{e:rwxnul}.
This gives the fixed point $\mathcal{Q} (\mathbf{b}) = \mathbf{b}$ with the required
property that $\sum_{h=0}^\infty b_h = \infty$.
\end{proof}

\subsection{$L_{p_0} > 0$ (Divergence)}

A goal of this section is to explain the outcome of numerical experiments such as depicted in Figure~\ref{f:divergence}. From a dynamics point of view this is done in Theorem~\ref{t:divergence} below, which follows other results on invariant measures. 

\begin{figure}[!ht]
	\begin{center}
		\includegraphics[height=3.5cm]{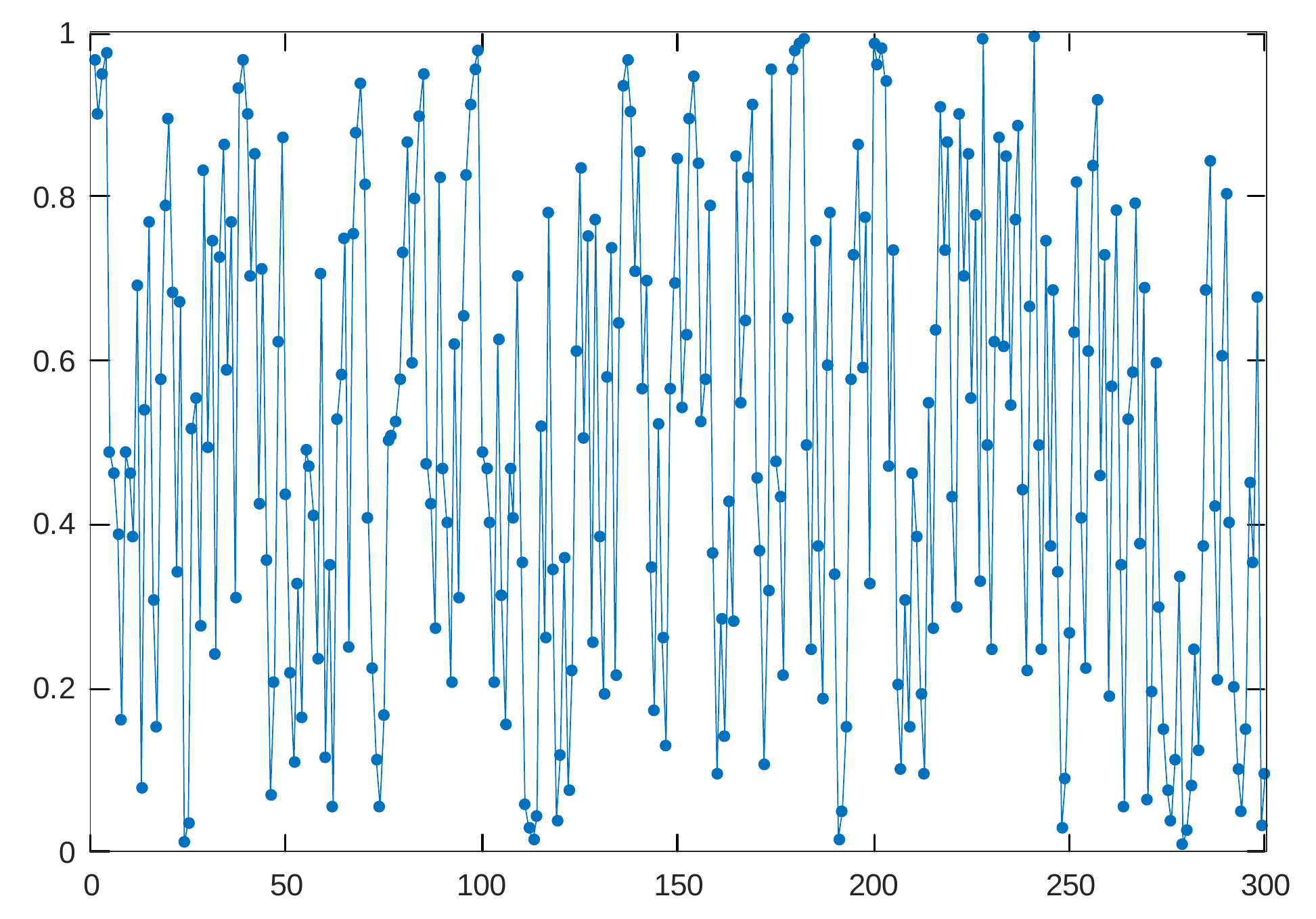}
		\hspace{1cm}
		\includegraphics[height=3.5cm]{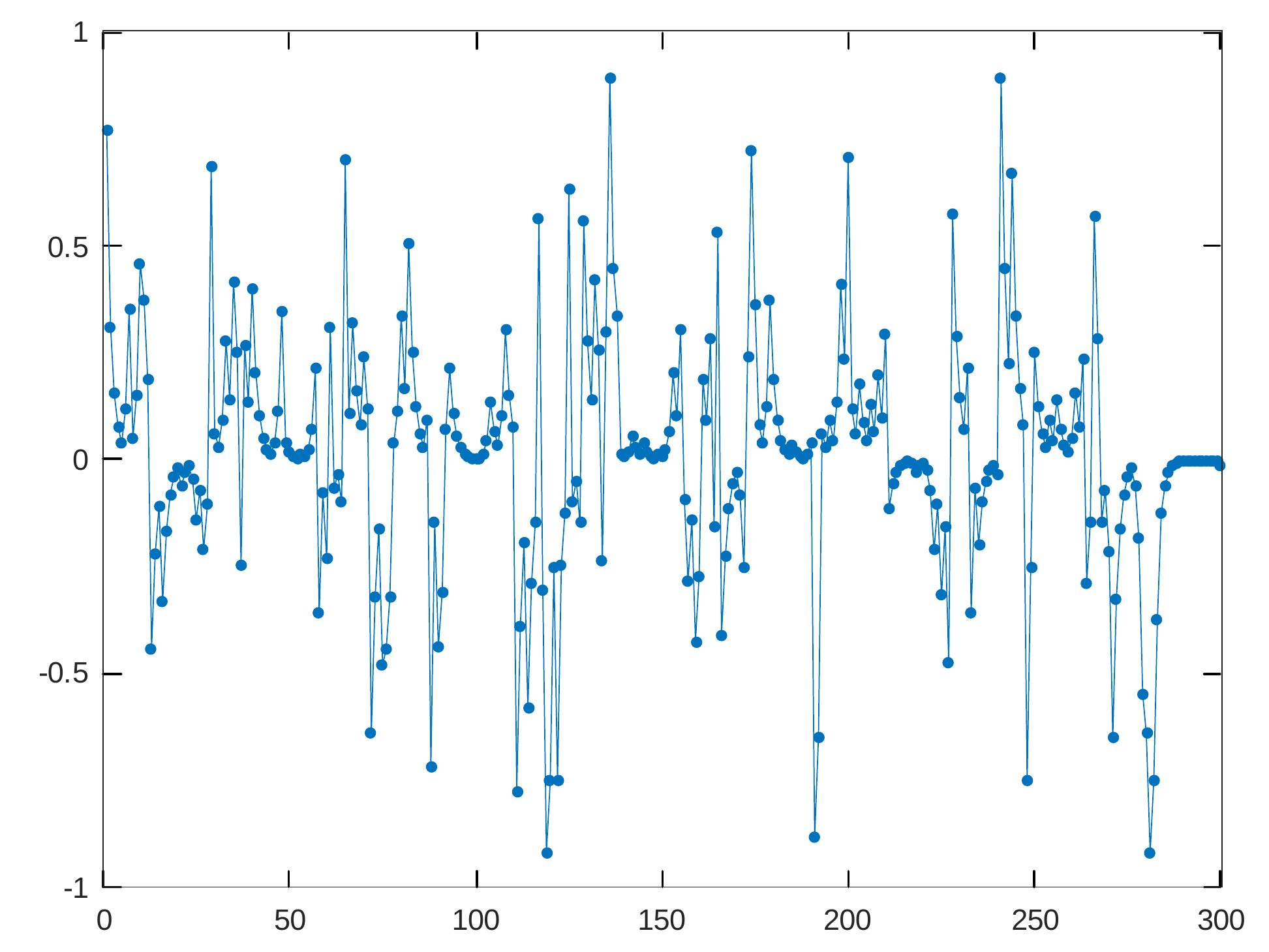}
		\caption{\label{f:divergence} A time series of $f^n_\omega (x)$ for $p_0 = 1/2$  and $(M,N)=(2,3)$. The right panel shows a signed difference with another time series with the same $\omega$.
		}
	\end{center}	
\end{figure}

First we establish the ergodicity of the product measure $\mu= \nu \times \lambda$, which proceeds by connecting to Theorem~\ref{t:M>Nergodic}. 

\begin{theorem}\label{t:M<Nergodic}
	Consider $L_{p_0}>0$. The measure $\mu$ is an ergodic invariant measure for $F$.
\end{theorem}	

\begin{proof}
	Recall the definition of the invertible map $\Gamma : [0,1)^3 \to [0,1)^3$ from \eqref{e:Gamma} with inverse
	\[ \Gamma^{-1} (w,x,y) = 
	\displaystyle{   \begin{cases}
			\left( p_0 w  , \frac{x+j}{N},  \frac{Ny}{p_0} -j \right), &  \frac{jp_0}{N}  \le y <  \frac{(j+1)p_0}{N},  \, 0\le j < N,\\
			\left(\frac{(1-p_0)(w+i)}{M}, M x -i ,   \frac{y-p_0}{1-p_0}  \right), & \begin{array}{@{}c@{}} p_0 \le y < 1, \\ 
				\frac{i}{M} \le x < \frac{i+1}{M},  \, 0\le i < M.
			\end{array}	
		\end{cases}
	}\]
	As in the proof of Theorem~\ref{t:M>Nergodic} one proves that three-dimensional Lebesgue measure is ergodic for $\Gamma^{-1}$, with the difference that now the map is expanding in the direction of $y$ and contracting in the direction of $w$. Lebesgue measure is therefore also ergodic for $\Gamma$. Reasoning as for Theorem~\ref{t:M>Nergodic}, it follows that $\mu$ is ergodic for $F$.
\end{proof}	

As a corollary we have that typical orbits are uniformly distributed, see Proposition~\ref{p:uniformlylambda<0}.

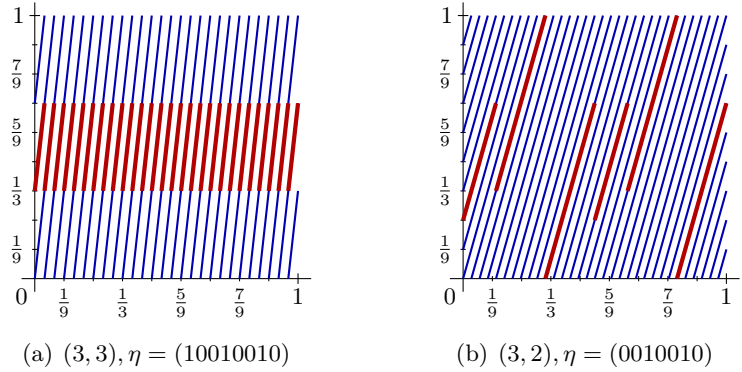
\begin{figure}[!ht]
	\begin{center}
		\subfigure[$(3,3),\eta = (10010010)$]{
			\begin{tikzpicture}[scale=3.5]
				\draw(-.05,0)node[below]{\footnotesize $0$}--(1/9,0)node[below]{\footnotesize $\frac19$}--(3/9,0)node[below]{\footnotesize $\frac13$}--(5/9,0)node[below]{\footnotesize $\frac59$}--(7/9,0)node[below]{\footnotesize $\frac79$}--(1,0)node[below]{\footnotesize $1$}--(1.05,0);
				\draw(0,-.05)--(0,1/9)node[left]{\footnotesize $\frac19$}--(0,1/3)node[left]{\footnotesize $\frac13$}--(0,5/9)node[left]{\footnotesize $\frac59$}--(0,7/9)node[left]{\footnotesize $\frac79$}--(0,1)node[left]{\footnotesize 1}--(0,1.05);
				\draw(1,-.02)--(1,.02)(1/9,-.01)--(1/9,.01)(2/9,-.01)--(2/9,.01)(3/9,-.01)--(1/3,.01)(4/9,-.01)--(4/9,.01)(5/9,-.01)--(5/9,.01)(6/9,-.01)--(6/9,.01)(7/9,-.01)--(7/9,.01)(8/9,-.01)--(8/9,.01);
				\draw(-.02,1)--(.02,1)(-.01,1/9)--(.01,1/9)(-.01,2/9)--(.01,2/9)(-.01,1/3)--(.01,1/3)(-.01,4/9)--(.01,4/9)(-.01,5/9)--(.01,5/9)(-.01,6/9)--(.01,6/9)(-.01,7/9)--(.01,7/9)(-.01,8/9)--(.01,8/9);
				\draw[thick, blue!70!black] (0,0)--(1/9,1)(1/27,0)--(4/27,1)(2/27,0)--(5/27,1)(3/27,0)--(6/27,1)(4/27,0)--(7/27,1)(5/27,0)--(8/27,1)(6/27,0)--(9/27,1)(7/27,0)--(10/27,1)(8/27,0)--(11/27,1)(9/27,0)--(12/27,1)(10/27,0)--(13/27,1)(11/27,0)--(14/27,1)(12/27,0)--(15/27,1)(13/27,0)--(16/27,1)(14/27,0)--(17/27,1)(15/27,0)--(18/27,1)(16/27,0)--(19/27,1)(17/27,0)--(20/27,1)(18/27,0)--(21/27,1)(19/27,0)--(22/27,1)(20/27,0)--(23/27,1)(21/27,0)--(24/27,1)(22/27,0)--(25/27,1)(23/27,0)--(26/27,1)(24/27,0)--(27/27,1)(0,1/3)--(2/27,1)(0,2/3)--(1/27,1)(25/27,0)--(1,2/3)(26/27,0)--(1,1/3);
				\draw[ultra thick, red!70!black] (0,1/3)--(1/27,2/3)(1/27,1/3)--(2/27,2/3)(2/27,1/3)--(3/27,2/3)(3/27,1/3)--(4/27,2/3)(4/27,1/3)--(5/27,2/3)(5/27,1/3)--(6/27,2/3)(6/27,1/3)--(7/27,2/3)(7/27,1/3)--(8/27,2/3)(8/27,1/3)--(9/27,2/3)(9/27,1/3)--(10/27,2/3)(10/27,1/3)--(11/27,2/3)(11/27,1/3)--(12/27,2/3)(12/27,1/3)--(13/27,2/3)(13/27,1/3)--(14/27,2/3)(14/27,1/3)--(15/27,2/3)(15/27,1/3)--(16/27,2/3)(16/27,1/3)--(17/27,2/3)(17/27,1/3)--(18/27,2/3)(18/27,1/3)--(19/27,2/3)(19/27,1/3)--(20/27,2/3)(20/27,1/3)--(21/27,2/3)(21/27,1/3)--(22/27,2/3)(22/27,1/3)--(23/27,2/3)(23/27,1/3)--(24/27,2/3)(24/27,1/3)--(25/27,2/3)(25/27,1/3)--(26/27,2/3)(26/27,1/3)--(27/27,2/3);
			\end{tikzpicture}
		}
		\hspace{1cm}
		\subfigure[$(3,2),\eta = (0010010)$]{
			\begin{tikzpicture}[scale=3.5]
				\draw(-.05,0)node[below]{\footnotesize $0$}--(1/9,0)node[below]{\footnotesize $\frac19$}--(3/9,0)node[below]{\footnotesize $\frac13$}--(5/9,0)node[below]{\footnotesize $\frac59$}--(7/9,0)node[below]{\footnotesize $\frac79$}--(1,0)node[below]{\footnotesize $1$}--(1.05,0);
				\draw(0,-.05)--(0,1/9)node[left]{\footnotesize $\frac19$}--(0,1/3)node[left]{\footnotesize $\frac13$}--(0,5/9)node[left]{\footnotesize $\frac59$}--(0,7/9)node[left]{\footnotesize $\frac79$}--(0,1)node[left]{\footnotesize 1}--(0,1.05);
				\draw(1,-.02)--(1,.02)(1/9,-.01)--(1/9,.01)(2/9,-.01)--(2/9,.01)(3/9,-.01)--(1/3,.01)(4/9,-.01)--(4/9,.01)(5/9,-.01)--(5/9,.01)(6/9,-.01)--(6/9,.01)(7/9,-.01)--(7/9,.01)(8/9,-.01)--(8/9,.01);
				\draw(-.02,1)--(.02,1)(-.01,1/9)--(.01,1/9)(-.01,2/9)--(.01,2/9)(-.01,1/3)--(.01,1/3)(-.01,4/9)--(.01,4/9)(-.01,5/9)--(.01,5/9)(-.01,6/9)--(.01,6/9)(-.01,7/9)--(.01,7/9)(-.01,8/9)--(.01,8/9);
				\draw[thick, blue!70!black] (0,0)--(9/32,1)(1/32,0)--(10/32,1)(2/32,0)--(11/32,1)(3/32,0)--(12/32,1)(4/32,0)--(13/32,1)(5/32,0)--(14/32,1)(6/32,0)--(15/32,1)(7/32,0)--(16/32,1)(8/32,0)--(17/32,1)(9/32,0)--(18/32,1)(10/32,0)--(19/32,1)(11/32,0)--(20/32,1)(12/32,0)--(21/32,1)(13/32,0)--(22/32,1)(14/32,0)--(23/32,1)(15/32,0)--(24/32,1)(16/32,0)--(25/32,1)(17/32,0)--(26/32,1)(18/32,0)--(27/32,1)(19/32,0)--(28/32,1)(20/32,0)--(29/32,1)(21/32,0)--(30/32,1)(22/32,0)--(31/32,1)(23/32,0)--(32/32,1)(0,1/9)--(8/32,1)(0,2/9)--(7/32,1)(0,1/3)--(6/32,1)(0,4/9)--(5/32,1)(0,5/9)--(4/32,1)(0,6/9)--(3/32,1)(0,7/9)--(2/32,1)(0,8/9)--(1/32,1)(24/32,0)--(1,8/9)(25/32,0)--(1,7/9)(26/32,0)--(1,6/9)(27/32,0)--(1,5/9)(28/32,0)--(1,4/9)(29/32,0)--(1,3/9)(30/32,0)--(1,2/9)(31/32,0)--(1,1/9);
				\draw[ultra thick, red!70!black] (0,2/9)--(4/32,2/3)(4/32,1/3)--(10/32,1)(10/32,0)--(16/32,2/3)(16/32,2/9)--(20/32,2/3)(20/32,1/3)--(26/32,1)(26/32,0)--(1,2/3);
			\end{tikzpicture}
		}
		\caption{\label{f:0010010} 
			Left picture: a plot of the graphs of $F^{8}_\eta$  for $(M,N) = (3,3)$ and $\eta = (10010010)$. The red graph is the graph of $f^{8}_\omega$ for $\omega = (20020020)$. Right picture: a plot of the graphs of $F^{7}_\eta$  for $(M,N) = (3,2)$ and $\eta = (0010010)$. The red graph is the graph of $f^{7}_\omega$ for $\omega = (0020030)$.}
	\end{center}	
\end{figure}

\vskip .2cm
Before we formulate and prove the result that provides the last part of Theorem~\ref{t:main}, we first focus on stationary measures for the iterated function system generated by the two-point maps. We will give two results. Firstly, for multiplicatively dependent $M, N$ we provide an explicit expression for a stationary measure $m^{(2)}$ that is absolutely continuous with respect to Lebesgue and has full topological support and that, contrary to Theorem~\ref{t:2point}, is a finite measure. Figure~\ref{f:density2} shows a numerical approximation of the density function of this stationary measure for $p_0=1/2$ and $(M,N) = (3,9)$. Secondly, we prove the existence of such a measure for all pairs $(M,N)$ without identifying an explicit expression. The proof of the second result will again run into the difficulties caused by the discontinuities of $f_0$, see Figure~\ref{f:0010010} that includes a plot of the graphs of $F^{7}_\eta$ with $F_0,F_1$ introduced in \eqref{e:F0}, \eqref{e:F1}.

\begin{figure}[!ht]
\begin{center}
		\includegraphics[height=6cm,width=6cm]{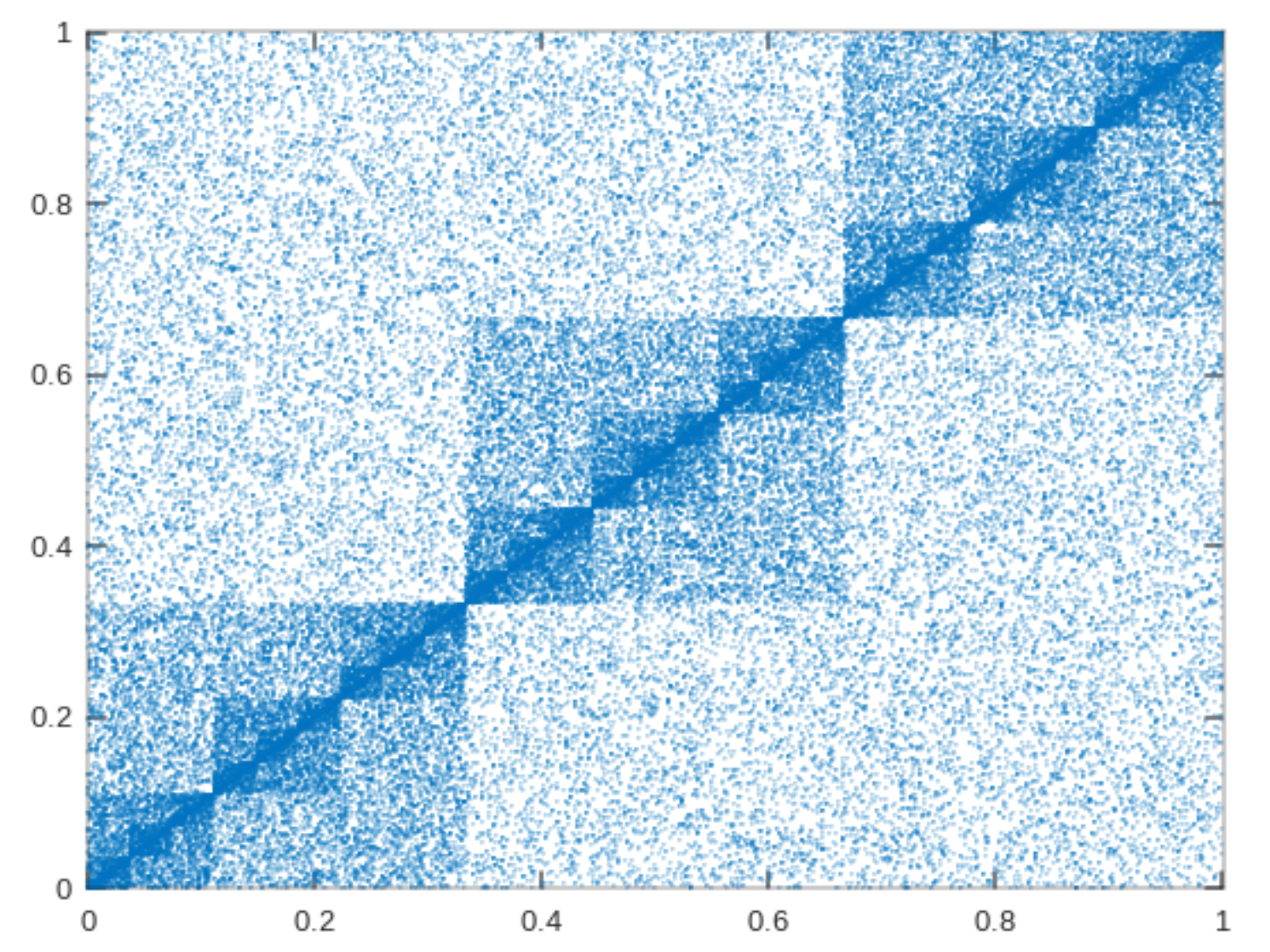}
		\caption{\label{f:density2} A plot of the stationary distribution for the two-point motion by a numerically computed histogram of an orbit, for $p_0=1/2$ and $(M,N) = (3,9)$.
		}
	\end{center}	
\end{figure}

If $N = \kappa^{k}$ and  $M =  \kappa^\ell$, then $L_{p_0} > 0$ reads
$p_0  k \ln (\kappa) - (1-p_0) \ell \ln(\kappa) > 0$. We will use that this implies
\begin{equation}\label{q:conspos}
\frac{\ell}{k+\ell}  < p_0 < 1.
\end{equation}

\begin{theorem}\label{t:stationary-multdep}
	Consider $L_{p_0} >0$ and $(M,N)$ with $N,M$ multiplicatively dependent:  $N = \kappa^{k}$ and  $M =  \kappa^\ell$.
	Then the  iterated function system generated by the two-point maps $f^{(2)}_i$, $0\le i \le M$, admits an absolutely continuous  stationary measure $m^{(2)}$
	of the form
	\begin{align*}
		m^{(2)} &= \sum_{h=0}^\infty  b_h  \sum_{j=0}^{\kappa^h-1}  \kappa^h  \left.\lambda\right|_{ [j/\kappa^h,(j+1)/\kappa^h )^2},
	\end{align*}
	with $b_h$ satisfying the recurrence equation \[b_{j+k+\ell} =  \frac{1}{p_0} b_{j+\ell} - \frac{1-p_0}{p_0} b_{j}, \quad j \ge 0,\] 
	and a suitable initial condition on $b_0,\ldots,b_{\ell-1}$.
	
	Moreover, with $\nu_1$ being the unique real solution in $(0,1)$  to
	$p_0 \zeta^{k+\ell} - \zeta + 1-p_0 = 0$,  
	\[\lim_{h\to \infty} b_h / \nu_1^h\] exists and is a positive number.
\end{theorem}

\begin{proof}
We take the setup of the proof of Theorem~\ref{t:2point}, which we  briefly repeat. We look for an invariant measure $m^{(2)}$ of the form $m^{(2)} = \sum_{h=0}^\infty m_h$ with		
\begin{align} \label{e:mtob}
	m_h &= b_h \sum_{j=0}^{\kappa^h-1}  \kappa^h  \left.\lambda\right|_{ [j/\kappa^h,(j+1)/\kappa^h )^2}.
\end{align}			
Consider the push-forward map $\mathcal P$ of measures from \eqref{e:Pnull}.
A measure of the sought for form $\sum_{h=0}^\infty m_h$ with $m_h$ as in \eqref{e:mtob} is determined by the sequence of numbers $(b_i)_{i\in\mathbb{N}}$. 
The push-forward map $\mathcal{P}$ from \eqref{e:Pnull} induces a map 
$\mathcal{Q}: \mathcal{M}_1 (\mathbb{N} ) \to \mathcal{M}_1 (\mathbb{N})$, where
\[
\mathcal{M}_1  (\mathbb{N}) = \left\{  (b_i)_{i\in\mathbb{N}} \; ; \; b_i \ge 0,   \sum_{i=0}^\infty  b_i = 1  \right\}.
\]

As noted in the proof of Theorem~\ref{t:2point},  $\mathcal{Q}$ is the push-forward operator associated to the random walk on $\mathbb{N}$ given by
\begin{align} \label{e:rwx}
	x_{n+1} = \left\{  
	\begin{array}{ll}   \max\{ 0 , x_n - k \}, &   \omega_{n+1} = 0, 
		\\     
		x_n + \ell, & \omega_{n+1} \in \{1,\ldots,M\}.
	\end{array}            
	\right.
\end{align}
As $L_{p_0} > 0$, this is a positive recurrent random walk. Hence $\mathcal{Q}$ admits a fixed point in $\mathcal{M}_1 (\mathbb{N})$.
	
Having established the existence of a fixed point $\mathcal{Q}(\mathbf{b}) = \mathbf{b}$, 	
we continue with calculations that will result in expressions for $b_h$.
The stationary measure $m^{(2)}$ satisfies $\mathcal P m^{(2)} = m^{(2)}$. For the coefficients $b_h$, $h \ge 0$, this gives equations
\begin{align*}
	b_k &= \frac{1-p_0}{p_0} b_0 - b_1- \cdots - b_{k-1},
	\\
	b_{k+1} &=   \frac{1}{p_0} b_{1},
	\\
	\vdots \; &=  \quad\vdots
	\\
	b_{k+\ell-1} &=   \frac{1}{p_0} b_{\ell-1},
	\\
	b_{k+\ell} &= \frac{1}{p_0}  b_\ell - \frac{1-p_0}{p_0} b_0,
	\\
	\vdots \; &=  \quad\vdots
	\\
	b_{j+k+\ell} &=  \frac{1}{p_0} b_{j+\ell} - \frac{1-p_0}{p_0} b_{j},
	\\
	\vdots  \; &= \quad \vdots	
\end{align*}	
The recurrence equation $b_{j+k+\ell} =  \frac{1}{p_0} b_{j+\ell} - \frac{1-p_0}{p_0} b_{j}$ that appears here,  is equivalent to the linear system
\begin{align*}
	\begin{pmatrix} b_{h+1}  \\ b_{h+2} \\ \vdots \\  b_{h+\ell+1} \\ \vdots \\ b_{h+k+\ell-1} \\ b_{h+k+\ell} \end{pmatrix} &= 
	\begin{pmatrix} 
			0      & 1      & 0            & \cdots&0      & \cdots & 0\\
			0      & 0      & 1            & \cdots&0       & \cdots & 0\\
			\vdots & \vdots & \vdots  & \ddots&\vdots      & \ddots & 0\\
			0 &       0 & 0  & \cdots  &1      & \cdots & 0\\
			\vdots & \vdots & \vdots  & \ddots&\vdots      & \ddots & 0\\
			0      & 0      & 0          & \cdots&0       & \cdots & 1\\
			-\frac{1-p_0}{p_0}& 0& 0      & \cdots     &   \frac{1}{p_0} & \cdots & 0
	\end{pmatrix}
	\begin{pmatrix} b_{h}  \\ b_{h+1}  \\ \vdots \\ b_{h+\ell} \\ \vdots \\ b_{h+k+\ell-2} \\ b_{h+k+\ell-1} 
	\end{pmatrix},
\end{align*}
$h \ge 0$.  Denote the above matrix by $A$. Its characteristic equation is $p_0 \zeta^{\ell+k} - \zeta^{\ell} + 1-p_0 = 0$, so that
\begin{align}\label{e:charzeta}
	p_0 (\zeta^{k+\ell}  -1) &= \zeta^{\ell}-1.
\end{align}
	
We claim that for any $p_0 < 1$ the zeros of the characteristic equation, thus the eigenvalues of $A$,  are as follows.
\begin{enumerate} 
	\item $A$ has a single eigenvalue at $1$, there are $k-1$ eigenvalues outside the unit circle and there are $\ell$ eigenvalues inside the unit circle.
Write them as
	\[
\eta_0 = 1, \qquad
\eta_1,\ldots, \eta_{k-1} \in \{ z\in\mathbb{C} \; ; \; |z| > 1\},
\qquad
\nu_1,\ldots,\nu_\ell  \in \{ z\in\mathbb{C} \; ; \; |z| < 1\};
\]
 \item
The eigenvalue with largest modulus among $\{\nu_1,\ldots,\nu_\ell\}$ is single, real and positive.
Let $\nu_1 \in (0,1)$ be this eigenvalue.
\end{enumerate}
	
To prove the statements on the eigenvalues, 
write $S_r (a) = \{ z \in \mathbb{C} \; ; \; |z-a| = r\}$ for the circle in the complex plane of radius $r$ and center $a$.
Consider $\zeta \in \mathbb{C}$ with  $|\zeta|=r$. Then  
\[ 
p_0(\zeta^{k+\ell}-1) \in  S_{p_0 r^{k+\ell}} (-p_0), \qquad 
\zeta^{\ell}-1 \in S_{r^\ell}(-1).
\]
Solutions to \eqref{e:charzeta} with $|\zeta|=r$ can occur only if these two circles
$ S_{p_0 r^{k+\ell}} (-p_0)$ and $ S_{r^\ell}(-1)$
intersect, so we consider their mutual position.

First consider the situation that $p_0=1$. For $r=1$ the resulting circles $S_{p_0} (-p_0)$ and $S_1 (-1)$ are identical and for any other value of $r$ the circles do not intersect. The solutions to \eqref{e:charzeta} are therefore given by 
$\{e^{2 \pi (j/k) i}, 0 \le  j < k \}$. Together with the solution at $0$, with multiplicity $\ell$, these are the solutions to \eqref{e:charzeta} with $p_0=1$.

Now assume $\frac{\ell}{k+\ell} < p_0 < 1$, see \eqref{q:conspos}. For $r=1$, $S_{p_0} (-p_0)$ and $S_1 (-1)$ intersect only at the origin.
As $k$ and $\ell$ are relatively prime, and therefore also $k + \ell$ and $\ell$
are relatively prime, $\zeta=1$ is the unique solution with $|\zeta|=1$ 
to \eqref{e:charzeta}.

Now consider $\frac{\ell}{k+\ell} < p_0 < 1$ with $r<1$.
Write $h_0 (r) =   p_0 (r^{k+\ell}-1)$ and $h_1 (r) =  r^\ell-1$.
Then $ -1 = h_1 (0) < h_0 (0) = -p_0$, while $ 0 = h_1(1)=h_0(1)$.
The function $h_0' (r) / h_1'(r) = p_0 \frac{k+\ell}{\ell}  r^k$ is monotone increasing on $[0,1]$,
from $0$ at $r=0$ to $ p_0 \frac{k+\ell}{\ell} > 1$ at $r=1$.
There is therefore a unique solution $\hat{r} = \hat{r}(p_0)$ in $(0,1)$ to $h_0(r) = h_1(r)$, see Figure~\ref{f:circles} for an illustration. Moreover, $\hat{r} \to 0$ as $p_0\to 1$ and $\hat{r} \to 1$ as $p_0 \to \frac{\ell}{k+\ell}$. 

This means that $ S_{p_0 {\hat{r}}^{k+\ell}} (-p_0)$ is tangent to 
$S_{{\hat{r}}^\ell}(-1)$ (at $\hat{r}^\ell-1$).  For $\hat{r} < r < 1$,  
$S_{p_0 r^{k+\ell}} (-p_0) \cap  S_{r^\ell}(-1) = \emptyset$.
We conclude that the real solution $\nu_1 = \hat{r}$ is the solution to \eqref{e:charzeta} of largest modulus of solutions inside the unit disc. Again as $k$ and $\ell$ are relatively prime, it is an isolated solution, and other solutions inside the unit disc have smaller modulus.

Since eigenvalues depend continuously on $p_0$, and $1$ is an isolated eigenvalue and the only eigenvalue on the unit circle for $\frac{\ell}{k+\ell} < p_0 < 1$,
eigenvalues can not cross the unit circle when varying $p_0$.
So for any $\frac{\ell}{k+\ell} < p_0 < 1$,  there are $\ell$ eigenvalues inside
the circle of radius $\hat{r}(p_0)$, there is an isolated eigenvalue at $1$, and the remaining
$k-1$ eigenvalues  eigenvalues  lie outside the unit circle.
This concludes the proof of the statements on the solutions to \eqref{e:charzeta}.
	
A solution to the equations for $b_h$ is determined by an 
initial vector $b_0,\ldots,b_{k+\ell-1}$.  For higher indices $h$, $b_h$ is given by the recurrence equation. 
For a solution with $b_h$ converging to $0$ as $h \to \infty$, we need the  initial vector to be contained in the sum of the (generalized) eigenspaces corresponding to the
contracting eigenvalues $\nu_1,\ldots,\nu_\ell$.			
Recall that we already know there is such a solution to the equations.

The initial condition can not be contained in the range of $(A - \nu_1)$ (thus must have a component  in the direction of the eigenvector corresponding to the unique real positive eigenvalue $\nu_1$, when decomposing in a basis of generalized eigenvectors),
since otherwise $b_h$ can not be positive for all $h$.  As $\nu_1$ is the eigenvalue of largest modulus of all eigenvalues inside the unit circle, this implies
$b_h / \nu_1^h$ converges to a positive value as $h \to \infty$.
\end{proof}

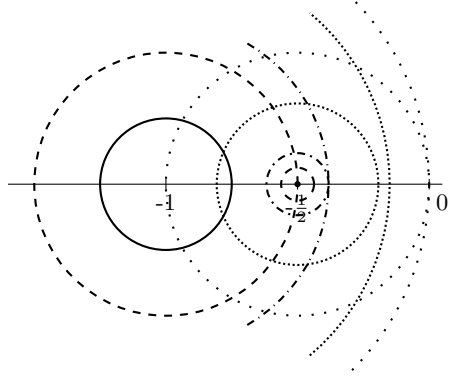
\begin{figure}[h]
\begin{center}
\begin{tikzpicture}[scale=3.5]
\draw(.05,0)node[below]{\footnotesize $0$}--(-1/2,0)node[below]{\footnotesize -$\frac12$}--(-1,0)node[below]{\footnotesize -$1$}--(-1.6,0);
\draw(0,-.02)--(0,.02);
\draw(-1,-.02)--(-1,.02)(-.5,-.02)--(-.5,.02);
\draw[thick] (-1,0) circle (.25cm);
\draw[thick, dashed] (-.5,0) circle (.0625cm);

\draw[thick, dashed] (-1,0) circle (.5cm);
\draw[thick] (-.5,0) circle (.0078cm);

\draw[thick, dash dot] (-.5,0) circle (.118cm);

\draw[thick, densely dotted] (-.5,0) circle (.307cm);
\draw[thick, loosely dotted] (-.5,0) circle (.5cm);

\draw [dash dot,thick, domain=-60:60] plot ({-1+.618*cos(\x)}, {.618*sin(\x)});
\draw [densely dotted,thick, domain=-50:50] plot ({-1+.85*cos(\x)}, {.85*sin(\x)});
\draw [loosely dotted,thick, domain=-45:45] plot ({-1+cos(\x)}, {sin(\x)});
\end{tikzpicture}
\label{f:circles}
\caption{The circles $S_{p_0 r^{k+l}}(-p_0)$ and $S_{r^l}(-1)$ for the values $p_0 = -\frac12$, $N=4$, $M=2$ and $r=\frac14$ (solid), $r=\frac12$ (dashed), $r=\hat r = \frac{\sqrt 5-1}{2}$ (dash dotted), $r= \frac{85}{100}$ (densely dotted) and $r=1$ (loosely dotted).}
\end{center}
\end{figure}

\begin{remark}
In the above proof we concluded the existence of a fixed point of 
$\mathcal{Q}$  in $\mathcal{M}_1 (\mathbb{N})$ from positive recurrence of \eqref{e:rwx}.
Here we connect to a different approach.
Consider the diffeomorphism $h : \mathbb{R} \to (0,1)$ given by
$
h(x) = \frac{e^x}{1+e^x}.
$
The random walk \eqref{e:rwx} considered on $\mathbb{R}$ is topologically conjugate through $h$ with the iterated function system
\begin{align}\label{e:rwy}
y_{n+1} = \left\{  \begin{array}{ll}   \max\left\{ 0 , \frac{ e^{-\ell } y_n}{1 + (e^{-\ell}- 1 ) y_n} \right\}, &   \omega_{n+1} = 0, 
		\\    \frac{ e^{k} y_n}{1 + (e^{k}- 1 ) y_n}   , & \omega_{n+1}=1
	\end{array}            
	\right.
\end{align}
on $(0,1)$.  By continuous extension we have $1$ as a common fixed point for the two maps generating \eqref{e:rwy}.
Moreover, the iterated function systems has a positive Lyapunov exponent $L_{p_0}$ at $1$. 
Following the reasoning of \cite[Lemma~3.2]{MR3600645} (see also \cite[Proposition~4.1]{MR3567274};  
it amounts to following a Krylov-Bogolyubov procedure 
on a suitable closed class of measures),
the iterated function system \eqref{e:rwy}
admits a stationary measure supported on $h(\mathbb{N})$. Hence $\mathcal{Q}$ admits a fixed point in $\mathcal{M}_1 (\mathbb{N})$.
\end{remark}	

\begin{example}
	We work out the general result of Theorem~\ref{t:stationary-multdep} in two special cases.
	\begin{enumerate}
		
		\item
		Pairs  $(M,N)$ with $N = M^k$ correspond to $\kappa = M$ and $\ell=1$. 
		By Theorem~\ref{t:stationary-multdep},
		the iterated function system generated by the two-point maps $f^{(2)}_i$, $0\le i \le M$, admits an absolutely continuous  stationary measure $m^{(2)}$
		of the form
		\[
		m^{(2)} = \sum_{h=0}^\infty  b_h  \sum_{j=0}^{M^h-1}  M^h  \left.\lambda\right|_{ [j/M^h,(j+1)/M^h )^2}.
		\]
		with $b_h = \nu_1^h$ for
		$\nu_1$ the unique real solution in $(0,1)$ to
		\[
		p_0\zeta^{k+1} - \zeta + 1-p_0= 0.
		\]
		For $k=1$ the solution $\nu_1$ is given by
		$
		\nu_1 = \frac{1-p_0}{p_0}.
		$
		For $k=2$ it is given by 
		$
		\nu_1 = -\frac{1}{2} + \frac{1}{2} \sqrt{ 1 + 4\frac{1-p_0}{p_0} }.
		$
		
		\item	
		The second special case we consider is of pairs $(M,N)$ with $M = N^{\ell}$.
		This corresponds to $\kappa = N$ and $k=1$. 
		Then the  iterated function system generated by the two-point maps $f^{(2)}_i$, $0\le i \le M$, admits an absolutely continuous  stationary measure $m^{(2)}$
		of the form
		\[
		m^{(2)} = \sum_{h=0}^\infty  b_h  \sum_{j=0}^{N^h-1}  N^h  \left.\lambda\right|_{ [j/N^h,(j+1)/N^h )^2}
		\]
		with
		\begin{align*}
			b_0 &= 1,
			\\
			b_j &=  (1-p_0)/p_0^j, &  1 \le j \le \ell,
			\\
			b_{j+1} &=  \frac{1}{p_0} b_j - \frac{1-p_0}{p_0} b_{j-\ell}, & j \ge \ell.
		\end{align*}	
	We have $b_h \sim \nu_1^h$, where $\nu_1$ is the unique solution in $(0,1)$ to
$p_0 \zeta^{\ell+1} - \zeta^\ell + 1-p_0=0$. For  $\ell=1$, this gives $\nu_1=\frac{1-p_0}{p_0}$. For $\ell=2$, $\nu_1= \frac{1-p_0 + \sqrt{(1-p_0)(1+3p_0)}}{2p_0}$.		
	\end{enumerate}
\end{example}	

\begin{remark}\label{r:boundeddensity}
	Writing 	$m^{(2)}$ obtained in Theorem~\ref{t:stationary-multdep} as \begin{align*}
		m^{(2)} &= \sum_{h=0}^\infty  b_h \kappa^h \sum_{j=0}^{\kappa^h-1}  \left.\lambda\right|_{ [j/\kappa^h,(j+1)/\kappa^h )^2},
	\end{align*}
and noting $b_h \sim \nu_1^h$, it is clear that its density is bounded  if $\nu_1  \kappa < 1$.
\end{remark}	

The two-point maps $f^{(2)}_i$ are examples of Jablonski maps \cite{MR744422}.	
In the literature, see \cite{MR1232960,Hsi08,MR2110096}, it is proved that random Jablonski maps admit an absolutely continuous
stationary measure under an expansion on average condition. In our setting this gives
that the iterated function system generated by $\{f^{(2)}_i\}$, $0\le i \le M$,
admits an absolutely continuous stationary measure if $\frac{p_0}{N} + (1-p_0)M < 1$.
Under this condition the stationary measure has bounded Tonelli variation.
We apply \cite{MR4342141} to get  an absolutely continuous stationary measure under the condition $L_{p_0} > 0$. This may not have bounded Tonelli variation, compare also
Remark~\ref{r:boundeddensity} and \cite[Section~4]{MR722776}.
 In contrast to Theorem~\ref{t:stationary-multdep}, here we do not have an explicit expression for the density function.

\begin{theorem}\label{t:stat2}
	Consider $L_{p_0} > 0$. The iterated function system on $[0,1)^2$ generated by $f^{(2)}_i$, $0\le i \le M$, admits an absolutely continuous stationary probability measure $m^{(2)}$. Furthermore, $m^{(2)}$ has full topological support.
\end{theorem}	

\begin{proof}
We will apply \cite{MR4342141} that considers skew product maps with an invertible base map. For this reason
we take the map $\hat{F}^{(2)}: \{0,\ldots,M\}^\mathbb{Z} \times [0,1)^2 \to \{0,\ldots,M\}^\mathbb{Z} \times [0,1)^2$ given by 
\[ \hat{F}^{(2)} (\omega,x,y) = (\sigma \omega , f^{(2)}_{\omega_0} (x,y)). \]
The base map of the skew product system, that is, $\sigma$ acting on $\{0,\ldots,M\}^\mathbb{Z}$, is invertible. The $\mathbf p$-Bernoulli measure on $\{0,\ldots,M\}^\mathbb{Z}$ with $\mathbf p$ as in \eqref{q:vectorp}, which we also denote by $\nu$ as in the one-sided case, is an ergodic invariant probability measure, so we fit the setting considered in \cite{MR4342141}.
Under the condition $L_{p_0}>0$, \cite[Theorem~4.2 and Remark~5]{MR4342141} provides a family $\mu^{(2)}_\omega$
of random absolutely continuous invariant measures on $[0,1)^2$. Invariant here means 
\[\left( f^{(2)}_{\omega_0}\right)_\ast \mu^{(2)}_\omega = \mu^{(2)}_{\sigma \omega}.\]  Furthermore, the measure $\hat{\mu}^{(2)}$ with marginal $\nu$ on $\{0,\ldots,M\}^\mathbb{Z}$ and fiber measures $\mu^{(2)}_\omega$ on $\{\omega\}\times [0,1)^2$,
is invariant under  $\hat{F}^{(2)}$.
 
Let $\hat \Pi : \{0,\ldots,M\}^\mathbb{Z} \times [0,1)^2 \to \Sigma \times [0,1)^2$
be the natural coordinate projection
\[\hat \Pi( (\omega_i)_{i\in\mathbb{Z}},x,y) = ( (\omega_i)_{i\in\mathbb{N}},x,y).\]  
Then $\mu^{(2)} = \hat \Pi_\ast \left(\hat{\mu}^{(2)}\right)$ is an invariant measure for $F^{(2)}$.
To show that  $\mu^{(2)}$ is an absolutely continuous measure, take a set  $A \subset  \Sigma \times [0,1)^2$  of zero measure for $\nu\times \lambda$. 
We wish to show that $\mu^{(2)} (A) = 0$.
Now $\hat \Pi^{-1} (A)$ has zero measure for $\nu\times \lambda$ on $\{0,\ldots,M\}^\mathbb{Z} \times [0,1)^2$. So $\hat \Pi^{-1} (A) \cap \left( \{ \omega\} \times [0,1)^2 \right)$ has zero Lebesgue measure for almost all $\omega \in \{0,\ldots,M\}^\mathbb{Z}$.  It thus has zero measure for $\mu^{(2)}_\omega$, for almost all $\omega\in \{0,\ldots,M\}^\mathbb{Z}$, by absolute continuity of $\mu^{(2)}_\omega$.	
Hence $A$ has zero measure for  $\mu^{(2)}$.

By \cite[Theorem~3.1 and Corollary~3.1]{MR935878}, $\mu^{(2)}$ is an invariant product measure,
so of the form $\mu^{(2)} = \nu \times m^{(2)}$. 	

By iterating under the expanding map $f_0^{(2)}$ 
we recognize that $m^{(2)}$ has full topological support.
Namely, take any open set $O \subset [0,1)^2$. 
Now $\left(f^{(2)}_0\right)^n$ maps rectangles \[R_{ij}^n = [i/N^n , (i+1)/N^n)\times [j/N^n,(j+1)/N^n)\]
onto $[0,1)^2$. Take a set of positive $m^{(2)}$ measure. As $m^{(2)}$ is absolutely continuous, we can take a Lebesgue density point of this set. For $n$ large and the rectangle  $R_{ij}^n$ containing this Lebesgue density point, $\left(f^{(2)}_0\right)^{-n} (O) \cap R_{ij}^n$
has positive $m^{(2)}$ measure.  The topological support of $m^{(2)}$ therefore intersects $O$.
\end{proof}

We also have the following related dynamical statement,
showing that orbits may stick close together for some iterates, but then diverge again.  Recall that $\Delta_\varepsilon$ denotes the $\varepsilon$-neighborhood $\{ (x,y) \in [0,1)^2 \; ; \; |x-y| < \varepsilon\}$ of the diagonal $\Delta$ in $[0,1)^2$. 

\begin{theorem}\label{t:divergence}
Consider $L_{p_0}>0$. 
Let $0<t<1$. There is a set of points $(\omega,x,y)$ in $\Sigma \times [0,1)^2$ of full $\nu\times\lambda$-measure, for which
\[
P(\varepsilon) = \lim_{n\to \infty}  \frac{1}{n} \left| \{ 0\le i < n \; ; \;  \left|f^i_\omega(x) - f^i_\omega(y)\right| < \varepsilon \} \right|  
\]
exists.
Moreover,
\[
\lim_{\varepsilon\to 0}  P(\varepsilon) = 0. 
\]
If $M$ and $N$ are multiplicatively dependent with $N = \kappa^k$, $M = \kappa^\ell$, then
\[
P(\varepsilon)
  \sim \varepsilon^{-\ln (\nu_1) / \ln(\kappa)},
\]
where $\nu_1$ is the unique solution in $(0,1)$ to $p_0 \zeta^{k+\ell}- \zeta +1-p_0 = 0$. 
\end{theorem}

\begin{proof}
The statement for general pairs $(M,N)$ follows from \cite[Proposition~4.1]{MR1707698}, see also  \cite[Theorem~4.5]{MR4342141}, combined with 
Theorem~\ref{t:stat2}.

Extra reasoning is needed to prove the statement for multiplicatively dependent $M$ and $N$.
Write \[S_h = \sum_{j=0}^{\kappa^h-1} [j / \kappa^h , (j+1)/\kappa^h)^2.\]
Take a point $(x_0,y_0) \in [0,1)^2$, which we consider to lie in $S_0 = [0,1)^2$. 
Iterate $(x_n,y_n) = \left( f^{(2)}_\omega \right)^n (x_0,y_0)$.
If we let $h_n$ with $h_0=0$ follow the random walk \eqref{e:rwx}, so
\begin{align*}
	h_{n+1} = \left\{  
	\begin{array}{ll}   \max\{ 0 , h_n - k \}, &   \omega_{n+1} = 0, 
		\\     
		h_n + \ell, & \omega_{n+1} \in \{1,\ldots,M\},
	\end{array}            
	\right.
\end{align*}
 then we find 
$(x_n,y_n) \in S_{h_n}$.
For the distance of $(x_n,y_n)$  to the diagonal $\Delta$ it is irrelevant in which
rectangle $[j / \kappa^{h_n} , (j+1)/\kappa^{h_n})^2)$ the point $(x_n,y_n)$ lies, but the position inside the rectangle is. If we rescale all rectangles to $[0,1)^2$, we find 
a sequence of points $(\tilde{x}_n,\tilde{y}_n) \in [0,1)^2$. The point  $(\tilde{x}_{n+1},\tilde{y}_{n+1})$ can only differ from
 $(\tilde{x}_n,\tilde{y}_n) $ if  $(\tilde{x}_n,\tilde{y}_n)$ and  $(\tilde{x}_{n+1},\tilde{y}_{n+1})$ both lie in  $S_{h_n}=S_{h_{n+1}} = S_0$; in this case
  $(\tilde{x}_{n+1},\tilde{y}_{n+1})  =  N  (\tilde{x}_n,\tilde{y}_n) \pmod 1$.
 Summarizing,
\begin{align*}
	(\tilde{x}_{n+1},\tilde{y}_{n+1}) &= \left\{   \begin{array}{ll}  
		N  (\tilde{x}_n,\tilde{y}_n) \pmod 1, & S_{h_n}=S_{h_{n+1}} = S_0,
		\\ (\tilde{x}_n,\tilde{y}_n), & S_{h_n} \ne S_{h_{n+1}}.
	  \end{array}
		   \right.
\end{align*}	 
As $h_n=h_{n+1}=0$ occurs for a positive proportion of iterates, for almost all $\omega$, and
$(x,y) \mapsto N (x,y) \pmod 1$ is ergodic with respect to Lebesgue measure,
for typical initial points $(x_0,y_0)$ and almost all $\omega$,  $(\tilde{x}_n,\tilde{y}_n)$ is uniformly distributed.
This implies
 $P(\varepsilon) = m^{(2)} (\Delta_\varepsilon)$ and the estimate
$m^{(2)} (\Delta_\varepsilon) \sim \varepsilon^{-\ln (\nu_1)/\ln(\kappa)}$.
\end{proof}

\begin{figure}[!ht]
	\begin{center}
		\includegraphics[height=3.5cm]{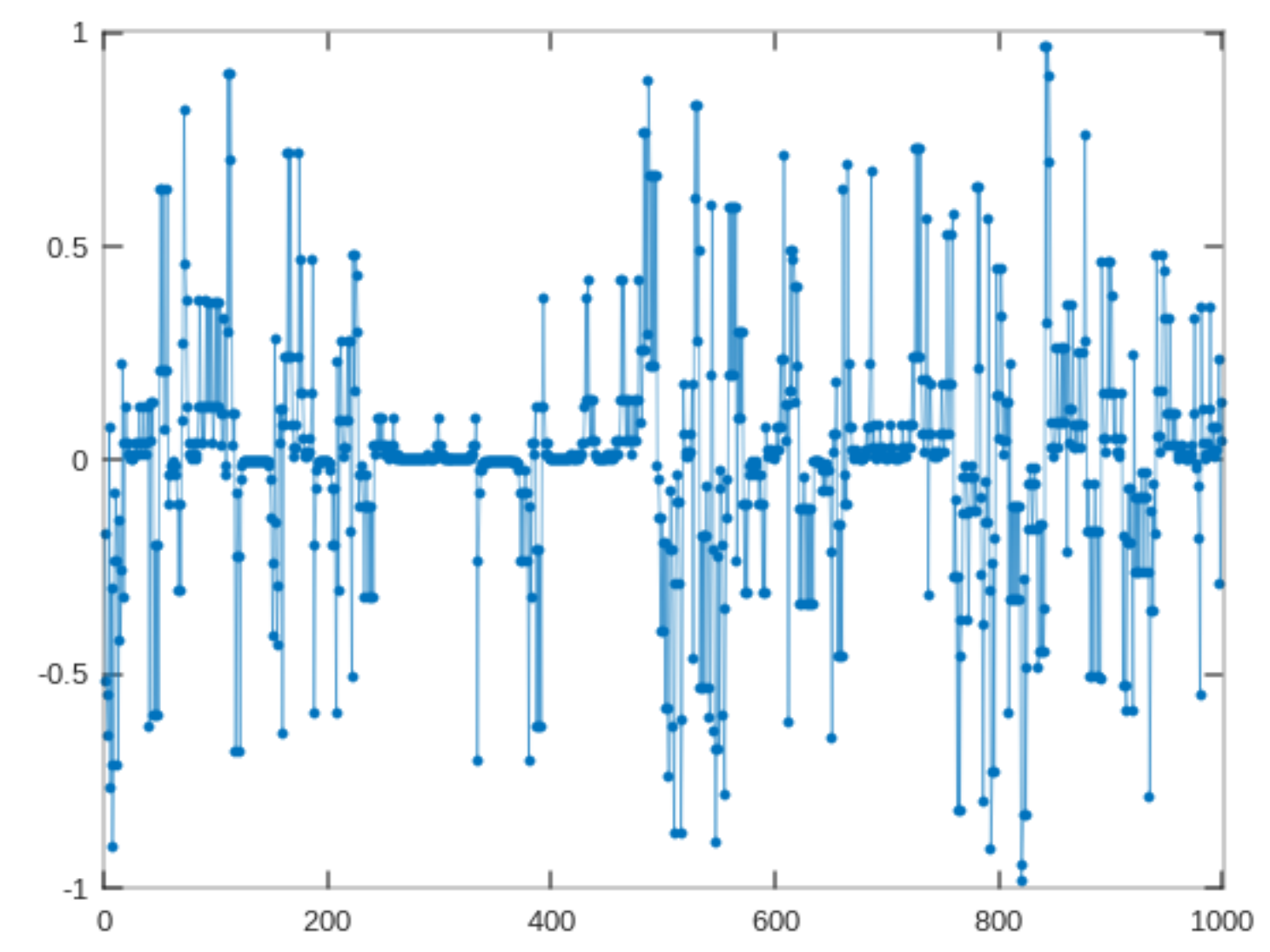}
		\hspace{1cm}
		\includegraphics[height=3.5cm]{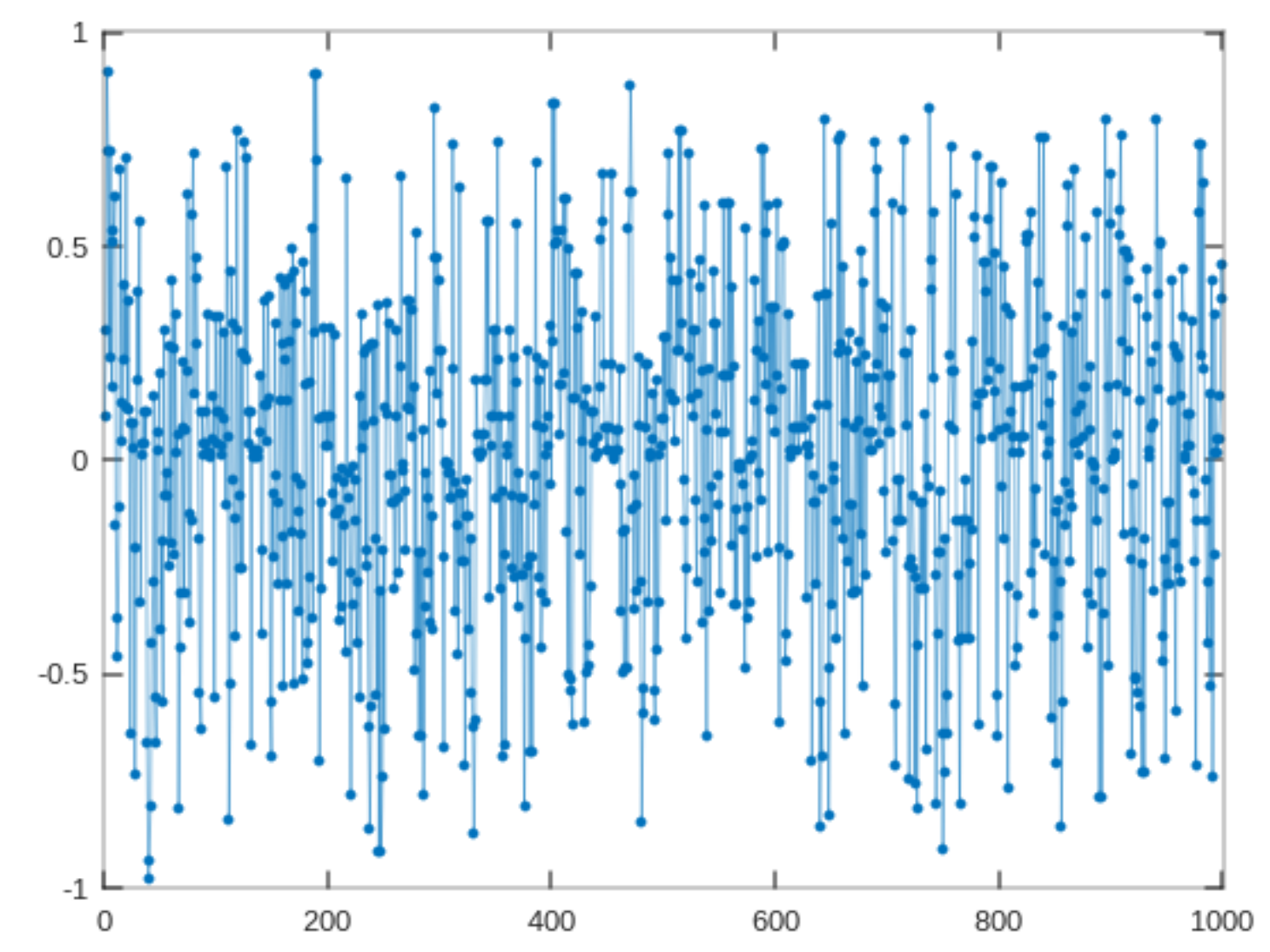}
		\caption{\label{f:stickiness} Shown are time series of signed differences of two orbits, for $(M,N) = (3,3)$.   The left panel is for $p_0=0.6$, the right panel is for $p_0=0.9$.  In both figures, $L_{p_0}>0$. The stationary measure $m^{(2)}$ of the two-point maps has bounded density
			function for $p_0 > 3/4$:  the density of $m^{(2)}$ is unbounded for the left panel and bounded for the right panel.
		}
	\end{center}	
\end{figure}
Figure~\ref{f:stickiness} shows time series of signed differences of two orbits, for $(M,N) = (3,3)$ and two different values of $p_0$.
The stationary measure has unbounded density for the left panel and bounded density for the right panel.  Theorem~\ref{t:divergence} explains and quantifies the relative stickiness of orbits visible in the left panel. 
Where we associate the dynamics for zero Lyapunov exponent to intermittency, 
one may argue that also the occurrence of a stationary density for the two-point iterated function system that blows up at the diagonal relates to intermittency. 

In addition to the above result, we have 
for $\nu\times\lambda$-almost all $(\omega,x,y)$,
 \begin{equation*} 
 	\begin{aligned}
	\liminf_{n\to\infty} \left|f^n_\omega(x) - f^n_\omega(y)\right| = 0,
	\\
	\limsup_{n\to\infty} \left|f^n_\omega(x) - f^n_\omega(y)\right| = 1,
	\end{aligned}
\end{equation*}
for multiplicatively dependent $M,N$.

\section{Final remarks}

We conclude with brief remarks on left out topics and on generalizations that would extend the core findings of this paper to a broader framework.

\subsection{Multiplicatively dependent pairs $(M,N)$}

We note that some of our results in Section~\ref{s:2point} admit simpler proofs in the specific and restricted case of multiplicatively dependent pairs $(M,N)$. Especially the proofs of Theorem~\ref{t:M>Nconvergence} and Theorem~\ref{t:Lp=0frequency} can be simplified by applying the following lemma~\ref{l:monotoneinstrip}, which stresses properties of graphs of $f^n_\omega$ in this case. We call intervals of the form $[i/\kappa^j,(i+1)/\kappa^j)$, for some integers $\kappa >1$, $j \ge 0$ and $0 \le i < \kappa^j$, $\kappa$-adic intervals.

\begin{lemma}\label{l:monotoneinstrip}
Consider $(M,N)$ with $M$ and $N$ multiplicatively dependent: $N = \kappa^{k}$ and $M = \kappa^{\ell}$. Then for each $n \ge 1$ and each $\omega \in \Sigma$ the graph of $f^n_\omega$ is contained in a strip of the form $[i/\kappa^j,(i+1)/\kappa^j)$, for some integers $j \ge 0$ and $0 \le i < \kappa^j$. Moreover, $f^n_\omega$ is a piecewise linear map with constant slope of which all branches have the interval $[i/\kappa^j,(i+1)/\kappa^j)$ as their image.
\end{lemma}

\begin{proof}
A direct computation shows that $f_i$, for any $0\le i \le M$, maps any $\kappa$-adic interval to a $\kappa$-adic interval. Note that, if the interval contains a discontinuity point of $f_0$ in its interior, then $f_0$ maps this interval to $[0,1)$.
\end{proof}

The pictures in Figures~\ref{f:1010010}~and~\ref{f:0010010}  illustrate the differences between multiplicatively dependent and multiplicatively independent pairs $(M,N)$.

\subsection{Phase transitions for general random systems}

The toy model studied here shows a phase transition for the two-point motion occurring when the Lyapunov exponent crosses zero. In Theorem~\ref{t:main2} we have seen that this involves an explosion of the support
of the stationary measure and an infinite stationary measure at the transition point.
 A central question is whether this is a typical scenario for more general classes of systems.

\vskip .2cm
In this article we looked at iterated function systems generated by expanding and contracting affine maps on the unit interval, determined by a pair of integers $(M,N)$,
and a probability vector $\mathbf p$ determined by a parameter $p_0$.
For fixed $M$ and $N$, a parameterized family of systems arises by varying the probability vector $\mathbf p$ more generally. It would be interesting to investigate how the dynamical properties of this family depend on $\mathbf p$. Where our setting has Lebesgue measure as stationary measure, a first question would be to determine a stationary measure for the iterated function system. For probability vectors other than $\mathbf p$ as in \eqref{q:vectorp}  one can start with constructing a stationary measure as in \cite{MR807105,MR1707698,MR722776,KALLE_2020} and continue from there.	

\vskip .2cm
There are various other generalizations thinkable.  One may consider  different classes of affine maps. One can generalize to nonlinear maps. The maps need not be uniformly expanding and contracting, and for instance a somewhat similar setup for
unimodal interval maps or for circle maps is then possible. One can also consider higher dimensional analogs.

\vskip .2cm
Finally, there is no need to restrict to iterated function systems, and one can study the same type of questions for 
skew product systems driven by more general noise.

\subsection{$d$-point motions} 

	Where we focused on $2$-point motions, 
	one may iterate more than two points by considering $d$-point maps
	$f^{(d)}_i : [0,1)^d \to [0,1)^d$ given by
	\[
	f^{(d)}_i (x_1,\ldots,x_d) = (f_i (x_1),\ldots,f_i(x_d)).
	\]
	This makes little difference in the cases of nonpositive Lyapunov exponent $L_{p_0}$, as we already analyzed the
	fate of the entire interval $[0,1)$ under iterations. 
	But take for instance the setting of Theorem~\ref{t:stationary-multdep} with $L_{p_0}>0$ and multiplicatively dependent $(M,N)$.
	The reasoning to prove Theorem~\ref{t:stationary-multdep}
	provides, under the given assumptions, an absolutely continuous  stationary measure $m^{(d)}$
	for the  iterated function system generated by the $d$-point maps $f^{(d)}_i$, $0\le i \le M$, 
	of the form
	\begin{align*}
	m^{(d)} &= \sum_{h=0}^\infty  b_h  \sum_{j=0}^{\kappa^h-1}  \kappa^{h(d-1)}  \left.\lambda\right|_{ [j/\kappa^h,(j+1)/\kappa^h )^d}.
    \end{align*}
	Its density is bounded if  $\nu_1 \kappa^{d-1} < 1$
	(compare Remark~\ref{r:boundeddensity}).
	The transition values of $p_0$ where the density  of 	$m^{(d)}$  changes from bounded to unbounded thus depend on $d$. 
	This agrees with the observation that it becomes increasingly less likely 
	to find iterates of higher numbers of points close to each other.

 \appendix

\section{Random walks with small drift}\label{s:A}

Reductions to random walks on the half line or the line are a recurring tool for the study of the iterated function systems in this paper. This appendix develops results that are used in the main text in the study of intermittency.

\subsection{Stopping times for random walks with a small negative drift}\label{s:a1}
Write $\Sigma_2 = \{0,1\}^\mathbb{N}$ endowed with the product topology and the Borel $\sigma$-algebra. Fix a $0 < p_0 < 1$ and write $\nu_2$ for the $(p_0,1-p_0)$-Bernoulli measure on $\Sigma_2$ determined on any cylinder set
\[ [a_0 \cdots a_k] = \{ \omega \in \Sigma_2 \, ; \, \omega_i = a_i, \, 0 \le i \le k \}, \, a_i \in \{0,1\},\]
by
\[
\nu_2 (  [a_0\cdots a_k]) = p_0^{\# \{ 0\le i \le k \; ; \; a_i=0 \}  }   (1-p_0)^{\# \{ 0\le i \le k \; ; \; a_i=1 \}}.
\]
For real numbers $L<0$ and $R>0$ consider the random walk given by 
\begin{align}
	\label{e:zn}
	z_{n+1} &= \left\{ \begin{array}{ll} z_n + L, & \omega_n=0, \\ z_n + R, & \omega_n=1, \end{array}\right.
\end{align}
so the step $L < 0$ is taken with probability $p_0$, the step $R>0$ with probability $1-p_0$. We assume that the average drift $\alpha$ given by
\[
\alpha = p_0 L + (1-p_0) R
\]
is small and negative, so $\alpha < 0$. 

\vskip 0.2cm
Consider escape from $[0,\infty)$ for the random walk \eqref{e:zn}.	
For $z_0 \in [0, \infty)$, define the stopping time
\begin{align*}
	T &= \min \{ n\in\mathbb{N}\; ; \; z_n < 0\}.
\end{align*}

\begin{lemma}\label{l:reciprocalpha}
Assume $z_0 \in [0,R)$. Then
	$	\int_{\Sigma_2} T (\omega) \, d\nu_2 \ge -p_0L/ | \alpha |$.
\end{lemma}	

\begin{proof}
	By Wald's identity, see for instance \cite[Section~VII.2]{MR1368405},
	\begin{align*}
		\int_{\Sigma_2} T (\omega) \, d\nu_2  &=  \frac{1}{ \alpha} \int_{\Sigma_2} z_{T (\omega)} - z_0 \, d\nu_2.
	\end{align*}
	Note that $T(\omega)$ as a function of $z_0$ is minimal for $z_0=0$. Then
\[ \begin{split}
		\frac{1}{ \alpha} \int_{\Sigma_2} z_{T (\omega)} - z_0 \, d\nu_2  =\ & \frac1{|\alpha |}  \int_{\Sigma_2} z_0 - z_{T (\omega)} \, d\nu_2 \\
		\ge \ & \frac1{|\alpha|} \int_{[0]} - z_{T(\omega)} \, d\nu_2 = -\frac{p_0 L}{|\alpha|}.\qedhere
	\end{split} \]
\end{proof}

Next consider escape from an interval $[0,K]$ with $K$ a large positive number.
So let $z_0 \in [0,K]$ and define
\begin{align*}
	T_K &= \min\{ n > 0 \; ; \;  z_n < 0 \text { or } z_n > K \}.
\end{align*}
It is well known that the average stopping time to reach $(-\infty,0)\cup (K,\infty)$ is finite: 	\[\int_{\Sigma_2}  T_K(\omega) \, d \nu_2(\omega) < \infty.\]
To see this it suffices to realize that escape from $[0,K]$ is guaranteed after a sufficient number of identical symbols, either $0$ or $1$, in $\omega$. 
We will discuss the probability of escape through $0$, when starting close to $K$, and establish that 
the probability of escaping through $0$ can be made arbitrarily small by taking the drift $\alpha$ close enough to 0 and $K$ large enough.
For the next lemma we take a context of 
parameterized families of random walks:  consider \eqref{e:zn} with 
\begin{equation}\label{e:L0R0}  
L = L_0 + \alpha, \qquad  R = R_0 + \alpha,
\end{equation}
where $p_0 L_0 + (1-p_0) R_0 = 0$. The average drift $\alpha <0$ is taken as the parameter.

	\begin{lemma}\label{l:AB}
		Consider \eqref{e:zn} with $L,R$ given by \eqref{e:L0R0}. 
		Given $\rho>0$, there is a small $\alpha_0<0$ and a large $K>0$, so that for $\alpha_0 < \alpha < 0$ and
		$z_0 \in (K+L,K] \subset [0,K]$, we have
		\[
		\nu_2 (\{ \omega \in \Sigma_2 \;  ;  \; z_{T_K} < 0 \}) \le \rho.
		\]	
	\end{lemma}
	
\begin{proof}
Write $\zeta_0=0$ and $\zeta_n = z_n - z_{n-1}$ for $n \ge 1$. For $n \ge 0$ write $S_n = \zeta_0 + \cdots + \zeta_n = z_n - z_0$. Note that for $n \ge 1$, $S_n$ depends on $\omega_0,\ldots,\omega_{n-1}$. Consider the function
	\[
	G_n = e^{r^* S_n},
	\]
	where $r^* > 0$ is the solution of
	\[
	p_0  e^{L r^*}  +  (1-p_0) e^{R r^*} = 1.
	\]
	By developing the exponential functions in a Taylor series, one can check that this equation has a unique solution $r^* > 0$ with $r^* \to 0$ as $\alpha \to 0$:
	\[p_0 \left(1 + Lr^* + \frac{1}{2} L^2 (r^*)^2\right) + (1-p_0) \left( 1 + Rr^* + \frac{1}{2} R^2 (r^*)^2 \right)	= 1 + \mathcal{O}\left((r^*)^3\right)\] 
	yields  
	\[
	\alpha r^* + \frac{1}{2} \left(p_0 L^2 + (1-p_0) R^2\right) (r^*)^2 = \mathcal{O}\left((r^*)^3\right)
	\]
	and thus
	\[\alpha + \frac{1}{2} \left(p_0 L^2 + (1-p_0) R^2\right) r^* = \mathcal{O}\left((r^*)^2\right)\]
	for $r^*\ne 0$, 
	from which $r^*$ can be solved by the implicit function theorem.	
	Now $G_n$ is a martingale as for any cylinder $[a_0 \ldots a_{n-1}] \subseteq \Sigma_2$,
	\begin{align*}
		\int_{[a_0 \ldots a_{n-1}]}   e^{r^* S_{n+1}}  \, d\nu_2 &=
		\int_{[a_0 \ldots a_{n-1}]}  e^{r^* S_{n}} e^{r^* \zeta_{n+1}}  \, d\nu_2 
		\\
		&= e^{r^* S_{n}}  \int_{[a_0\ldots a_{n-1}]}  e^{r^* \zeta_{n+1}}  \, d\nu_2
		\\
		&= e^{r^* S_{n}} \left(  \int_{[a_0\ldots a_{n-1} 0]}  e^{ L r^*}  \, d\nu_2 +
		\int_{[a_0\ldots a_{n-1} 1]}  e^{R r^*}  \, d\nu_2  \right)
		\\
		&=  e^{r^* S_{n}} \left( \int_{[a_0\ldots a_{n-1}]}  p_0  e^{L r^*}  +  (1-p_0) e^{R r^*}\, d\nu_2  \right)
		\\
		&=  \int_{[a_0\ldots a_{n-1}]}  e^{r^* S_{n}}  \, d\nu_2.
	\end{align*}
	By Doob's optional stopping theorem,
	see for instance \cite[Theorem~VII.2.2]{MR1368405}, we see that for any large $K$, 
	\begin{align*}
		\int_{\Sigma_2} e^{r^* S_{T_K}} \, d\nu_2  = e^{r^* S_0}
		=  1.
	\end{align*}
	This gives
	\begin{align*}
		\int_{\Sigma_2} e^{r^* z_{T_K}} \, d\nu_2   =  e^{z_0 r^*}.
	\end{align*}
	Observe that $z_{T_K} \in [L,0)$ or $z_{T_K} \in (K,K+R]$. For any large $K$ let
	\[A_K =
	\nu_2 (\{ \omega \in \Sigma_2 \;  ;  \; z_{T_K}(\omega) < 0 \})>0
	\]
	be the probability that $z_{T_K} < 0$. Let $0 < \rho < 1$. Our goal is to show that we can find $\alpha_0$ and $K$ such that $A_K \le \rho$ for all $\alpha \in (\alpha_0,0)$. For any $\alpha <0$ write
	\begin{align*}
		\int_{\Sigma_2} e^{r^* z_{T_K}} \, d\nu_2  &=
		A_K   e^{c_1 r^*} +  (1-A_K)  e^{K r^*} e^{c_2 r^*},
	\end{align*}
	with
	\begin{align*}
		e^{c_1 r^*} &= \frac{1}{A_K} \int_{ \{ \omega \in \Sigma_2 \; ; \; z_{T_K} < 0\}} e^{r^* z_{T_K}} \, d\nu_2,
		\\
		e^{c_2 r^*} &= \frac{1}{1-A_K} \int_{ \{ \omega \in \Sigma_2 \; ; \; z_{T_K} > K\}}e^{r^* (z_{T_K} - K)} \, d\nu_2.
	\end{align*}
	Note that $c_1 \in [L,0]$ and $c_2 \in [0,R]$; $c_1$ represents the average value that $z_{T_K}$ takes if the random walk escapes through 0 and $c_2+K$ is the average value that $z_{T_K}$ takes if the random walk escapes through $K$.
	We obtain
	\begin{align*}
		A_K   &= 
		\frac{ e^{c_2 r^*} - e^{z_0 r^*}e^{-K r^*}}{ e^{c_2 r^*} - e^{c_1 r^*} e^{-K r^*}}.
	\end{align*}

Since $r^* \to 0$ as $\alpha \to 0$, the terms
$e^{c_1 r^*}$ , $e^{c_2 r^*}$ and  $e^{z_0 r^*}e^{-K r^*}$
converge to $1$  as $\alpha \to 0$.
So we can take $\alpha_0$ small so that for the corresponding $r_0^*$,
	\[
	e^{c_2 r_0^*} - e^{z_0 r_0^*}e^{-K r_0^*} < \rho/2.
	\] 
Now taking $K$ large enough ensures that
$A_K < \rho$ for these values of $\alpha_0$ and $K$.	

\vskip 0.2cm	 
To show that $A_K < \rho$ for any $\alpha$ with $\alpha_0 < \alpha < 0$, fix such an $\alpha$. We must now consider that $z_n$ depends on $\alpha$ and write
$z_{\alpha,n}$.  Clearly, for a fixed value of $z_0$,  \[z_{\alpha_0,n} < z_{\alpha,n} .\]
This implies that  $\nu_2 (\{ \omega \in \Sigma_2 \;  ;  \; z_{T_K} < 0 \})$ of escape through $0$ decreases with $\alpha$. 
The lemma follows.
\end{proof}

\subsection{Stopping times for random walks with time dependent levels}

We stay with the random walk $z_n$ from \eqref{e:zn}, but now with the no drift condition
\[ 
 p_0 L + (1-p_0)R = 0. 
\] 
Let $\varepsilon_n = 1 / n^{2}$ and note that $\sum_{n=1}^\infty \varepsilon_n < \infty$. Let also $\varepsilon >0$ and $p \ge 1$ be such that $\varepsilon_{p+1} < \varepsilon < \varepsilon_p$. We assume that $\varepsilon$ is small or equivalently that $p$ is large. Define $K_n = -\ln(\varepsilon_{n+p}) + \ln(\varepsilon)$, so 
\begin{align}\label{e:Kn=}
K_n &= 2 \ln (n+p) + \ln(\varepsilon).
\end{align}
Note that $K_0  = 2 \ln(p) + \ln(\varepsilon) < 0$ and that $\lim_{n \to \infty} K_n = \infty$. Suppose $z_0 > K_0$.
We want to know the average stopping time to reach the $n$-dependent level
$K_n$.  Time dependent stopping levels like these have been considered in \cite{MR394857} and \cite[Section~4.5]{MR2489436}. We use the statements that are derived here in the study of intermittency in Section~\ref{s:2point=0}.
Write 
\[
S = \min \{ n > 0 \; ; \; z_n < K_n \}.
\]
\begin{figure}[ht]
	\begin{center}
		\resizebox{10cm}{!}{
			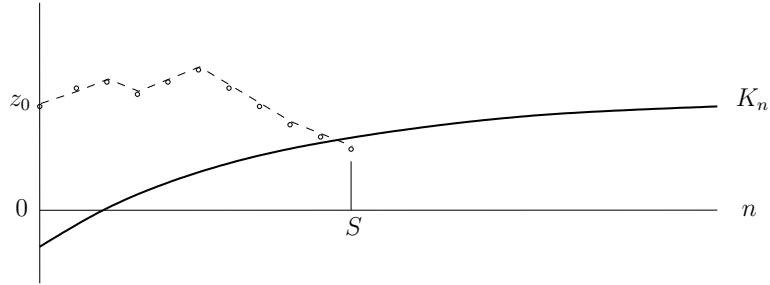  
			}
	\end{center}
	\caption{
		Time dependent stopping levels where stopping levels $K_n$ lie on a non constant curve. 
	 }\label{f:stoppinglevels}
\end{figure}
It is standard that in the case of no drift the expected stopping time to reach a point smaller than the fixed level $K_0$ is infinite.
We will show that the expected value of $S$ is still infinite, using the slow growth of $K_n$. Figure~\ref{f:stoppinglevels} illustrates the setting.

\begin{lemma}\label{l:stoppingtimeinfinite}
	\[\int_{\Sigma_2}  S(\omega) \, d\nu_2 (\omega) = \infty. \]
\end{lemma}	

\begin{proof}
The strategy of the proof is to give a lower bound for the integral from the lemma by considering, instead of the stopping level $K_n$, a sequence of affine stopping levels $(M_n^{(m)})_{m \ge 0}$ that have slope decreasing to 0 as $m \to \infty$. We then translate the situation of having a random walk with no drift and affine stopping levels to a random walk with small negative drift and a constant stopping level $K$, so that we can apply the results from the previous section. The stopping levels $M_n^{(m)}$ are obtained by taking tangent lines to $K_n$ at suitable moments $m$. We will first explain this last part for an arbitrary suitable time $m$.

\vskip .2cm
We will work with a large positive number $K$; a condition for $K$ will be given in the course of the proof. As we have $z_m >  K + K_m$ with positive probability, for a suitable positive integer $m$, to prove the lemma we may assume
\[z_0 > K + K_0.\]
Counting iterates from $m$ on by writing $n = m+i$, $i\ge 0$, and translating the values $K_{n}$ by $K_0-K_m$ replaces $K_n$ by $K_{i+m}  + K_0 - K_m = 2 \ln(i+m+p) + 2 \ln(p) - 2\ln(p+m) + \ln(\varepsilon)$, $i \ge 0$. This is of the form \eqref{e:Kn=} and shows that we may also assume that $p$ is large in \eqref{e:Kn=}; a condition for $p$ will also be given in the course of the proof. 
	
\vskip .2cm
We construct a stochastic process $u_n$ built from random walks $v^{(m_i)}_n$, $m_i \le n \le m_{i+1}$ for certain stopping times $m_i$.  We start with two ingredients, the introduction of stopping times $S_m$ and $U_{m,l}$.  \\

\noindent {\sc Definition of a stopping time $S_m$.}	
Start with an integer $m \ge 0$ such that $z_m > K + K_m$. For all $n>m$ we replace the level $K_n$ by a level $M_n = M^{(m)}_n$ depending affinely on $n$: 
	\[
	M_n = \alpha_{m} + \beta_{m} n
	\] 
	with
	$\alpha_{m},\beta_{m}$ so that the line $x \mapsto \alpha_{m} + \beta_{m} x$ 
	is tangent to $x \mapsto  2 \ln(x+p) + \ln(\varepsilon)$ at $x=m$.
	This gives
	\[
	\alpha_{m} = K_{m}  -  2 m /(m+p) \quad \textrm{and} \quad \beta_{m} = 2/(m+p).
	\]
	As the graph of $ x \mapsto 2\ln (x+p)$ is concave, we have $K_n \le M_n$. Consider the stopping time
	\[
	U :=  \min \{ n > m  \; ; \; z_n < M_n  + K\} \le \min \{ n > m \; ; \; z_n < K_n + K\}.
	\]
Define 
	\[
	v^{(m)}_n = z_n - M_n = z_n - \beta_{m} (n-m) - K_{m}.
	\] 
So $v^{(m)}_{m} = z_{m} - K_{m}$ and thus $v^{(m)}_{m}  > K$. The sequence $v^{(m)}_n$ defines a random walk given by  
\[ v^{(m)}_{n+1} =   z_{n+1} - \beta_{m} (n+1-m) - K_m = \begin{cases}
v^{(m)}_n +L  -\beta_m, & \omega_n=0,\\
v^{(m)}_n + R - \beta_m, & \omega_n=1.
\end{cases}\]	
Hence, the random walk $v^{(m)}_n$ has a negative drift $-\beta_m = -\frac{2}{m+p}$, which is small if $p$ is large and depends on and is decreasing in $m$. The demand $z_n < M_n +K$ is equivalent to $v^{(m)}_n < K$. Write 
	\[
	S_{m}  (\omega) = \min \{ n > m  \; ; \; v^{(m)}_n < K\}
	\]
and denote $\mathbb{E}_m (S_m) = \int_{\Sigma_2}  S_m \, d\nu_2$. Note that the smallest value of $\mathbb{E}_m (S_m)$, for varying $v_m^{(m)} \ge K$, is obtained for $v^{(m)}_m = K$. As the expected stopping time is similar to the reciprocal of the average drift, see Lemma~\ref{l:reciprocalpha}, 
\begin{align}\label{e:reciprocalpha}
\mathbb{E}_m (S_m ) &\ge  \frac{-p_0 (L-\beta_m)}{\beta_m} \ge \frac{-p_0L}{2} (m+p). 
\end{align}
	
\noindent {\sc Definition of a stopping time $U_{m,l}$.}
	The second ingredient is the random walk $v^{(m)}_n$ with values inside the interval $[0,K]$.
	Assume $v^{(m)}_l \in (K-L,K]$ for a positive integer $l >m$ and 
	consider a second stopping time
	\[
	U_{m,l}  (\omega) = \min \left\{ n > l \; ; \; v^{(m)}_n < 0 \quad \textrm{or} \quad v^{(m)}_n> K \right\}.
	\]
	Note that $v^{(m)}_n < 0$ is equivalent to $z_n < M_n$ and 
	$v^{(m)}_n> K$ is equivalent to $z_n > M_n+K$.
	Write 
	\[
	\rho_m = \nu_2 \left( \left\{ \omega \in \Sigma_2 \; ; \;  v^{(m)}_{U_{m,l}} < 0 \right\}  \right)
	\]
	for the probability that $v^{(m)}_n$ crosses $0$.  	
	As $v^{(m)}_l \in (K+L,K]$, by Lemma~\ref{l:AB} 
	we find that $\rho_m$ will be small
	for all sufficiently large $m$ if $K$ is large. More precisely, given any $\rho >0$ we can choose $p$ sufficiently small (so that the drift $-\beta_m$ is close enough to zero for all $m$) and $K$ sufficiently large such that $\rho_m < \rho$ for all $m$ or equivalently
	\begin{align}\label{e:roleofrho}
	\nu_2\left( \left\{ \omega \in \Sigma_2 \; ; \;  v^{(m)}_{U_{m,l}} >K \right\}  \right) &\ge 1-\rho
	\end{align}
	for all $m$.	 \\
	
	\noindent {\sc Construction of the process $u_n$.}
	\begin{figure}[ht]
		\begin{center}
			\resizebox{12cm}{!}{
				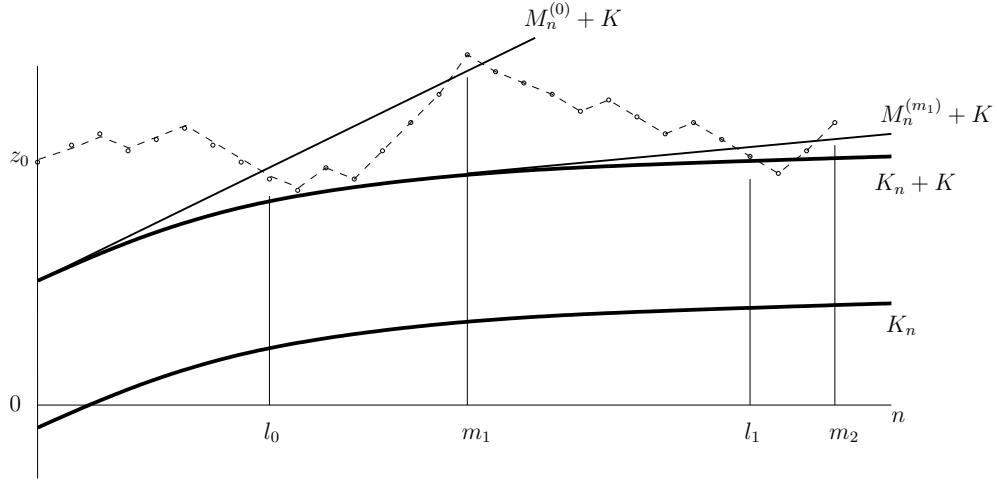  }
		\end{center}
		\caption{
			A possible path $z_n$ with indicated stopping times.
		}\label{f:stoppinglevels+}
	\end{figure}

Let $m_0 = 0$, assume $u_0 >K$,  and 
	define the process $u_n$, $n \ge 0$, as follows.
	Write $l_0 = S_{m_0}$. If $l_0 < \infty$, let $m_1 = U_{m_0,l_0}$.
	For $m_0 \le n \le m_1$ we let $u_n = v^{(m_0)}_n$.
	Inductively, 
	suppose $l_i$ is defined for $0 \le i < k$ and $m_i$ is defined for $0\le i \le k$.
	We have $u_n$ defined for $0\le i \le m_k$.
	Then set
	\[
	l_k = S_{m_k}. 
	\]
	If $l_k < \infty$, we let $u_n =  v^{(m_k)}_n$ for $m_k \le n \le l_{k}$ and let
	\[
	m_{k+1} = U_{m_k,l_k}.
	\]
	For $m_{k+1} < \infty$, we have either  $v^{(m_{k})}_{m_{k+1}} <0$  or   $v^{(m_{k})}_{m_{k+1}} > K$.
	If  $v^{(m_{k})}_{m_{k+1}} > K$, then we let 
	\[
	u_n = v^{(m_k)}_n, \qquad l_k \le n \le m_{k+1}.
	\]
	Note that $u_{m_{k+1}} - K_{m_{k+1}} > K$.
	
	If some $v^{(m_{i})}_{m_{i+1}} < 0$, we let $m_j = m_{i+1}$ for $j > i$. A path of the corresponding walk $z_n$ with indicated stopping times is depicted in Figure~\ref{f:stoppinglevels+}. A similar visualization 	of paths $u_n$ is presented in Figure~\ref{f:stoppingtimes}.\\
	  
\noindent {\sc Estimating the average of the stopping time $S$.}	  
	\begin{figure}[ht]
	\begin{center}
		\resizebox{12cm}{!}{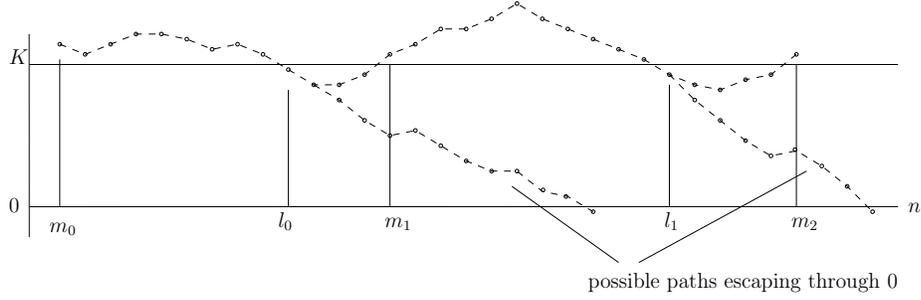}
	\end{center}
	\caption{
		Visualization of breaking up random walks $u_n$ in parts where $u_n > K$ or $u_n \in [0,K]$. 
	}\label{f:stoppingtimes}
\end{figure}
Having constructed the path $u_n$ with the sequence of stopping times $m_k$, we can estimate the expected value of the stopping time $S$. We first set the parameters. Set $c = -\frac{p_0L}{2}$ and let $\rho >0$ be small enough such that $(1-\rho)(1 +c) >1$.  Let $\alpha_0 < 0$ and $K>0$ be as given by Lemma~\ref{l:AB}. Choose $p$ so that $-\frac2p \in (\alpha_0,0)$. Let $\frac1{(p+1)^2} < \varepsilon < \frac1{p^2}$ and let $K_n$, $M_n^{(m)}$ be as defined before. Note that the choice of $p$ implies that $-\beta_m \in (\alpha_0,0)$ for all $m \ge 0$. Let 
\[ E_0 = \{ \omega \in \Sigma_2 \, ; \, l_0(\omega) < \infty \, \text{ and }\, v_{m_1}^{(m_0)}(\omega) >K\}\]
and for $n \ge 1$ let
\[ E_n = \{ \omega \in E_{n-1} \, ; \, l_n(\omega) < \infty \, \text{ and }\, v_{m_{n+1}}^{(m_n)}(\omega) >K\}.\]
Note that $S(\omega) \ge l_0(\omega)$ for all $\omega \in \Sigma_2$ and that for $\omega \in E_n$ we have
\[ S(\omega) \ge l_0 (\omega) + l_1(\omega)-m_1(\omega) + \cdots + l_{n+1}(\omega)-m_{n+1}(\omega), \, n \ge 0.\]
Hence,
\begin{equation}\label{q:sumS}
\int_{\Sigma_2} S \, d\nu_2 \ge \int_{\Sigma_2} l_0 \, d\nu_2 +\sum_{n \ge 0} \int_{E_n} l_{n+1} -m_{n+1}\, d\nu_2.
\end{equation}
As established in \eqref{e:reciprocalpha}, 
\[
 \mathbb{E} (l_0)  \ge  -\frac{p_0 L}{2}p = cp.
\]
The set $E_0$ is a union of cylinders on which the time $m_1$ is constant. Let $\eta = \eta_0 \cdots \eta_k \in \{0,1\}^k$ be such that the cylinder $C= [\eta_0 \cdots \eta_k]$ is in $E_0$ with $k=m_1(\omega)=: m_1(\eta)$ for each $\omega \in C$ and write $l_0(\eta) $ for the value $l_0(\omega)$, $\omega \in C$. Then by Lemma~\ref{l:reciprocalpha},
\[ \int_{C} l_1-m_1 \, d\nu_2 \ge -\frac{p_0L}{2} (m_1(\eta)+p) \nu_2(C) \ge c (l_0(\eta)+p) \nu_2(C).\]
From \eqref{e:roleofrho} we see that
\[ \int_{E_0} l_1-m_1 \, d\nu_2 \ge \nu_2(E_0) c(cp+p) \ge (1-\rho)cp (1+c).\]
Similarly, let $\eta = \eta_0 \cdots \eta_k \in \{0,1\}^k$ be such that the cylinder $C= [\eta_0 \cdots \eta_k] \subseteq E_1$ with $k=m_2(\omega)=: m_2(\eta)$ for each $\omega \in C$ and write $l_0(\eta), m_1(\eta),l_1(\eta)$ for the values $l_0(\omega), m_1(\omega), l_1(\omega)$, $\omega \in C$, respectively. From Lemma~\ref{l:reciprocalpha} we get
\[ \begin{split}
\int_{C} l_2-m_2 \, d\nu_2 \ge \ & -\frac{p_0L}{2} (m_2(\eta)+p) \nu_2(C)\\
 \ge \ & c (l_1(\eta)+p) \nu_2(C) \ge c(l_0(\eta) + l_1(\eta) - m_1(\eta)+p)\nu_2 (C).
 \end{split}\]
Then \eqref{e:roleofrho} gives
\[ \int_{E_1} l_2-m_2 \, d\nu_2 \ge \nu_2(E_1) c (cp + c(cp+p)+p) \ge (1-\rho)^2 cp (1+c)^2.\]
Continuing, we find for each $n \ge 1$ and $\eta = \eta_0 \cdots \eta_k \in \{0,1\}^k$ for which the cylinder $C = [\eta_0\cdots\eta_k] \subseteq E_{n-1}$ satisfies $k = m_n(\omega)=: m_n(\eta) $ for each $\omega \in C$ that
\[ \int_{E_{n-1}} l_n-m_n \, d\nu_2 \ge (1-\rho)^n cp (1+c)^n.\]
Together with \eqref{q:sumS} and the assumption that $(1-\rho)(1+c)>1$ this yields
\[ \int_{\Sigma_2}  S \, d\nu_2 \ge  \sum_{i=0}^\infty cp (1-\rho)^{i} (1+c)^{i}  = \infty.  \qedhere \] 	
\end{proof}

\bibliographystyle{plain} 
\bibliography{exp-contr-biblio}

\end{document}